\documentclass[11pt,reqno]{amsart}
\usepackage[
  a4paper,
  asymmetric,
  left=20.4mm,
  right=19.6mm,
  top=33.4mm,
  bottom=23.6mm,
  headheight=12pt,
  headsep=25pt,
  footskip=30pt
]{geometry}

\usepackage{amsmath,amssymb,mathtools,bm}
\usepackage{xcolor}
\usepackage{microtype}
\definecolor{equationrefred}{rgb}{1,0,0}
\definecolor{citationrefgreen}{rgb}{0,0.78,0}
\usepackage{hyperref}
\usepackage[nameinlink,capitalize]{cleveref}

\hypersetup{
  colorlinks=true,
  linkcolor=equationrefred,
  citecolor=citationrefgreen,
  urlcolor=black,
  filecolor=black,
  pdftitle={Limiting Pointwise Decay for the Viscous p-System},
  pdfsubject={Space--time pointwise remainder analysis for the one-dimensional viscous p-system},
  pdfkeywords={viscous p-system, diffusion wave, pointwise estimate, Cole--Hopf transformation, Green function}
}

\numberwithin{equation}{section}

\newtheorem{theorem}{Theorem}[section]
\newtheorem{proposition}[theorem]{Proposition}
\newtheorem{lemma}[theorem]{Lemma}

\theoremstyle{definition}

\theoremstyle{remark}
\newtheorem{remark}[theorem]{Remark}

\allowdisplaybreaks
\author{Rui Li}
\address{Department of Mathematics, The Chinese University of Hong Kong,
Shatin, Hong Kong}
\email{ruili001@cuhk.edu.hk}

\author{Guanghui Wang}
\address{School of Mathematics, Guangxi University, Nanning, China}
\email{guanghui06@foxmail.com}

\author{Xu Song}
\address{School of Mathematical Sciences, South China Normal University,
Guangzhou, China}
\email{songxu@m.scnu.edu.cn}

\author{Lingda Xu}
\address{The Chinese University of Hong Kong, Shatin, Hong Kong}
\email{lingdaxu@cuhk.edu.hk}
\begin{document}

\title[A pointwise endpoint for the compressible isentropic Navier-Stokes equations]
{Limiting Pointwise Decay for the compressible isentropic Navier-Stokes equations.}

\begin{abstract}
We study the long-time pointwise behavior of small localized perturbations of a constant state for the one-dimensional compressible isentropic Navier-Stokes equations. After subtracting the two Burgers diffusion waves, the convergent sum of all higher-order diffusion waves, and the cross-family viscous corrections, we prove a cone-preserving pointwise estimate for the exact physical remainder and, in particular,
\[
 |R_i(x,t)|\leq C E_N\log(2+t)\Psi_i(x,t),
 \qquad
 \sup_{x\in\mathbb R}\Psi_i(x,t)\leq C(1+t)^{-1}.
\]
Here \(E_N\) measures the size of the initial data and
\(\Psi_i\) is the cone-resolved weight; both are defined precisely in
the main theorem below.
Thus $\|R_i(t)\|_{L^\infty}\leq C E_N(1+t)^{-1}\log(2+t)$.
The key new idea is to apply a familywise Cole--Hopf transformation to the
spatial antiderivative of the remainder, which exactly eliminates the critical
same-family first-order feedback.  We further construct an approximate Green
function adapted to the two characteristic families and combine it with
Gaussian-mode extraction and a Kawashima-type energy argument.  This yields a
cone-preserving estimate at the limiting decay rate, up to a logarithmic loss.
\end{abstract}


\maketitle
\section{Introduction}

We study the long-time behavior of small, spatially localized perturbations of a
constant state for the one-dimensional viscous \(p\)-system
\begin{equation}
 v_t-u_x=0,
 \qquad
 u_t+p(v)_x=\nu\left(\frac{u_x}{v}\right)_x,
 \qquad x\in\mathbb R,\quad t>0.
 \label{eq:p-system}
\end{equation}
Here \(v\) is the specific volume, \(u\) is the velocity, and \(x\) is the
Lagrangian mass coordinate. We work near \((v,u)=(1,0)\) under
\[
 p'(1)<0,\qquad p''(1)\ne0,\qquad \nu>0;
\]
the precise local assumptions on \(p\) are stated in
Theorem~\ref{thm:main-endpoint}. The system is a basic model of barotropic
viscous gas dynamics and, at the same time, a prototype of a partially
dissipative hyperbolic--parabolic conservation law. Only the velocity equation
is directly parabolic, while decay of the specific volume is transmitted
through the hyperbolic coupling. The general mechanism behind this indirect
dissipation was developed in the Shizuta--Kawashima and Kawashima theories
\cite{ShizutaKawashima1985,Kawashima1987}.

The asymptotic problem contains substantially more structure than global decay
alone. The linearized system has two sound speeds, and a localized perturbation
separates into two diffusive packets traveling in opposite characteristic
directions. Since \(p''(1)\ne0\), the quadratic self-interaction in each
genuinely nonlinear family survives on the diffusive scale and changes the
corresponding linear heat packet into a convected Burgers wave. The mass
carried by each characteristic component is conserved and determines the
associated Burgers profile. Thus the diffusion waves are not merely convenient
approximations; they describe the nonlinear effective dynamics selected by the
low-frequency part of the system.

Pointwise information is indispensable for resolving this asymptotic
structure. A Burgers diffusion wave is Gaussian in a strip of width
\(O(\sqrt t)\) around its own characteristic line, but is exponentially small
near the spatial origin and throughout most of the region between the two
outgoing lines. Cross-family interactions, by contrast, leave algebraically
decaying signals in precisely these regions. A global \(L^p\) estimate may
identify the dominant mass-carrying waves while concealing such interior
tails. A space--time pointwise description must therefore distinguish the
receiving characteristic cone, the opposite cone, the region between them, and
the exterior region.

The diffusion-wave theory for viscous conservation laws was initiated by
Chern and Liu \cite{ChernLiu1987}, who proved convergence toward
superpositions of heat and Burgers diffusion waves in \(L^p\). Their
correction \cite{ChernLiu1989} revealed the important fact that cumulative
cross-family interactions can change, and even determine, the optimal decay
rate. Zeng \cite{Zeng1994} established the corresponding \(L^1\) asymptotics
for one-dimensional isentropic viscous flow. Liu and Zeng subsequently
developed a systematic pointwise Green-function theory for general
quasilinear hyperbolic--parabolic systems, resolving the characteristic
families, nonlinear diffusion waves, and their secondary interactions
\cite{LiuZeng1997,LiuZeng2009}. Liu and Wang obtained related pointwise
diffusion-wave results for Navier--Stokes flow in odd spatial dimensions
\cite{LiuWang1998}.

Liu and Xin \cite{LiuXin1997} made another foundational contribution in the
setting of viscous contact waves, where they identified the shifted contact
profile and transverse diffusion waves and controlled the higher-order
deviation through an approximate fundamental solution and a detailed
wave-interaction analysis. Xin and Zeng extended this framework to general
perturbations \cite{XinZeng2010}. Liu and Yu separated the short-time singular
waves from the long-time diffusive waves for the one-dimensional Boltzmann
equation and obtained a pointwise description of convergence to equilibrium
\cite{LiuYu2004}. Deng and Yu subsequently developed a regular--singular
Green-function construction for the compressible Navier--Stokes equations
\cite{DengYu2017}. Most directly relevant to the present work, Liu and Yu
established a detailed Green-function theory for one-dimensional gas dynamics,
including propagation of singularities, pointwise large-time behavior,
continuous dependence, and weak-solution well-posedness
\cite{LiuYu2022}. Their regular--singular Green-function estimates provide one
of the principal analytic foundations of our pointwise argument.

A higher-order description of the two-component viscous \(p\)-system was
developed by van Baalen, Popovi\'c, and Wayne
\cite{vanBaalenPopovicWayne2008}. They constructed arbitrarily high-order
asymptotic corrections in an \(L^2\)-based framework and showed that
hyperbolic--parabolic coupling produces long algebraic tails. Building on this
perspective, Koike \cite{Koike2023} introduced a recursive hierarchy of
higher-order diffusion waves and proved pointwise remainder estimates after
every fixed finite number of corrections. These estimates capture the
power-law asymptotics in regions where the leading Burgers waves are
exponentially small and, after integration, yield expansions in every \(L^p\),
including \(L^1\). The cross-family viscous derivative corrections in Koike's
expansion are also essential near the opposite characteristic cone.

These results lead naturally to a pointwise endpoint problem. If \(n\) denotes
the number of retained higher-order waves, the exponent
\[
 \alpha_n=2-2^{-(n+1)}
\]
in Koike's pointwise weight approaches the limiting value \(2\). Equivalently,
the power-law exponent in the global \(L^\infty\) estimate approaches
\((1+t)^{-1}\). However, the estimates in \cite{Koike2023} are formulated for
each fixed finite order, with constants allowed to depend on that order, and
therefore do not justify a direct passage \(n\to\infty\). Van Baalen,
Popovi\'c, and Wayne \cite{vanBaalenPopovicWayne2008} identified a
logarithmically weighted infinite-order limit in their \(L^2\) framework,
whereas Koike formulated the corresponding pointwise endpoint as a
conjectural estimate in \cite[Remark~2.5]{Koike2023}. The purpose of the
present paper is to estimate the physical remainder directly after subtracting
the entire convergent hierarchy together with the cross-family viscous
corrections.

Let
\[
 \Xi_i=\sum_{n\ge1}\xi_i^{(n)}
\]
denote the convergent sum of the higher-order diffusion waves, and let
\(R_i\), defined in \eqref{eq:physical-remainder}, be the physical remainder
after the two Burgers waves, the resummed hierarchy, and the cross-family
viscous corrections have been removed. Theorem~\ref{thm:main-endpoint} proves
that
\[
 |R_i(x,t)|
 \leq C E_N\log(2+t)\Psi_i(x,t),
 \qquad
 \|R_i(t)\|_{L^\infty}
 \leq C E_N(1+t)^{-1}\log(2+t),
\]
where \(\Psi_i\) is defined in \eqref{eq:endpoint-weight}. We also obtain a
differentiated pointwise estimate and a uniform \(H^N\) bound for the physical
remainder.

We call the pointwise bound above cone-preserving because its weight
\(\Psi_i\) retains both convected spatial variables
\[
 x-\lambda_i(1+t)
 \qquad\hbox{and}\qquad
 x-\lambda_{3-i}(1+t),
\]
rather than replacing them by a single time-dependent \(L^\infty\) envelope.
Its different components describe the receiving and opposite tubes
\[
 |x-\lambda_i(1+t)|=O(\sqrt{1+t}),
 \qquad
 |x-\lambda_{3-i}(1+t)|=O(\sqrt{1+t}),
\]
the region between the two outgoing characteristic tubes, and the exterior
region beyond them. A signal emitted by one characteristic family and received
by the other therefore remains localized relative to the correct pair of
characteristic cones. The global \(L^\infty\) estimate follows by taking the
spatial supremum, but this emitting--receiving geometry is then lost.

The resulting weight is weaker than the full conjectural pointwise weight in
\cite[Remark~2.5]{Koike2023}, and we do not claim every component of that
conjecture. Nevertheless, it reaches the limiting global algebraic exponent,
up to a single logarithmic loss, for the exact remainder after subtraction of
the entire convergent hierarchy. This is fundamentally different from taking
an arbitrarily high but fixed finite-order expansion.

The proof is based on two successive structural decompositions. First, the zero
mass of \(R_i\) permits the introduction of the spatial antiderivative
\[
 W_i(x,t)=\int_{-\infty}^{x}R_i(y,t)\,dy.
\]
Its equation contains the critical same-family first-order feedback
\[
 \theta_i W_{i,x},
\]
where \(\theta_i\) is the Burgers diffusion wave in the \(i\)-th family. The
coefficient \(\theta_i\) decays only at the nonintegrable rate
\(O((1+t)^{-1/2})\), so a direct perturbative treatment of this term is
incompatible with the endpoint estimate. We therefore introduce the
Cole--Hopf weight \(\phi_i\) associated with \(\theta_i\) and set
\[
 Z_i=\phi_iW_i.
\]
The Cole--Hopf identity eliminates the entire term
\(\theta_iW_{i,x}\) exactly, rather than estimating it. Thus the Cole--Hopf
transformation is used here not merely to represent the Burgers profile, but
as a familywise conjugation that removes the principal variable-coefficient
obstruction from the coupled remainder equations.

The Cole--Hopf cancellation alone is not sufficient at the endpoint. The
Duhamel representation of \(Z_i\) contains principal Gaussian outputs whose
coefficients are only borderline integrable in time. Our second structural
step is therefore the time-dependent decomposition
\[
 Z_i(x,t)=a_i(t)\Gamma_i(x,1+t)+Z_i^{\rm rem}(x,t),
\]
where \(\Gamma_i\) is the convected heat kernel in the \(i\)-th family. The
coefficient \(a_i(t)\) is not, in general, the spatial mass
\(\int_{\mathbb R}Z_i(x,t)\,dx\). Instead, it is defined directly from the
Duhamel sources by extracting all principal same-family Gaussian coefficients
together with the cone- and time-localized opposite-family Gaussian
coefficients generated by the ordinary nonlinear terms, the exceptional
feedback, and the explicit profile sources. With this definition, no
potentially logarithmic principal Gaussian remains in
\(Z_i^{\rm rem}\).

After the nonlinear bootstrap is closed, the modal coefficient and the
remainder satisfy
\[
 |a_i(t)|\leq C E_N\log(2+t),
 \qquad
 |a_i'(t)|\leq C E_N(1+t)^{-1},
\]
and, for \(0\leq k\leq2\),
\[
 \|\partial_x^kZ_i^{\rm rem}(t)\|_{L^\infty}
 \leq C E_N(1+t)^{-(k+1)/2}.
\]
Thus the logarithmic loss is confined to the extracted Gaussian mode, while
\(Z_i^{\rm rem}\) satisfies the log-free estimates displayed above.  The
derivative structure of the recovery formula below then converts these bounds
into the endpoint \((1+t)^{-1}\log(2+t)\) decay of the physical remainder.
Writing \(d=\nu/2\), the physical remainder is recovered exactly through
\[
 R_i
 =\phi_i^{-1}\left(Z_{i,x}+\frac{\theta_i}{2d}Z_i\right),
\]
and hence
\[
 R_i
 =\phi_i^{-1}
 \left[
  a_i(t)\left(
   \Gamma_{i,x}+\frac{\theta_i}{2d}\Gamma_i
  \right)
  +\partial_xZ_i^{\rm rem}
  +\frac{\theta_i}{2d}Z_i^{\rm rem}
 \right].
\]
The leading endpoint contribution to \(R_i\) is therefore the derivative of an
explicit Gaussian mode, whereas the remaining terms are controlled by
cone-resolved zero-mass and integrable-source estimates.

This decomposition is indispensable for the exceptional opposite-family
feedback. A coarse estimate gives
\[
 \|\theta_jR_j(\tau)\|_{L^1}
 \lesssim E_N\frac{\log(2+\tau)}{1+\tau},
\]
which would lead to
\[
 (1+t)^{-1}
 \int_0^{t/2}
 \frac{\log(2+\tau)}{1+\tau}\,d\tau
 \sim
 (1+t)^{-1}\log^2(2+t)
\]
and hence lose an additional logarithm. Substituting the exact recovery formula
for \(R_j\) and using
\[
 Z_j=a_j\Gamma_j+Z_j^{\rm rem},
\]
we split the exceptional feedback into an opposite-speed divergence, a full
Gaussian mass already included in \(a_i\), and two zero-mass contributions.
The divergence gains a derivative of the Green kernel, the Gaussian mass is
removed by the definition of \(a_i\), and the zero-mass terms are estimated
through their spatial primitives. This avoids the spurious
\(\log^2(2+t)\) loss without introducing a coupled Volterra system for the
modal coefficients.

The two decompositions are also compatible with the cone-preserving objective.
The Cole--Hopf conjugation is scalar and is performed separately in each
characteristic family; it neither changes the transport speeds nor mixes the
two propagation directions. The subsequent Gaussian extraction distinguishes
full same-family masses from cone-localized opposite-family masses. As a
result, every Duhamel contribution retains its emitting and receiving
characteristic families. The explicit mode \(a_i\Gamma_i\) remains attached to
the \(i\)-th receiving cone, while the nonmodal terms are treated separately
in the receiving tube, the opposite tube, the region between the two tubes,
and the exterior region. This combination of the Cole--Hopf cancellation and
the Gaussian-mode decomposition is the mechanism that simultaneously yields
the endpoint decay and preserves the two-cone spatial geometry.

For the linear propagation, we construct an approximate Green function for the
conjugated system by inserting the familywise Cole--Hopf weights into the
constant-state Green matrix:
\[
 \widetilde G
 =
 \begin{pmatrix}
  G^0_{11}&\dfrac{\phi_1}{\phi_2}G^0_{12}\\[2mm]
  \dfrac{\phi_2}{\phi_1}G^0_{21}&G^0_{22}
 \end{pmatrix},
 \qquad
 G^0=G^r+G^s,
 \qquad
 G^r=\mathcal G+\mathcal R^r.
\]
Here \(\mathcal G\) is the explicit convected-Gaussian approximation determined
by the low-frequency characteristic branches and their spectral projectors.
This construction is adapted simultaneously to the two propagation directions
and the partially dissipative viscous coupling, and its commutator with the
conjugated operator has the divergence and zero-mass structure needed for the
endpoint argument.  The estimates for \(\mathcal R^r\) and \(G^s\) follow from
the regular--singular Green-function theory of Liu and Yu
\cite{LiuYu2022}, in the normalization recorded by Wang and Zhang
\cite{WangZhang2023}.  Finally, a Kawashima compensator combined with a
Gaussian-weighted energy estimate closes the uniform high-order bound.  The
cone weight is imposed only on the zero-order physical remainder; the
differentiated pointwise estimates and the \(H^N\) estimate remain spatially
unweighted.

The remainder of the paper is organized as follows.
Section~\ref{sec:main-results} defines the complete asymptotic profile,
introduces the cone-resolved weight, and states the main theorem.
Section~\ref{sec:exact-remainder} derives the exact physical remainder
equations, introduces the spatial antiderivatives, and performs the familywise
Cole--Hopf conjugation. Section~\ref{sec:green} develops the Green-function representation
for the conjugated system. Section~\ref{sec:pointwise} extracts the critical
Gaussian modes, estimates the modal remainders, and proves the cone-resolved
pointwise mapping. Section~\ref{sec:energy} establishes the uniform high-order energy estimate
and completes the continuation argument. The endpoint convolution estimates
used in the pointwise analysis are proved in the Appendix \ref{app:five-regions}.

\section{Main results}
\label{sec:main-results}

Recall \(c,d,\lambda_1,\lambda_2\), and introduce \(p_2\), as follows:
\[
 c=\sqrt{-p'(1)},\qquad d=\frac{\nu}{2},\qquad p_2=p''(1),
 \qquad \lambda_1=c,\quad \lambda_2=-c.
\]
Throughout the paper $i\in\{1,2\}$, and $3-i$ denotes the other
characteristic family.
Introduce the characteristic variables
\begin{equation}
 u_1=\frac{p_2}{4c}\left[-(v-1)+\frac{u}{c}\right],
 \qquad
 u_2=\frac{p_2}{4c}\left[(v-1)+\frac{u}{c}\right].
 \label{eq:characteristic-variables}
\end{equation}
The transformation is invertible, with
\begin{equation}
 v-1=\frac{2c}{p_2}(u_2-u_1),
 \qquad
 u=\frac{2c^2}{p_2}(u_1+u_2).
 \label{eq:inverse-characteristic-variables}
\end{equation}

Section~\ref{sec:exact-remainder} derives the exact characteristic balance
\eqref{eq:exact-characteristic}.  Its right-hand side is a complete spatial
derivative; hence, for the decaying solutions considered here,
\[
 \frac{d}{dt}\int_{\mathbb R}u_i(x,t)\,dx=0.
\]
Assume $u_i(\cdot,0)\in L^1(\mathbb R)$.  For $i=1,2$, define the conserved
mass and its total size by
\begin{equation*}
 M_i=\int_{\mathbb R}u_i(x,t)\,dx
    =\int_{\mathbb R}u_i(x,0)\,dx,
 \qquad \epsilon=|M_1|+|M_2|.
\end{equation*}
For $s>0$ define the unit-mass convected heat kernel
\begin{equation}
 \Gamma_i(x,s)=\frac{1}{\sqrt{4\pi ds}}
 \exp\left[-\frac{(x-\lambda_i s)^2}{4ds}\right],
 \qquad i=1,2.
 \label{eq:receiving-gaussian}
\end{equation}
Let $\delta_0$ denote the Dirac mass at $x=0$.  The centered Burgers wave is
not uniquely determined by its equation and mass.  We
choose the canonical solution issued from the point mass at time $t=-1$:
\begin{equation}
 \theta_i(x,t)=
 \frac{2d\bigl(1-e^{-M_i/(2d)}\bigr)\Gamma_i(x,1+t)}
 {1+\bigl(e^{-M_i/(2d)}-1\bigr)
       \displaystyle\int_{-\infty}^x\Gamma_i(y,1+t)\,dy}.
 \label{eq:explicit-burgers-wave}
\end{equation}
The Cole--Hopf formula and its standard estimates
\cite[(3)--(4)]{Koike2023} give
\begin{equation}
 \theta_{i,t}+\lambda_i\theta_{i,x}+\theta_i\theta_{i,x}
   =d\theta_{i,xx},
 \qquad
 \int_{\mathbb R}\theta_i(x,t)\,dx=M_i,
 \qquad
 \theta_i(\cdot,t)\rightharpoonup M_i\delta_0
 \quad\text{as }t\downarrow-1.
 \label{eq:burgers-wave}
\end{equation}
The backward trace at $t=-1$ fixes the translation of the self-similar profile;
the physical solution is
considered only for $t\geq0$.
It satisfies, for every integer $k\geq0$,
\begin{equation}
 |\partial_x^k\theta_i(x,t)|
 \leq C_k\epsilon(1+t)^{-(k+1)/2}
 \exp\left[-\frac{[x-\lambda_i(1+t)]^2}{C(1+t)}\right].
 \label{eq:burgers-pointwise-bound}
\end{equation}

Here \(n\geq0\) denotes the order in the higher-wave hierarchy.
Set
\begin{equation*}
 \xi_i^{(0)}=\frac12\theta_i,
 \qquad i=1,2,
\end{equation*}
and, for every $n\geq1$, define $\xi_i^{(n)}$ recursively by
\begin{equation*}
 \left\{
 \begin{aligned}
 &\partial_t\xi_i^{(n)}+\lambda_i\partial_x\xi_i^{(n)}
 +\partial_x(\theta_i\xi_i^{(n)})
 +\partial_x(\theta_{3-i}\xi_{3-i}^{(n-1)})
 =d\partial_x^2\xi_i^{(n)},\\
 &\xi_i^{(n)}(x,0)=0.
 \end{aligned}
 \right.
 \qquad i=1,2.
\end{equation*}
This is Koike's recursion \cite[(7)]{Koike2023}, with seed
\(\xi_i^{(0)}=\theta_i/2\).  By
\cite[(24)--(25), Remark~3.1 and Lemma~3.3]{Koike2023},
for sufficiently small masses the series
\begin{equation*}
 \Xi_i(x,t):=\sum_{n=1}^{\infty}\xi_i^{(n)}(x,t)
\end{equation*}
converges, together with every fixed finite number of spatial derivatives used
below, in Koike's pointwise decay class, and its sum is the classical solution
of
\begin{equation}
 \left\{
 \begin{aligned}
 &\Xi_{i,t}+\lambda_i\Xi_{i,x}+\partial_x(\theta_i\Xi_i)
  +\partial_x\left[\theta_{3-i}
       \left(\frac12\theta_{3-i}+\Xi_{3-i}\right)\right]
   =d\Xi_{i,xx},\\
 &\Xi_i(x,0)=0.
 \end{aligned}
 \right.
 \label{eq:resummed-profile}
\end{equation}
The same cited results give, for each fixed nonnegative integer \(k\),
\begin{align}
 |\partial_x^k\Xi_i(x,t)|
 &\leq \frac{C_k\epsilon^2(1+t)^{-k/2}}
 {\big([x-\lambda_i(1+t)]^2+1+t\big)^{3/4}}
 +\frac{C_k\epsilon^2(1+t)^{-k/2}}
 {|x-\lambda_{3-i}(1+t)|^{3/2}+1+t},
 \label{eq:resummed-pointwise-bound}\\
 \|\partial_x^k\Xi_i(t)\|_{L^q}
 &\leq C_{k,q}\epsilon^2
 (1+t)^{-3/4+1/(2q)-k/2},
 \qquad 1\leq q\leq\infty.
 \label{eq:resummed-Lq-bound}
\end{align}
In \cite[Lemma~3.3]{Koike2023}, \(n=0\) denotes the summed tail
\(\Xi_i=\sum_{m\geq1}\xi_i^{(m)}\), rather than the seed used above.
The bounds justify integration of \eqref{eq:resummed-profile} over
\(\mathbb R\).  Every term except
$\Xi_{i,t}$ is a spatial derivative, so
\begin{equation}
 \frac{d}{dt}\int_{\mathbb R}\Xi_i(x,t)\,dx=0,
 \qquad
 \int_{\mathbb R}\Xi_i(x,t)\,dx=0
 \label{eq:resummed-zero-mass}
\end{equation}
by the zero initial condition.

With
\begin{equation*}
 \gamma_i=(-1)^i\frac{\nu}{4c},\qquad i=1,2,
\end{equation*}
define the approximate profile and the physical remainder by
\begin{equation}
 u_i=u_i^{\mathrm{app}}+R_i,
 \qquad
 u_i^{\mathrm{app}}
 :=\theta_i+\Xi_i+\gamma_{3-i}\partial_x(\theta_{3-i}+\Xi_{3-i}).
 \label{eq:physical-remainder}
\end{equation}
The correction \(u_i^{\mathrm{app}}-\theta_i\) has zero spatial mass
by \eqref{eq:resummed-zero-mass}; the zero-mass property of \(R_i\), needed for
its antiderivative, is proved at \eqref{eq:remainder-zero-mass}.  The immediate
\(L^q\) bounds for \(u_i^{\mathrm{app}}-\theta_i\) follow from
\eqref{eq:burgers-pointwise-bound} and \eqref{eq:resummed-Lq-bound}.

To state the pointwise estimate, fix a sufficiently large $K$ and choose
$0\leq\chi_K\leq1$ so that
\begin{align*}
 \chi_K(x,t)&=1
 &&\text{if }\lambda_2(1+t)+2K\sqrt{1+t}
 \leq x\leq\lambda_1(1+t)-2K\sqrt{1+t},\\
 \chi_K(x,t)&=0
 &&\text{if }x\leq\lambda_2(1+t)+K\sqrt{1+t}
 \text{ or }x\geq\lambda_1(1+t)-K\sqrt{1+t}.
\end{align*}
Any fixed smooth interpolation in the two intervening strips may be used.  Set
\begin{align}
 \Psi_i(x,t)=&
 \frac1{[x-\lambda_i(1+t)]^2+1+t}
 +\frac{1+\log(2+t)}{[x-\lambda_{3-i}(1+t)]^2+(1+t)^{3/2}}
 \notag\\
 &+\chi_K(x,t)
 \frac1{\big(1+[x-\lambda_i(1+t)]^2\big)^{{1/2}}
          \big(1+[x-\lambda_{3-i}(1+t)]^2\big)^{1/4}}
 \notag\\
 &+(1+t)^{-1/2}
 \left\{
 \frac1{\big([x-\lambda_i(1+t)]^2+1+t\big)^{3/4}}
 +\frac1{|x-\lambda_{3-i}(1+t)|^{3/2}+1+t}
 \right\}.
 \label{eq:endpoint-weight}
\end{align}
Directly from \eqref{eq:endpoint-weight},
\begin{equation*}
 \sup_{x\in\mathbb R}\Psi_i(x,t)\leq C(1+t)^{-1},
 \qquad
 \|\Psi_i(t)\|_{L^1}\leq C(1+t)^{-1/2}{\log(2+t)}.
\end{equation*}

For an integer $N\geq5$, let
\begin{equation}
\begin{aligned}
 E_N=&|M_1|+|M_2|+\|(v(\cdot,0)-1,u(\cdot,0))\|_{H^{N+2}}\\
 &+\sum_{i=1}^2\left[
 \|\langle x\rangle^2R_i(\cdot,0)\|_{H^{N+1}}
 +\int_{\mathbb R}\langle x\rangle^2|R_i(x,0)|\,dx\right],
 \qquad \langle x\rangle=(1+x^2)^{1/2}.
 \label{eq:initial-size}
\end{aligned}
\end{equation}
Here the initial remainder is explicit because \(\Xi_i(\cdot,0)=0\):
\begin{equation}
 R_i(x,0)=u_i(x,0)-\theta_i(x,0)
 -\gamma_{3-i}\partial_x\theta_{3-i}(x,0).
 \label{eq:initial-remainder-explicit}
\end{equation}
By \eqref{eq:initial-size}, $\epsilon\leq E_N$; hence smallness of $E_N$ also supplies the
small-mass assumption used in the Burgers and Koike profiles.

Fix an open interval $\mathcal I\subset(0,\infty)$ containing $1$.  Choose
$\rho>0$ so that the two compact intervals
\begin{equation}
 I_{\mathrm{in}}=[1-\rho,1+\rho]
 \subset \operatorname{int}I_0,
 \qquad
 I_0=[1-2\rho,1+2\rho]\subset\mathcal I,
 \qquad \operatorname{dist}(I_0,\partial\mathcal I)>0.          \label{eq:fixed-state-intervals}
\end{equation}
are fixed once and for all.  The outer interval $I_0$ is used in every
coefficient and energy constant below; $I_{\mathrm{in}}$ leaves a fixed buffer
for the continuation argument.

\begin{lemma}
\label{lem:local-continuation}
Let $s\geq3$, let $p\in C^{s+2}(\mathcal I)$, and suppose
\begin{equation}
 \inf_{v\in I_0}[-p'(v)]>0.                                  \label{eq:fixed-hyperbolicity}
\end{equation}
If $(v_0-1,u_0)\in H^s(\mathbb R)$ and
$v_0(\mathbb R)\subset I_{\mathrm{in}}$, then there is a unique strong
solution on a nontrivial interval $[0,T]$ such that
\begin{equation*}
 (v-1,u)\in C([0,T];H^s),
 \qquad u_x\in L^2(0,T;H^s),
 \qquad v(\mathbb R\times[0,T])\subset\operatorname{int}I_0.
\end{equation*}
Let $[0,T_{\max})$ be its maximal interval relative to the last range
condition.  If $T_{\max}<\infty$, then at least one of
\begin{align}
 &\liminf_{t\uparrow T_{\max}}
   \inf_{x\in\mathbb R}\operatorname{dist}
       (v(x,t),\partial I_0)=0,                               \label{eq:continuation-range}\\
 &\sup_{0\leq t<T_{\max}}\|(v-1,u)(t)\|_{H^s}=\infty
 \quad\text{or}\quad
 \int_0^{T_{\max}}\|u_x(t)\|_{H^s}^2\,dt=\infty.              \label{eq:continuation-sobolev}
\end{align}
must occur.

For the data measured by $E_N$, take $s=N$ in the preceding construction and
let $T_{\max}$ denote this $H^N$ lifespan.  The additional regularity in
\eqref{eq:main-data-class} and the polynomially weighted Sobolev norm in
\eqref{eq:initial-size} persist on every compact subinterval of
$[0,T_{\max})$: for every $T<T_{\max}$,
\begin{equation}
 \sup_{0\leq t\leq T}\left\{
 \|(v-1,u)(t)\|_{H^{N+2}}
 +\sum_{i=1}^2\|\langle x\rangle^2R_i(t)\|_{H^{N+1}}\right\}
 +\int_0^T\|u_x(t)\|_{H^{N+2}}^2\,dt<\infty .              \label{eq:local-weighted-persistence}
\end{equation}
The separately listed weighted $L^1$ moment is used only at the initial time.
In particular, all entries of $X_N(T)$, defined later in
\eqref{eq:pointwise-continuity-norm}, are finite.  Moreover, if at some
$T<T_{\max}$ the range of $v$ has positive distance from $\partial I_0$ and
the bounds defining $X_N(T)$ are strict, then the same bounds, with an
arbitrarily small enlargement, persist on $[0,T+\delta_T]$ for some
$\delta_T>0$.
\end{lemma}

\begin{proof}
The unweighted existence, uniqueness, and continuation assertions follow from
the standard local theory for viscous systems of conservation laws
\cite[Theorem~1.2]{Serre2010}: on the fixed positive range $I_0$, the
condition $-p'(v)>0$ supplies the symmetrizer and the viscosity matrix has
constant rank.  We use that theorem without reproducing its Friedrichs
construction.

We record the persistence step because it is used in the first-exit argument.
Fix $T<T_{\max}$.  The standard differentiated symmetrized energy estimate,
first at orders $N+1$ and $N+2$ and then by regularization, has coefficients
controlled on $[0,T]$ by the $H^N$ norm and the fixed distance of the range of
$v$ from $\partial I_0$.  Gronwall's inequality therefore propagates
$H^{N+2}$ on $[0,T]$; uniqueness identifies this higher-regularity solution
with the $H^N$ solution.  For the weighted assertion, write
$w(x)=\langle x\rangle^2$.  Commuting the exact remainder equations with
$w\partial_x^k$, $0\leq k\leq N+1$, produces only the commutators
$[\partial_x,w]=w_x$ and
$[\partial_x^2,w]=2w_x\partial_x+w_{xx}$.  Since
$|w_x|\leq2w$ and $|w_{xx}|\leq2w$, these terms contain no higher weight and
no higher derivative than the principal weighted energy.  The background
forcing is smooth and has finite weighted Sobolev norms on every bounded time
interval: \(\theta_i\) is Gaussian, and the zero-initial-data system
\eqref{eq:resummed-profile} propagates every fixed polynomial weight by the
same commutator estimate.  Substitution in the exact background residual then
has the same property.  The symmetrized estimate and Gronwall's inequality give
\eqref{eq:local-weighted-persistence}.

By \eqref{eq:remainder-zero-mass} below, $\int R_i=0$; hence
Cauchy--Schwarz and Fubini give
\[
 \|R_i(t)\|_1+\|W_i(t)\|_1
 \leq C\|\langle x\rangle^2R_i(t)\|_2,
\]
so the definitions of $W_i$, $Z_i$, and $Z_i^{\mathrm{rem}}$ give local persistence
of the entries of $X_N$.  Finally, differentiation of
\eqref{eq:ai-definition} under the integral sign gives an endpoint term and a
term containing the derivative of its fixed smooth cutoff.  On every compact
time interval both are dominated by the weighted norms just proved.  Thus
$a_i$ is locally Lipschitz, the local essential supremum of $a_i'$ is stable
under a short continuation, and suprema involving $a_i'$ are understood as
essential suprema.  This proves the last assertion as well.
\end{proof}

\begin{theorem}
\label{thm:main-endpoint}
Let $N\geq5$, let $p\in C^{N+5}(\mathcal I)$ satisfy
$p''(1)\ne0$ and \eqref{eq:fixed-hyperbolicity}, and let the initial data
satisfy
\begin{equation}
 (v_0-1,u_0)\in H^{N+2}(\mathbb R),\qquad
 v_0(\mathbb R)\subset I_{\mathrm{in}},\qquad
 u_i(\cdot,0)\in L^1(\mathbb R),\qquad E_N<\infty.           \label{eq:main-data-class}
\end{equation}
Let $(v,u)$ be the maximal $H^N$ strong solution supplied by
Lemma~\ref{lem:local-continuation}; the additional two derivatives in
\eqref{eq:main-data-class} are used in the pointwise initial estimates.
There is an $\epsilon_0>0$, depending only on $N,\nu,I_0$, the fixed buffer
in \eqref{eq:fixed-state-intervals}, and upper bounds for
$\|p\|_{C^{N+5}(I_0)}$, $\|(-p')^{-1}\|_{L^\infty(I_0)}$, and
$|p''(1)|^{-1}$, such that, if
$E_N\leq\epsilon_0$, then $T_{\max}=\infty$,
$v(x,t)\in I_{\mathrm{in}}$ for all $(x,t)$, and the decomposition
\eqref{eq:physical-remainder} satisfies, for $i=1,2$,
\begin{equation}
 |R_i(x,t)|\leq C E_N\log(2+t)\Psi_i(x,t),
 \qquad x\in\mathbb R,\quad t\geq0.
 \label{eq:main-pointwise-estimate}
\end{equation}
Moreover,
\begin{equation}
 \|R_i(t)\|_{L^\infty}
 \leq C E_N(1+t)^{-1}\log(2+t),
 \qquad
 \|R_{i,x}(t)\|_{L^\infty}
 \leq C E_N(1+t)^{-3/2}\log(2+t),
 \label{eq:main-uniform-estimates}
\end{equation}
and
\begin{equation}
 \sup_{t\geq0}\|(R_1,R_2)(t)\|_{H^N}\leq C E_N.
 \label{eq:main-energy-estimate}
\end{equation}
The original variables are recovered exactly by
\begin{align}
 v-1&=\frac{2c}{p_2}
 \big[(u_2^{\mathrm{app}}-u_1^{\mathrm{app}})+(R_2-R_1)\big],
 \label{eq:recover-v} \qquad
 u=\frac{2c^2}{p_2}
 \big[(u_1^{\mathrm{app}}+u_2^{\mathrm{app}})+(R_1+R_2)\big].
\end{align}
\end{theorem}

\begin{remark}
The estimate is asserted for the physical remainder $R_i$, not for the
Cole--Hopf proof variable.  No cone-weighted estimate for higher derivatives is
claimed.  The second term in \eqref{eq:endpoint-weight} permits the faster
opposite-family response to carry one additional logarithm, but its spatial
supremum is still $O((1+t)^{-1})$ after multiplication by the single logarithm
in \eqref{eq:main-pointwise-estimate}.
\end{remark}

\section{Exact equations for the remainder}
\label{sec:exact-remainder}

\subsection{The characteristic equations}

In the characteristic variables introduced in \eqref{eq:characteristic-variables},
the two nonlinear functions which occur throughout the proof are
\begin{align}
 F(z)
 &:=-\frac{p_2}{4c^2}
 \left[p\left(1+\frac{2c}{p_2}z\right)-p(1)+\frac{2c^3}{p_2}z\right],\qquad
 B(z)
 :=d\left[\left(1+\frac{2c}{p_2}z\right)^{-1}-1\right].
 \notag
\end{align}
{Using \eqref{eq:inverse-characteristic-variables}, the pressure
part is}
\begin{equation}
 \partial_xF(u_2-u_1),                                      \label{eq:pressure-substitution}
\end{equation}
whereas the viscous part is
\begin{align}
 \frac{p_2\nu}{4c^2}\partial_x\left(\frac{u_x}{v}\right)
 % &=d\partial_x\!\left[
 % \frac{u_{1,x}+u_{2,x}}{1+\frac{2c}{p_2}(u_2-u_1)}\right]\notag\\
 % &
 =d(u_{1,xx}+u_{2,xx})
 +\partial_x\left\{B(u_2-u_1)(u_{1,x}+u_{2,x})\right\}.
 \label{eq:viscosity-substitution}
\end{align}
Combining \eqref{eq:pressure-substitution}--
\eqref{eq:viscosity-substitution} gives
\begin{equation}\label{eq:exact-characteristic}
 u_{i,t}+\lambda_i u_{i,x}-d(u_{1,xx}+u_{2,xx})
 =\partial_x \left\{
 F(u_2-u_1)+B(u_2-u_1)(u_{1,x}+u_{2,x})
 \right\},
 \quad i=1,2.
\end{equation}
The condition $1+\frac{2c}{p_2}(u_2-u_1)>0$ is exactly $v>0$.
Taylor's formula is used below only in the following exact form:
\begin{align}
 F(z)&=-\frac12z^2-\frac{cp'''(1)}{3p_2^2}z^3+F_4(z),
 \label{eq:F-exact-expansion}\\
 F_4(z)&=-\frac{2c^2z^4}{3p_2^3}
 \int_0^1(1-\zeta)^3
 p^{(4)}\left(1+\frac{2c}{p_2}\zeta z\right)d\zeta,
 \label{eq:F4}\\
 B(z)&=-\frac{\nu c}{p_2}z+B_2(z),
 \qquad
 B_2(z)=\frac{4dc^2z^2}{p_2^2\left(1+\frac{2c}{p_2}z\right)}.
 \label{eq:B-exact-expansion}
\end{align}
Thus $F_4$ and $B_2$ are explicit.

\subsection{The background and its exact residual}

Equivalently, the two components of the approximate profile are
\begin{equation}\label{eq:background-pair}
 u_1^{\mathrm{app}}=\theta_1+\Xi_1+\frac{\nu}{4c}\partial_x(\theta_2+\Xi_2),
 \qquad
 u_2^{\mathrm{app}}=\theta_2+\Xi_2-\frac{\nu}{4c}\partial_x(\theta_1+\Xi_1).
\end{equation}
The profile bounds imply
\begin{equation*}
 \sup_{x\in\mathbb R,\ t\geq0}|u_2^{\mathrm{app}}(x,t)-u_1^{\mathrm{app}}(x,t)|
 \leq C(\epsilon+\epsilon^2).
\end{equation*}
After reducing the mass threshold if necessary,
\begin{equation*}
 v^{\mathrm{app}}(x,t):=1+\frac{2c}{p_2}(u_2^{\mathrm{app}}-u_1^{\mathrm{app}})
 \in I_{\mathrm{in}}
 \qquad (x\in\mathbb R,\ t\geq0).
\end{equation*}
In particular, \(v^{\mathrm{app}}\geq1/2\).  Thus
\(B(u_2^{\mathrm{app}}-u_1^{\mathrm{app}})\) and \(B_2(u_2^{\mathrm{app}}-u_1^{\mathrm{app}})\) are uniformly
well defined.  Since \(I_{\mathrm{in}}\) is an interval containing \(1\), every
argument of \(p\) in \eqref{eq:F4} also stays in \(I_{\mathrm{in}}\).
{Define the primitive residual by}
\begin{align}
 \mathcal E_i^{\mathrm{app}}(x,t)
 :=\int_{-\infty}^{x}&\Bigg[
 u_{i,t}^{\mathrm{app}}+\lambda_iu_{i,y}^{\mathrm{app}}
 -d(u_{1,yy}^{\mathrm{app}}+u_{2,yy}^{\mathrm{app}})
 \notag\\[-2mm]
 &-\partial_y\left\{
 F(u_2^{\mathrm{app}}-u_1^{\mathrm{app}})
 +B(u_2^{\mathrm{app}}-u_1^{\mathrm{app}})(u_{1,y}^{\mathrm{app}}+u_{2,y}^{\mathrm{app}})
 \right\}\Bigg](y,t)\,dy.
 \label{eq:primitive-background-residual}
\end{align}
By \eqref{eq:burgers-wave}, \eqref{eq:resummed-zero-mass}, and
\eqref{eq:background-pair},
\(\int_{\mathbb R}u_i^{\mathrm{app}}(x,t)\,dx=M_i\).  All terms in the integrand of
\eqref{eq:primitive-background-residual} are integrable: first use
\eqref{eq:burgers-wave} and \eqref{eq:resummed-profile} to replace the time
derivatives, and then apply \eqref{eq:burgers-pointwise-bound} and
\eqref{eq:resummed-pointwise-bound}.  Hence
\[
 \int_{\mathbb R}(\text{residual integrand})(x,t)\,dx
 =\frac d{dt}\int_{\mathbb R}u_i^{\mathrm{app}}(x,t)\,dx=0.
\]
Consequently \(\mathcal E_i^{\mathrm{app}}\) vanishes at both spatial infinities.

The weight required in the residual decomposition and in the remainder equation is
\begin{equation}\label{eq:def-phi}
 \phi_i(x,t):=\exp\left[-\frac1{2d}
 \int_{-\infty}^x\theta_i(y,t)\,dy\right].
\end{equation}
{Equation \eqref{eq:explicit-burgers-wave} gives}
\begin{align}
 &\phi_i(x,t)
 =1+\bigl(e^{-M_i/(2d)}-1\bigr)
       \int_{-\infty}^x\Gamma_i(y,1+t)\,dy,\notag\\
 &\phi_{i,x}
 =\bigl(e^{-M_i/(2d)}-1\bigr)\Gamma_i(x,1+t), \qquad
 \phi_i\theta_i
 =2d\bigl(1-e^{-M_i/(2d)}\bigr)\Gamma_i(x,1+t).
 \label{eq:phi-explicit-identities}
\end{align}
Writing \(F_i(x,t)=\int_{-\infty}^x\Gamma_i(y,1+t)\,dy\), we have
\[
 \phi_i=(1-F_i)+e^{-M_i/(2d)}F_i,
 \qquad 0\leq F_i\leq1.
\]
Therefore
\begin{equation*}
 \min\{1,e^{-M_i/(2d)}\}\leq\phi_i(x,t)
 \leq\max\{1,e^{-M_i/(2d)}\},
 \qquad
 \|\phi_i\|_\infty+\|\phi_i^{-1}\|_\infty\leq2e^{|M_i|/(2d)}.
\end{equation*}
Moreover, \eqref{eq:def-phi} is the unique positive solution of
\(2d\phi_{i,x}+\theta_i\phi_i=0\) normalized by
\(\lim_{x\to-\infty}\phi_i(x,t)=1\), and for small \(|M_i|\),
\begin{equation*}
 \|\phi_i-1\|_\infty+\|\phi_i^{-1}-1\|_\infty\leq C|M_i|.
\end{equation*}
Combining the convected heat equation for $\Gamma_i$ with
\eqref{eq:def-phi}, and using the Burgers equation for $\theta_i$, gives
\begin{equation}\label{eq:phi-identities}
 (\partial_t+\lambda_i\partial_x-d\partial_x^2)\phi_i=0,
 \qquad 2d\phi_{i,x}+\theta_i\phi_i=0.
\end{equation}
Whenever a later formula contains \(1-e^{-M_i/(2d)}\) in a denominator, it is
understood by continuous extension at \(M_i=0\), using the exact identity
\begin{equation*}
 \frac{\theta_i}{1-e^{-M_i/(2d)}}=\frac{2d\,\Gamma_i}{\phi_i}.
\end{equation*}
This convention also covers products with higher powers of \(\theta_i\).

For later estimates the primitive residual must be decomposed before absolute values
are taken.  Put
\begin{equation*}
 \mu_{\mathrm{cub}}:=\frac{cp'''(1)}{3p_2^2}+\frac{3}{16c}+\frac{c}{2p_2}
\end{equation*}
and
\begin{align}
 V_i^{\mathrm{bg}}
 :=&\;(-1)^id\phi_i
 \left[\frac{d}{2c}\theta_{i,x}
 +\left(\frac{1}{8c}-\frac{c}{p_2}\right)\theta_i^2\right]
 +(-1)^{i+1}d\left(\frac{1}{2c}-\frac{c}{p_2}\right)
 \phi_i\theta_{3-i}^2.
 \label{eq:V-bg}
\end{align}
The algebra needed below can be recorded at its first use.  Adding
\eqref{eq:burgers-wave} and \eqref{eq:resummed-profile} gives
\begin{equation*}
 (\theta_i+\Xi_i)_t+\lambda_i(\theta_i+\Xi_i)_x
 -d(\theta_i+\Xi_i)_{xx}
 =-\partial_x\left[
 \frac12(\theta_1^2+\theta_2^2)+\theta_1\Xi_1+\theta_2\Xi_2\right].
\end{equation*}
After one integration in \(x\), substitution of \eqref{eq:background-pair}
into \eqref{eq:primitive-background-residual} reduces the quadratic terms by
the exact cancellation
\begin{align}
 &\frac12(\theta_1^2+\theta_2^2)+\theta_1\Xi_1+\theta_2\Xi_2
 -\frac12(\theta_2+\Xi_2-\theta_1-\Xi_1)^2=(\theta_1+\Xi_1)(\theta_2+\Xi_2)
 -\frac12(\Xi_1^2+\Xi_2^2).
 \label{eq:key-quadratic-cancellation}
\end{align}
The two weighted derivative identities
\begin{align*}
 &\phi_i\theta_\alpha\theta_{\alpha,x}
 =\frac12\partial_x(\phi_i\theta_\alpha^2)
 +\frac{\theta_i}{4d}\phi_i\theta_\alpha^2,\\
 &\phi_i\theta_{i,xx}
 =\partial_x\left(\phi_i\theta_{i,x}
 +\frac{1}{4d}\phi_i\theta_i^2\right)
 +\frac{1}{8d^2}\phi_i\theta_i^3
\end{align*}
place every pure-family derivative in \(\partial_xV_i^{\mathrm{bg}}\).
The terms left outside that divergence are defined exactly by
\begin{align}
 \mathcal R_i^{\mathrm{int}}=&
 \phi_i(\theta_1\theta_2+\Xi_1\Xi_2)
 +(-1)^{i+1}\left(\frac{1}{4c}-\frac{c}{2p_2}\right)
 \phi_i\theta_i\theta_{3-i}^2
 \notag\\
 &+\phi_i\Bigg\{
 (-1)^{i+1}\frac{\nu}{4c}(\theta_1\Xi_1+\theta_2\Xi_2)_x
 +(-1)^i\frac{d\nu}{4c}\Xi_{i,xx}-\frac{d^2}{8c^2}
 \left[\partial_x(\theta_1+\Xi_1+\theta_2+\Xi_2)\right]^2
 \notag\\
 &\qquad\quad+\frac{d}{2c}\Big[
 (\theta_2+\Xi_2-\theta_1-\Xi_1)
 \partial_x(\theta_1+\Xi_1+\theta_2+\Xi_2)
 +\theta_1\theta_{1,x}-\theta_2\theta_{2,x}\Big]
 \notag\\
 &\qquad\quad-\frac{cp'''(1)}{3p_2^2}\Big[
 (u_2^{\mathrm{app}}-u_1^{\mathrm{app}})^3
 +\theta_1^3-\theta_2^3\Big]
 +F_4(u_2^{\mathrm{app}}-u_1^{\mathrm{app}})
 \notag\\
 &\qquad\quad-\frac{\nu c}{p_2}\Big[
 (u_2^{\mathrm{app}}-u_1^{\mathrm{app}})
 (u_{1,x}^{\mathrm{app}}+u_{2,x}^{\mathrm{app}})
 +\theta_1\theta_{1,x}-\theta_2\theta_{2,x}\Big]
 \notag\\
 &\qquad\quad+B_2(u_2^{\mathrm{app}}-u_1^{\mathrm{app}})
 (u_{1,x}^{\mathrm{app}}+u_{2,x}^{\mathrm{app}})
 \Bigg\}.
 \label{eq:complete-integrable-remainder}
\end{align}
Substitution of \eqref{eq:F-exact-expansion}--\eqref{eq:B-exact-expansion}
and \eqref{eq:key-quadratic-cancellation} gives the exact identity
\begin{align}
 -\phi_i\mathcal E_i^{\mathrm{app}}
 =&(-1)^{i+1}\mu_{\mathrm{cub}}\phi_i\theta_i^3
 +(-1)^i\left(\frac{cp'''(1)}{3p_2^2}+\frac{1}{4c}\right)
 \phi_i\theta_{3-i}^3
 \notag\\
 &+\phi_i\left[
 \theta_1\left(\Xi_2-\frac{\theta_1^2}{4c}\right)
 +\theta_2\left(\Xi_1+\frac{\theta_2^2}{4c}\right)
 \right]
 -\frac12\phi_i(\Xi_1^2+\Xi_2^2)
+\partial_xV_i^{\mathrm{bg}}+\mathcal R_i^{\mathrm{int}}.
 \label{eq:exact-background-decomposition}
\end{align}

The apparently critical mixed expression in
\eqref{eq:exact-background-decomposition} is controlled by
\cite[Lemma~3.3, with \(n=0\)]{Koike2023}.  In Koike's tail notation, index
\(-1\) corresponds to \(\theta_i/2+\Xi_i\), while index \(0\) corresponds to
\(\Xi_i\).  Thus, for every fixed \(k\geq0\),
\begin{align}
 \left|\partial_x^k\left[
 \Xi_2-\frac{\theta_1}{2c}\left(\frac12\theta_1+\Xi_1\right)\right]\right|
 &\leq C_k\epsilon^2(1+t)^{-k/2}
 \bigl([x-\lambda_2(1+t)]^2+1+t\bigr)^{-3/4},\notag\\
 \left|\partial_x^k\left[
 \Xi_1+\frac{\theta_2}{2c}\left(\frac12\theta_2+\Xi_2\right)\right]\right|
 &\leq C_k\epsilon^2(1+t)^{-k/2}
 \bigl([x-\lambda_1(1+t)]^2+1+t\bigr)^{-3/4}.
 \notag
\end{align}
Rewriting the two mixed products with these combinations and using
\eqref{eq:burgers-pointwise-bound}--\eqref{eq:resummed-pointwise-bound} yields
\begin{align}
 &\left\|\partial_x^k\left[
 \theta_1\left(\Xi_2-\frac{\theta_1^2}{4c}\right)
 +\theta_2\left(\Xi_1+\frac{\theta_2^2}{4c}\right)
 \right](t)\right\|_{L^1}
 \leq C_k\epsilon^3(1+t)^{-5/4-k/2},\notag\\
 &\left\|\partial_x^k\left[
 \theta_1\left(\Xi_2-\frac{\theta_1^2}{4c}\right)
 +\theta_2\left(\Xi_1+\frac{\theta_2^2}{4c}\right)
 \right](t)\right\|_{L^\infty}
 \leq C_k\epsilon^3(1+t)^{-7/4-k/2}.
 \label{eq:refined-product-bounds}
\end{align}
In particular, this mixed profile is time integrable in \(L^1\) and does not
produce the endpoint logarithm.

\subsection{The remainder and its antiderivative}

By \eqref{eq:physical-remainder} and \eqref{eq:background-pair},
\(R_i=u_i-u_i^{\mathrm{app}}\).  The exact nonlinear remainder is
\begin{align}
 N(R)=&(R_2-R_1)^2\int_0^1(1-\zeta)
 F''\left(u_2^{\mathrm{app}}-u_1^{\mathrm{app}}+\zeta(R_2-R_1)\right)d\zeta
 \notag\\
 &+(R_2-R_1)(R_{1,x}+R_{2,x})\int_0^1
 B'\left(u_2^{\mathrm{app}}-u_1^{\mathrm{app}}+\zeta(R_2-R_1)\right)d\zeta
 \notag\\
 &+(R_2-R_1)^2(u_{1,x}^{\mathrm{app}}+u_{2,x}^{\mathrm{app}})\int_0^1(1-\zeta)
 B''\left(u_2^{\mathrm{app}}-u_1^{\mathrm{app}}+\zeta(R_2-R_1)\right)d\zeta.
 \label{eq:exact-N}
\end{align}
{Set}
\begin{equation}\label{eq:def-eta}
\eta:=F'(u_2^{\mathrm{app}}-u_1^{\mathrm{app}})
 +B'(u_2^{\mathrm{app}}-u_1^{\mathrm{app}})(u_{1,x}^{\mathrm{app}}+u_{2,x}^{\mathrm{app}})
 -(\theta_1-\theta_2).
\end{equation}
{From \eqref{eq:background-pair},} one has
\begin{align}
 u_2^{\mathrm{app}}-u_1^{\mathrm{app}}
 =&\theta_2+\Xi_2-\theta_1-\Xi_1
 -\frac{d}{2c}\partial_x(\theta_1+\Xi_1+\theta_2+\Xi_2),\notag\\
 u_{1,x}^{\mathrm{app}}+u_{2,x}^{\mathrm{app}}
 =&\partial_x(\theta_1+\Xi_1+\theta_2+\Xi_2)
 +\frac{d}{2c}\partial_x^2(\theta_2+\Xi_2-\theta_1-\Xi_1).
 \notag
\end{align}
Using $F'(z)=-z+[F'(z)+z]$ and
$B'(z)=-\nu c/p_2+[B'(z)+\nu c/p_2]$ therefore gives the exact identity
\begin{align}
 \eta=&\Xi_1-\Xi_2
 +\left(\frac{d}{2c}-\frac{\nu c}{p_2}\right)
   \partial_x(\theta_1+\Xi_1+\theta_2+\Xi_2)\notag\\
 &-\frac{\nu d}{2p_2}
   \partial_x^2(\theta_2+\Xi_2-\theta_1-\Xi_1)\notag\\
 &+F'(u_2^{\mathrm{app}}-u_1^{\mathrm{app}})+(u_2^{\mathrm{app}}-u_1^{\mathrm{app}})\notag\\
 &+\left[B'(u_2^{\mathrm{app}}-u_1^{\mathrm{app}})+\frac{\nu c}{p_2}\right]
   (u_{1,x}^{\mathrm{app}}+u_{2,x}^{\mathrm{app}}).
 \label{eq:eta-explicit-profile}
\end{align}
Likewise, \eqref{eq:B-exact-expansion} gives
\begin{align}
 B(u_2^{\mathrm{app}}-u_1^{\mathrm{app}})
 =&\frac{\nu c}{p_2}(\theta_1-\theta_2)
 -\frac{\nu c}{p_2}(\Xi_2-\Xi_1)
 +\frac{\nu d}{2p_2}\partial_x(\theta_1+\Xi_1+\theta_2+\Xi_2)+B_2(u_2^{\mathrm{app}}-u_1^{\mathrm{app}}).
 \label{eq:B-explicit-profile}
\end{align}
Taylor's theorem and \eqref{eq:exact-N} give the exact linear and nonlinear
flux decomposition.  Therefore \eqref{eq:exact-characteristic} and
\eqref{eq:primitive-background-residual} give, before spatial integration,
\begin{align}
 R_{i,t}+\lambda_iR_{i,x}-d(R_{1,xx}+R_{2,xx})
 =\partial_x\Big\{&[\theta_1-\theta_2+\eta](R_2-R_1)\notag\\
 &+B(u_2^{\mathrm{app}}-u_1^{\mathrm{app}})(R_{1,x}+R_{2,x})
 +N(R)-\mathcal E_i^{\mathrm{app}}\Big\}.
 \label{eq:exact-R-equation}
\end{align}
The mass of $u_i$ is conserved by \eqref{eq:exact-characteristic}.  Moreover,
$\int\theta_i=M_i$, $\int\Xi_i=0$, and the derivative in
\eqref{eq:background-pair} has zero integral.  Hence, we have
\begin{equation}
 \int_{\mathbb R}R_i(x,t)\,dx
 =\int_{\mathbb R}u_i(x,t)\,dx-\int_{\mathbb R}u_i^{\mathrm{app}}(x,t)\,dx
 =M_i-M_i=0.
 \label{eq:remainder-zero-mass}
\end{equation}
The weighted persistence in Lemma~\ref{lem:local-continuation} implies
\(\langle x\rangle R_i(\cdot,t)\in L^2\) on every compact local-existence
interval.  Since \(\langle x\rangle^{-1}\in L^2(\mathbb R)\),
Cauchy--Schwarz gives \(R_i(\cdot,t)\in L^1\).  Together with
\eqref{eq:remainder-zero-mass}, this shows that the primitive below vanishes at
both spatial infinities.
We may therefore set
\begin{equation}\label{eq:def-W}
 W_i(x,t):=\int_{-\infty}^xR_i(y,t)\,dy,
 \qquad W_{i,x}=R_i.
\end{equation}
Integrating \eqref{eq:exact-R-equation} once in \(x\), setting \(j=3-i\) and $i=1,2$,
and moving the same-family term to the left gives
\begin{align*}
 (\partial_t+\lambda_i\partial_x-d\partial_x^2)W_i+\theta_iW_{i,x}
 =&dW_{j,xx}+(\theta_i-\theta_j)W_{j,x}+\theta_jW_{i,x}+N(R)-\mathcal E_i^{\mathrm{app}}\\
 &+\eta(W_{2,x}-W_{1,x})
 +B(u_2^{\mathrm{app}}-u_1^{\mathrm{app}})(W_{1,xx}+W_{2,xx}).
\end{align*}

\subsection{Cole--Hopf conjugation}\label{Cole--Hopf conjugation}

For every smooth \(f\), the identities \eqref{eq:phi-identities} and the product
rule give the single exact calculation
\begin{align}
 (\partial_t+\lambda_i\partial_x-d\partial_x^2)(\phi_if)
 &=\phi_i(\partial_t+\lambda_i\partial_x-d\partial_x^2)f
 +f(\partial_t+\lambda_i\partial_x-d\partial_x^2)\phi_i
 -2d\phi_{i,x}f_x
 \notag\\
 &=\phi_i\left[(\partial_t+\lambda_i\partial_x-d\partial_x^2)f
 +\theta_if_x\right].
 \label{eq:conjugation-identity}
\end{align}
{Set}
\begin{equation*}
 Z_i:=\phi_iW_i.
\end{equation*}
Then \eqref{eq:conjugation-identity} absorbs the same-family term.

The recovery formulas are
\begin{align}
 R_i&=\phi_i^{-1}\left(Z_{i,x}+\frac{\theta_i}{2d}Z_i\right),
 \label{eq:recover-R}\\
 R_{i,x}&=\phi_i^{-1}\left[
 Z_{i,xx}+\frac{\theta_i}{d}Z_{i,x}
 +\left(\frac{\theta_{i,x}}{2d}+\frac{\theta_i^2}{4d^2}\right)Z_i
 \right].
 \label{eq:recover-Rx}
\end{align}
The variable-viscosity contribution also has the exact divergence form
\begin{align}
 &B(u_2^{\mathrm{app}}-u_1^{\mathrm{app}})\phi_i(R_{1,x}+R_{2,x})
 \notag\\
 &=\partial_x\left[B(u_2^{\mathrm{app}}-u_1^{\mathrm{app}})\phi_i(R_1+R_2)\right]
 \notag\\
 &\quad\;\;-\left[B'(u_2^{\mathrm{app}}-u_1^{\mathrm{app}})(u_{2,x}^{\mathrm{app}}-u_{1,x}^{\mathrm{app}})\phi_i
 +B(u_2^{\mathrm{app}}-u_1^{\mathrm{app}})\phi_{i,x}\right](R_1+R_2).
 \label{eq:variable-viscosity-divergence}
\end{align}
Let \(j=3-i\).  Substitution of \eqref{eq:recover-R}--\eqref{eq:recover-Rx}
gives
\begin{align*}
 d\phi_iW_{j,xx}+(\theta_i-\theta_j)\phi_iW_{j,x}
 =&d\frac{\phi_i}{\phi_j}Z_{j,xx}
 +\frac{\phi_i}{\phi_j}\theta_iZ_{j,x}+\frac{\phi_i}{\phi_j}
 \left(\frac12\theta_{j,x}-\frac{\theta_j^2}{4d}
 +\frac{\theta_i\theta_j}{2d}\right)Z_j,
\end{align*}
while
\(\phi_i\theta_jW_{i,x}
=\theta_jZ_{i,x}+\theta_i\theta_jZ_i/(2d)\).
Hence the complete conjugated system is
\begin{align}
 &Z_{i,t}+\lambda_iZ_{i,x}-dZ_{i,xx}
 -d\frac{\phi_i}{\phi_j}Z_{j,xx} \notag\\
 =&\theta_jZ_{i,x}+\frac{\phi_i}{\phi_j}\theta_iZ_{j,x}
 +\frac{\theta_i\theta_j}{2d}Z_i+\frac{\phi_i}{\phi_j}
 \left(\frac12\theta_{j,x}-\frac{\theta_j^2}{4d}
 +\frac{\theta_i\theta_j}{2d}\right)Z_j\notag\\
 &+\phi_i\eta(R_2-R_1)
 +\phi_iB(u_2^{\mathrm{app}}-u_1^{\mathrm{app}})(R_{1,x}+R_{2,x})+\phi_iN(R)-\phi_i\mathcal E_i^{\mathrm{app}},
 \qquad i=1,2.
 \label{eq:complete-Z-system}
\end{align}

The diffusion matrix already present on the left-hand side is
\[
 \begin{pmatrix}
  1&\dfrac{\phi_1}{\phi_2}\\[2mm]
  \dfrac{\phi_2}{\phi_1}&1
 \end{pmatrix}.
\]
Since
\((\phi_1/\phi_2)(\phi_2/\phi_1)=1\), its determinant is zero: this is
the rank-one physical viscosity matrix, not a uniformly parabolic
\(2\times2\) diffusion matrix.

\section{Green function and the conjugated commutator}
\label{sec:green}

\subsection{The constant-coefficient kernel}

Let \(G^0(y,\tau;x,t)\) be the backward Green kernel of the constant-state
linearization:
\begin{equation}\label{eq:G0-terminal}
 G^0_\tau+G^0_y
 \begin{pmatrix}c&0\\0&-c\end{pmatrix}
 +d\partial_y^2\left[
 G^0\begin{pmatrix}1&1\\1&1\end{pmatrix}\right]=0,
 \qquad
 G^0(y,t;x,t)=\delta(x-y)I_2.
\end{equation}
With $z=y-x$ and the convention
\(\widehat f(\xi)=\int_{\mathbb R}e^{-i\xi z}f(z)\,dz\), one has
$\partial_y\mapsto i\xi$ and $\partial_y^2\mapsto-\xi^2$.  Hence
\begin{equation}
 \partial_{(t-\tau)}\widehat G^0
 =\widehat G^0
 \begin{pmatrix}
 ic\xi-d\xi^2&-d\xi^2\\
 -d\xi^2&-ic\xi-d\xi^2
 \end{pmatrix},
 \qquad \widehat G^0(\xi,0)=I_2.
 \label{eq:G0-Fourier-ODE}
\end{equation}
Solving this matrix ODE gives
\begin{equation}\label{eq:G0-Fourier}
 \widehat G^0(\xi,t-\tau)
 =\exp\left[(t-\tau)\left\{
 i\xi\begin{pmatrix}c&0\\0&-c\end{pmatrix}
 -d\xi^2\begin{pmatrix}1&1\\1&1\end{pmatrix}
 \right\}\right].
\end{equation}
Denote the matrix in \eqref{eq:G0-Fourier-ODE} by $L(\xi)$.  The exact inverse
Fourier representation is
\begin{equation*}
 G^0(y,\tau;x,t)
 =\mathcal F_{\xi\to y-x}^{-1}
 \left[e^{(t-\tau)L(\xi)}\right].
\end{equation*}
The inverse transform is taken in \(\mathcal S'(\mathbb R)\), since the exact
kernel contains a damped Dirac mass.  The eigenvalues are
\(\mu_\pm=-d\xi^2\pm\sqrt{d^2\xi^4-c^2\xi^2}\).  Direct evaluation of the
matrix exponential gives the analytic identity
\begin{align}
 \widehat G^0(\xi,t-\tau)
 =e^{-d\xi^2(t-\tau)}\Bigg[&
 \cosh\left((t-\tau)\sqrt{d^2\xi^4-c^2\xi^2}\right)I_2
 \notag\\
 &+\frac{\sinh\left((t-\tau)\sqrt{d^2\xi^4-c^2\xi^2}\right)}
 {\sqrt{d^2\xi^4-c^2\xi^2}}
 \begin{pmatrix}ic\xi&-d\xi^2\\-d\xi^2&-ic\xi\end{pmatrix}
 \Bigg].
 \label{eq:green-analytic-multiplier}
\end{align}
At the three points where
\(d^2\xi^4-c^2\xi^2=0\), the quotient in
\eqref{eq:green-analytic-multiplier} is understood by analytic continuation
and is assigned the exact value \(t-\tau\).  Thus the analytic formula remains
regular at \(\xi=0,\pm c/d\).

For \(i=1,2\), let
\begin{equation*}
 g_i(y,\tau;x,t)
 =\frac{1}{\sqrt{4\pi d(t-\tau)}}
 \exp\left[-\frac{(x-y-\lambda_i(t-\tau))^2}{4d(t-\tau)}\right].
\end{equation*}
The identities used below are
\begin{align}
 &(\partial_\tau+c\partial_y+d\partial_y^2)g_1=0,
 \qquad
 (\partial_\tau-c\partial_y+d\partial_y^2)g_2=0,\notag\\
 &(\partial_\tau-c\partial_y+d\partial_y^2)g_1=-2c\,g_{1,y},
 \qquad
 (\partial_\tau+c\partial_y+d\partial_y^2)g_2=2c\,g_{2,y},
 \notag\\
 &(\partial_\tau+\lambda_i\partial_y+d\partial_y^2)
 \bigl[(t-\tau)\partial_y^kg_i\bigr]
 =-\partial_y^kg_i,
 \qquad i=1,2,\quad k\geq0.
 \notag
\end{align}
Fix an even $\chi\in C_c^\infty(\mathbb R)$ such that $\chi=1$ on
$|\xi|\leq R_\chi$, $\chi=0$ on $|\xi|\geq2R_\chi$, and
$R_\chi>2c/d$.  For $t-\tau\geq1$, let $\mathcal G$ denote the explicit
Gaussian matrix
\begin{align}
 \mathcal G=
 \begin{pmatrix}
 g_1+\dfrac{d^2}{2c}(t-\tau)g_{1,yyy}
 &\dfrac{d}{2c}(g_{1,y}-g_{2,y})\\[2mm]
 \dfrac{d}{2c}(g_{1,y}-g_{2,y})
 &g_2-\dfrac{d^2}{2c}(t-\tau)g_{2,yyy}
 \end{pmatrix}.
 \label{eq:explicit-gaussian-matrix}
\end{align}
The coefficients in \eqref{eq:explicit-gaussian-matrix} are fixed by the
low-frequency branches and their spectral projectors.  Near \(\xi=0\), we have
\begin{align*}
 \mu_1&=ic\xi-d\xi^2-\frac{id^2}{2c}\xi^3+O(\xi^5),
 &P_1&=\begin{pmatrix}1&0\\0&0\end{pmatrix}
 +\frac{id\xi}{2c}\begin{pmatrix}0&1\\1&0\end{pmatrix}+O(\xi^2),\\
 \mu_2&=-ic\xi-d\xi^2+\frac{id^2}{2c}\xi^3+O(\xi^5),
 &P_2&=\begin{pmatrix}0&0\\0&1\end{pmatrix}
 -\frac{id\xi}{2c}\begin{pmatrix}0&1\\1&0\end{pmatrix}+O(\xi^2).
\end{align*}
Thus the diagonal heat kernels, the first off-diagonal projector correction,
and the cubic phase corrections are precisely the three parts displayed in
\(\mathcal G\).  More explicitly, for an even cutoff \(\chi_0\) supported
near the origin, some \(\omega_0>0\), \(s\geq1\), and
\(0\leq k\leq3\), one may write
\begin{equation*}
 (i\xi)^k\chi_0(\xi)
 \bigl(e^{sL(\xi)}-\widehat{\mathcal G}(\xi,s)\bigr)
 =\sum_{\ell=1}^2e^{i\lambda_\ell s\xi}m_{k\ell}(\xi,s),
\end{equation*}
where, for \(0\leq q\leq2\),
\begin{equation*}
 |\partial_\xi^qm_{k\ell}(\xi,s)|
 \leq Cs^{-(k+2-q)/2}(1+|\sqrt{s}\xi|)^{k+8}
 e^{-\omega_0s\xi^2}.
\end{equation*}
The \(L^1_\xi\) bound and two integrations by parts give the two-cone
pointwise remainder below; on the fixed intermediate annulus, the spectral
gap and \eqref{eq:green-analytic-multiplier} give the exponentially localized
part.
The regular and singular pieces are defined in Fourier variables by
\begin{align}
 \widehat G^r(\xi,t-\tau)
 &=\chi(\xi)\widehat G^0(\xi,t-\tau)
 +(1-\chi(\xi))\widehat{\mathcal G}(\xi,t-\tau),
 \label{eq:exact-regular-cutoff}\\
 \widehat G^s(\xi,t-\tau)
 &=(1-\chi(\xi))
 \left[\widehat G^0(\xi,t-\tau)-\widehat{\mathcal G}(\xi,t-\tau)\right].
 \label{eq:exact-singular-cutoff}
\end{align}
Thus \eqref{eq:exact-regular-cutoff}--\eqref{eq:exact-singular-cutoff} give the
exact identity
\begin{equation*}
 G^0=G^r+G^s.
\end{equation*}
The global regular--singular Green estimates completing this normalization
calculation are those of \cite[Theorems~2.6, 2.7 and 2.9]{LiuYu2022} and
\cite[Lemma~2.6, (2.15)--(2.20)]{WangZhang2023}, after the constant
characteristic transformation \eqref{eq:characteristic-variables}.  In the
present variables, one has
\begin{equation}\label{eq:regular-green-expansion}
 G^r=\mathcal G+\mathcal R^r.
\end{equation}
For \(0\le k\le3\) and some fixed
\(0<\kappa_0<c^2/(2d)\),
\begin{align}
 &|\partial_y^k\mathcal R^r(y,\tau;x,t)|
 \le C(t-\tau)^{-(k+3)/2}
 \sum_{i=1}^2\left(1+
 \frac{|x-y-\lambda_i(t-\tau)|^2}{t-\tau}\right)^{-1}
 \notag\\
 &\qquad\qquad\qquad\qquad\quad+Ce^{-\kappa_0(t-\tau)}(1+|x-y|)^{-2},
 \label{eq:Rr-pointwise}\\
 &\|\partial_y^k\mathcal R^r(t-\tau)\|_{1}
 \le C(t-\tau)^{-1-k/2},
 \qquad
 \|\partial_y^k\mathcal R^r(t-\tau)\|_{\infty}
 \le C(t-\tau)^{-3/2-k/2}.
 \label{eq:Rr-norms}
\end{align}
For the same cutoff $\chi$, define
\begin{equation}\label{eq:def-J}
 \widehat J(\xi)=\frac{i(1-\chi(\xi))}{\xi},
 \qquad
 P_-=\frac12\begin{pmatrix}1&-1\\-1&1\end{pmatrix}.
\end{equation}
On \(\operatorname{supp}(1-\chi)\), the slow high-frequency branch satisfies
\begin{equation*}
 \mu_{\mathrm{slow}}(\xi)=-\frac{c^2}{2d}+\rho_\mu(\xi),
 \qquad
 P_{\mathrm{slow}}(\xi)=P_-+\frac{ic}{2d\xi}
 \begin{pmatrix}1&0\\0&-1\end{pmatrix}+\rho_P(\xi),
\end{equation*}
where
\( |\partial_\xi^q\rho_\mu(\xi)|+|\partial_\xi^q\rho_P(\xi)|
\leq C_q|\xi|^{-2-q}\) for \(0\leq q\leq6\).
Together with \eqref{eq:def-J}, this fixes both the sign and the coefficient
of the singular part:
\begin{equation}\label{eq:singular-green}
 G^s(y,\tau;x,t)
 =e^{-\frac{c^2}{2d}(t-\tau)}
 \left[
 \delta(x-y)P_-+\frac{c}{2d}J(y-x)
 \begin{pmatrix}1&0\\0&-1\end{pmatrix}
 \right]+\mathcal R^s.
\end{equation}
Here \(\mathcal M\) denotes the space of finite matrix-valued Radon measures,
with total-variation norm.  Since
\(\check\chi:=\mathcal F^{-1}\chi\) is Schwartz and
\(J'=\check\chi-\delta_0\), the cutoff localization and the cited Green
estimates give, for \(t-\tau\geq1\),
\begin{align}
 &\int_{\mathbb R}(1+|z|)^4|J(z)|\,dz
 +\int_{\mathbb R}(1+|z|)^4\,d|J'|(z)\leq C,\notag\\
 &\int_{\mathbb R}(1+|z|)^4
 d|\partial_z^k\mathcal R^s|(z,t-\tau)
 \leq Ce^{-\kappa_0(t-\tau)},
 \qquad 0\leq k\leq2.
 \label{eq:singular-remainder-measure-estimate}
\end{align}
From \eqref{eq:singular-green}--
\eqref{eq:singular-remainder-measure-estimate},
\begin{equation}\label{eq:singular-action}
 \|\partial_x^k(G^s(t-\tau)*f)\|_{L^q}
 \le Ce^{-\kappa_0(t-\tau)}\|\partial_x^k f\|_{L^q},
 \qquad 1\le q\le\infty,\quad 0\le k\le2.
\end{equation}
For \(0<t-\tau\le1\), the exact kernel is used as a finite measure and satisfies
\begin{equation}\label{eq:short-time-green-measure}
 \int_{\mathbb R}(1+|y-x|)^4\,d|G^0|(y-x,t-\tau)\le C,
\end{equation}
and
\begin{equation}\label{eq:short-time-green-derivative}
 \int_{\mathbb R}(1+|y-x|)^4\,
 d\left|\partial_y\left[G^0(t-\tau)\binom11\right]\right|(y-x)
 \le C(t-\tau)^{-1/2}.
\end{equation}
By \eqref{eq:short-time-green-measure}--
\eqref{eq:short-time-green-derivative},
\begin{equation*}
 \|G^0(t-\tau)\|_{\mathcal M}\le C,
 \qquad
 \left\|\partial_y\left[G^0(t-\tau)\binom11\right]\right\|_{\mathcal M}
 \le C(t-\tau)^{-1/2},
 \qquad 0<t-\tau\le1.
\end{equation*}
These measure estimates follow from the Gaussian, Dirac, and exponentially
localized terms in \cite[Lemma~2.6, (2.15)--(2.20)]{WangZhang2023}; the
preceding low- and high-frequency expansions record the only coefficients
specific to the present normalization.

\subsection{The exact conjugated commutator}

Conjugate the exact kernel entry by entry:
\begin{equation}\label{eq:conjugated-kernel}
 \widetilde G(y,\tau;x,t)
 :=\begin{pmatrix}
 G^0_{11}&\dfrac{\phi_1}{\phi_2}G^0_{12}\\[2mm]
 \dfrac{\phi_2}{\phi_1}G^0_{21}&G^0_{22}
 \end{pmatrix}(y,\tau;x,t).
\end{equation}
We use \(\widetilde G\) as a cone-adapted approximate Green function for the
conjugated system.  Its familywise weights preserve the emitting and receiving
characteristic directions, while its defect is the exact commutator computed
below.
For \(t-\tau\geq1\), the exact regular--singular split after conjugation is
\begin{equation*}
 \widetilde G
 =
 \begin{pmatrix}
 G^r_{11}&\dfrac{\phi_1}{\phi_2}G^r_{12}\\[2mm]
 \dfrac{\phi_2}{\phi_1}G^r_{21}&G^r_{22}
 \end{pmatrix}
 +
 \begin{pmatrix}
 G^s_{11}&\dfrac{\phi_1}{\phi_2}G^s_{12}\\[2mm]
 \dfrac{\phi_2}{\phi_1}G^s_{21}&G^s_{22}
 \end{pmatrix}.
\end{equation*}
Here and below the weights and their derivatives in a kernel formula are evaluated at
\((y,\tau)\).  The matrix product which enters the diffusion term is
\begin{equation*}
 \widetilde G
 \begin{pmatrix}
 1&\dfrac{\phi_1}{\phi_2}\\[2mm]
 \dfrac{\phi_2}{\phi_1}&1
 \end{pmatrix}
 =\begin{pmatrix}
 G^0_{11}+G^0_{12}&\dfrac{\phi_1}{\phi_2}(G^0_{11}+G^0_{12})\\[2mm]
 \dfrac{\phi_2}{\phi_1}(G^0_{21}+G^0_{22})&G^0_{21}+G^0_{22}
 \end{pmatrix}.
\end{equation*}
Define the \(2\times2\) defect matrix \(E=(E_{ij})_{i,j=1}^2\) by applying
the backward operator corresponding to the left-hand side of
\eqref{eq:complete-Z-system}:
\begin{align}\label{eq:conjugated-backward-operator}
 E(y,\tau;x,t):=
 \widetilde G_\tau+\widetilde G_y
 \begin{pmatrix}c&0\\0&-c\end{pmatrix}
 +d\partial_y^2 \left[
 \widetilde G
 \begin{pmatrix}
 1&\dfrac{\phi_1}{\phi_2}\\[2mm]
 \dfrac{\phi_2}{\phi_1}&1
 \end{pmatrix}\right].
\end{align}
Since \(G^0\) solves \eqref{eq:G0-terminal} exactly, the diagonal entries vanish and
the two off-diagonal entries are
\begin{align}
 \begin{aligned}
 E_{12}=&
 \left(\partial_\tau-c\partial_y\right) \left(\frac{\phi_1}{\phi_2}\right)G^0_{12}
 +2d\partial_y\!\left(\frac{\phi_1}{\phi_2}\right)
 \partial_y(G^0_{11}+G^0_{12})+d\partial_y^2 \left(\frac{\phi_1}{\phi_2}\right)
 (G^0_{11}+G^0_{12}),
 \end{aligned}
 \label{eq:exact-commutator-12}\\
 \begin{aligned}
 E_{21}=&
 \left(\partial_\tau+c\partial_y\right)\left(\frac{\phi_2}{\phi_1}\right)G^0_{21}
 +2d\partial_y\!\left(\frac{\phi_2}{\phi_1}\right)
 \partial_y(G^0_{21}+G^0_{22})+d\partial_y^2\left(\frac{\phi_2}{\phi_1}\right)
 (G^0_{21}+G^0_{22}).
 \end{aligned}
 \notag
\end{align}
Thus, explicitly,
\begin{equation*}
 E(y,\tau;x,t)=
 \begin{pmatrix}
 0&E_{12}(y,\tau;x,t)\\
 E_{21}(y,\tau;x,t)&0
 \end{pmatrix}.
\end{equation*}
The constant-coefficient terms cancel by \eqref{eq:G0-terminal}; hence the
defect contains only derivatives of the Cole--Hopf ratios.

For reference, the coefficient derivatives are
\begin{align}
 &\partial_y\left(\frac{\phi_1}{\phi_2}\right)
 =-\frac{\theta_1-\theta_2}{2d}\frac{\phi_1}{\phi_2},
 \qquad\qquad
 \partial_y^2\left(\frac{\phi_1}{\phi_2}\right)
 =\frac{\phi_1}{\phi_2}
 \left[-\frac{\theta_{1,y}-\theta_{2,y}}{2d}
 +\frac{(\theta_1-\theta_2)^2}{4d^2}\right],
 \label{eq:ratio12-derivatives}\\
 &(\partial_\tau-c\partial_y)\left(\frac{\phi_1}{\phi_2}\right)
 =\frac{\phi_1}{\phi_2}
 \left[\frac{c}{d}\theta_1-\frac12(\theta_{1,y}-\theta_{2,y})
 +\frac{\theta_1^2-\theta_2^2}{4d}\right],
 \label{eq:ratio12-material}\\
 &\partial_y\left(\frac{\phi_2}{\phi_1}\right)
 =\frac{\theta_1-\theta_2}{2d}\frac{\phi_2}{\phi_1},
 \qquad\qquad\;
 \partial_y^2\left(\frac{\phi_2}{\phi_1}\right)
 =\frac{\phi_2}{\phi_1}
 \left[\frac{\theta_{1,y}-\theta_{2,y}}{2d}
 +\frac{(\theta_1-\theta_2)^2}{4d^2}\right],
 \label{eq:ratio21-derivatives}\\
 &(\partial_\tau+c\partial_y)\left(\frac{\phi_2}{\phi_1}\right)
 =\frac{\phi_2}{\phi_1}
 \left[-\frac{c}{d}\theta_2+\frac12(\theta_{1,y}-\theta_{2,y})
 +\frac{\theta_2^2-\theta_1^2}{4d}\right].
 \notag
\end{align}

The regular contribution is defined by
\begin{align}
 E^r_{12}:=&
 (\partial_\tau-c\partial_y)\left(\frac{\phi_1}{\phi_2}\right)G^r_{12}
 +2d\partial_y\left(\frac{\phi_1}{\phi_2}\right)
 \partial_y(G^r_{11}+G^r_{12})
 +d\partial_y^2\left(\frac{\phi_1}{\phi_2}\right)
 (G^r_{11}+G^r_{12}),
 \notag\\
 E^r_{21}:=&
 (\partial_\tau+c\partial_y)\left(\frac{\phi_2}{\phi_1}\right)G^r_{21}
 +2d\partial_y\left(\frac{\phi_2}{\phi_1}\right)
 \partial_y(G^r_{21}+G^r_{22})
 +d\partial_y^2\left(\frac{\phi_2}{\phi_1}\right)
 (G^r_{21}+G^r_{22}).
 \notag
\end{align}
To use the endpoint cancellation one must
insert the complete Gaussian matrix in \eqref{eq:regular-green-expansion}, not
merely its two diagonal entries.  Define its two regular-remainder contributions
by
\begin{align}
 E_{12}^{r,\mathrm{rem}}:=&
 (\partial_\tau-c\partial_y)\left(\frac{\phi_1}{\phi_2}\right)\mathcal R^r_{12}
 +2d\partial_y\left(\frac{\phi_1}{\phi_2}\right)
 \partial_y(\mathcal R^r_{11}+\mathcal R^r_{12})
 +d\partial_y^2\left(\frac{\phi_1}{\phi_2}\right)
 (\mathcal R^r_{11}+\mathcal R^r_{12}),
 \notag\\
 E_{21}^{r,\mathrm{rem}}:=&
 (\partial_\tau+c\partial_y)\left(\frac{\phi_2}{\phi_1}\right)\mathcal R^r_{21}
 +2d\partial_y\left(\frac{\phi_2}{\phi_1}\right)
 \partial_y(\mathcal R^r_{21}+\mathcal R^r_{22})
 +d\partial_y^2\left(\frac{\phi_2}{\phi_1}\right)
 (\mathcal R^r_{21}+\mathcal R^r_{22}).
 \notag
\end{align}

Substitution of \eqref{eq:explicit-gaussian-matrix} and
\eqref{eq:ratio12-derivatives}--\eqref{eq:ratio21-derivatives} into the exact
commutator, followed by collection of the two divergence terms, gives the
identities used in the Duhamel estimates:
\begin{align}
 E^r_{12}=&-\frac12\partial_y\left(\frac{\phi_1}{\phi_2}\theta_1g_1\right)
 -\frac{d}{2c}\partial_y\left[
 \frac{\phi_1}{\phi_2}(\theta_1-\theta_2)(g_{1,y}-g_{2,y})\right]-\frac{d^2}{2c}\frac{\phi_1}{\phi_2}(t-\tau)
 (\theta_1-\theta_2)g_{1,yyyy}
 \notag\\
 &+\frac{\phi_1}{\phi_2}
 \left(\theta_2g_{1,y}+\frac12\theta_{2,y}g_1-\frac12\theta_1g_{2,y}\right)
 +\frac{\phi_1}{\phi_2}\theta_2(\theta_1-\theta_2)
 \left(\frac{g_{1,y}-g_{2,y}}{4c}-\frac{g_1}{4d}\right)
 \notag\\
 &+\frac{\phi_1}{\phi_2}(t-\tau)
 \left[-\frac{d^2}{4c}(\theta_{1,y}-\theta_{2,y})
 +\frac{d}{8c}(\theta_1-\theta_2)^2\right]g_{1,yyy}
 +E_{12}^{r,\mathrm{rem}},
 \label{eq:expanded-commutator-12}\\
 E^r_{21}=&-\frac12\partial_y\left(\frac{\phi_2}{\phi_1}\theta_2g_2\right)
 +\frac{d}{2c}\partial_y\left[
 \frac{\phi_2}{\phi_1}(\theta_1-\theta_2)(g_{1,y}-g_{2,y})\right]-\frac{d^2}{2c}\frac{\phi_2}{\phi_1}(t-\tau)
 (\theta_1-\theta_2)g_{2,yyyy}
 \notag\\
 &+\frac{\phi_2}{\phi_1}
 \left(\theta_1g_{2,y}+\frac12\theta_{1,y}g_2-\frac12\theta_2g_{1,y}\right)
 +\frac{\phi_2}{\phi_1}\theta_1(\theta_1-\theta_2)
 \left(\frac{g_2}{4d}-\frac{g_{1,y}-g_{2,y}}{4c}\right)
 \notag\\
 &-\frac{\phi_2}{\phi_1}(t-\tau)
 \left[\frac{d^2}{4c}(\theta_{1,y}-\theta_{2,y})
 +\frac{d}{8c}(\theta_1-\theta_2)^2\right]g_{2,yyy}
 +E_{21}^{r,\mathrm{rem}}.
 \label{eq:expanded-commutator-21}
\end{align}
The first divergence in each row is retained;
estimating it before integration by parts would lose the endpoint cancellation.

The singular contribution is defined by
\begin{align}
 E^s_{12}:=&
 (\partial_\tau-c\partial_y)\left(\frac{\phi_1}{\phi_2}\right)G^s_{12}
 +2d\partial_y\left(\frac{\phi_1}{\phi_2}\right)
 \partial_y(G^s_{11}+G^s_{12})+d\partial_y^2\left(\frac{\phi_1}{\phi_2}\right)
 (G^s_{11}+G^s_{12}),
 \notag\\
 E^s_{21}:=&
 (\partial_\tau+c\partial_y)\left(\frac{\phi_2}{\phi_1}\right)G^s_{21}
 +2d\partial_y\left(\frac{\phi_2}{\phi_1}\right)
 \partial_y(G^s_{21}+G^s_{22})+d\partial_y^2\left(\frac{\phi_2}{\phi_1}\right)
 (G^s_{21}+G^s_{22}).
 \notag
\end{align}
Since
\[
 P_-\begin{pmatrix}1&1\\1&1\end{pmatrix}=0,
\]
its Dirac part is simply
\begin{align}
 E^s_{12}\big|_\delta
 &=-\frac12e^{-\frac{c^2}{2d}(t-\tau)}
 (\partial_\tau-c\partial_y)\left(\frac{\phi_1}{\phi_2}\right)\delta(x-y),
 \notag\\
 E^s_{21}\big|_\delta
 &=-\frac12e^{-\frac{c^2}{2d}(t-\tau)}
 (\partial_\tau+c\partial_y)\left(\frac{\phi_2}{\phi_1}\right)\delta(x-y).
 \notag
\end{align}
No term proportional to \((\theta_1-\theta_2)\delta_y\) remains.  Finally,
\begin{equation*}
 E_{11}=E_{22}=0,
 \qquad E_{12}=E^r_{12}+E^s_{12},
 \qquad E_{21}=E^r_{21}+E^s_{21}.
\end{equation*}
The superscripts here split the single exact commutator according to
$G^0=G^r+G^s$.  They do not mean that $E^r$ and $E^s$ separately equal the
backward operator applied to the two conjugated kernels: each separate action
also contains a constant-coefficient defect, and those two defects cancel only
after $G^r+G^s=G^0$ is used.

\section{Duhamel formula and pointwise estimates}\label{sec:pointwise}

The cross--diffusion terms are already contained in the principal Green matrix.
Writing \(Z=(Z_1,Z_2)^\top\) and \(R=(R_1,R_2)^\top\), the pointwise argument has
the following structure:
\[
 \begin{gathered}
  Z=\text{initial response}+\text{leading rows}
     +\text{commutator rows}+\text{remaining source rows},\\
  Z_i=a_i(t)\Gamma_i(\cdot,1+t)+Z_i^{\rm rem},\quad
  R_i=\phi_i^{-1}\left(Z_{i,x}+\frac{\theta_i}{2d}Z_i\right),
 \end{gathered}
\]
where $a_i$ is defined in \eqref{eq:ai-definition}. From the identity
\[
 Z_{i,x}=\phi_iR_i-\frac{\theta_i}{2d}Z_i,
\]
we obtain
\begin{gather}
 \theta _2Z_{1,x}+\frac{\phi _1}{\phi _2}\theta _1Z_{2,x}
   +\frac{\theta _1\theta _2}{2d}Z_1
   +\frac{\phi _1}{\phi _2}
       \left(\frac12\theta _{2,x}-\frac{\theta _2^2}{4d}
             +\frac{\theta _1\theta _2}{2d}\right)Z_2 \notag\\*
 =
   \phi _1\theta _2R_1+\phi _1\theta _1R_2
   +\frac{\phi _1}{\phi _2}
       \left(\frac12\theta _{2,x}-\frac{\theta _2^2}{4d}\right)Z_2,
                                                        \label{eq:source-row-one}\\
 \frac{\phi _2}{\phi _1}\theta _2Z_{1,x}+\theta _1Z_{2,x}
   +\frac{\phi _2}{\phi _1}
       \left(\frac12\theta _{1,x}-\frac{\theta _1^2}{4d}
             +\frac{\theta _1\theta _2}{2d}\right)Z_1
   +\frac{\theta _1\theta _2}{2d}Z_2 \notag\\*
 =
   \phi _2\theta _2R_1+\phi _2\theta _1R_2
   +\frac{\phi _2}{\phi _1}
       \left(\frac12\theta _{1,x}-\frac{\theta _1^2}{4d}\right)Z_1.
                                                        \label{eq:source-row-two}
\end{gather}
For the Duhamel variables \((y,\tau)\), set
\begin{equation*}
 F_{\mathrm{lead}}(y,\tau):=
 \binom{\displaystyle
   \phi _1\theta _2R_1+\phi _1\theta _1R_2
   +\frac{\phi _1}{\phi _2}
       \left(\frac12\theta _{2,y}-\frac{\theta _2^2}{4d}\right)Z_2}
 {\displaystyle
   \phi _2\theta _2R_1+\phi _2\theta _1R_2
   +\frac{\phi _2}{\phi _1}
       \left(\frac12\theta _{1,y}-\frac{\theta _1^2}{4d}\right)Z_1}.
\end{equation*}
Thus the terms containing \(\theta_1\theta_2Z_i\) cancel before any estimate is
taken.

The remaining source column is
\begin{align}
 F_{\mathrm{src}}(y,\tau)&:=
 \binom{\displaystyle
 \phi_1\left[\eta(R_2-R_1)
 +B(u_2^{\mathrm{app}}-u_1^{\mathrm{app}})(R_{1,y}+R_{2,y})
 +N(R)-\mathcal E_1^{\mathrm{app}}\right]}
 {\displaystyle
 \phi_2\left[\eta(R_2-R_1)
 +B(u_2^{\mathrm{app}}-u_1^{\mathrm{app}})(R_{1,y}+R_{2,y})
 +N(R)-\mathcal E_2^{\mathrm{app}}\right]},\notag\\
 &:=F_{\mathrm{src}}^{\text{coef}}+F_{\mathrm{src}}^{\text{visc}}+F_{\mathrm{src}}^{\text{nl}}+F_{\mathrm{src}}^{\text{bg}}.
 \label{eq:remaining-source-column}
\end{align}
All entries in \eqref{eq:remaining-source-column} are evaluated at
\((y,\tau)\).  The Green pairing satisfies
\begin{align}
 &\frac{d}{d\tau}\int_{\mathbb R}\widetilde G(y,\tau;x,t)Z(y,\tau)\,dy \label{eq:duhamel-pairing-derivative}\\
 =&\int_{\mathbb R}E(y,\tau;x,t)Z(y,\tau)\,dy+\int_{\mathbb R}\widetilde G(y,\tau;x,t)
 F_{\mathrm{lead}}(y,\tau)\,dy+\int_{\mathbb R}\widetilde G(y,\tau;x,t)
 F_{\mathrm{src}}(y,\tau)\,dy.\notag
\end{align}
Since the terminal pairing is \(Z(x,t)\), integration in \(\tau\) gives
\begin{align}
 Z(x,t)
 =&\int_{\mathbb R}\widetilde G(y,0;x,t)Z(y,0)\,dy  
  +\int_0^t\int_{\mathbb R}\widetilde G(y,\tau;x,t)
 F_{\mathrm{lead}}(y,\tau)\,dy\,d\tau                             \notag\\
 &+\int_0^t\int_{\mathbb R}E(y,\tau;x,t)Z(y,\tau)\,dy\,d\tau+\int_0^t\int_{\mathbb R}\widetilde G(y,\tau;x,t)
 F_{\mathrm{src}}(y,\tau)\,dy\,d\tau .
  \label{eq:complete-duhamel}
\end{align}
The time interval is split as
\[
 [0,t]=[0,(t-1)_+]\cup[(t-1)_+,t].
\]
On the first interval we use \(G^r+G^s\); on the last time unit we keep the
source undecomposed and use \(G^0\). By
\eqref{eq:short-time-green-measure}--\eqref{eq:short-time-green-derivative}, if
\(0<t-\tau\leq1\), the exact constant-state kernel satisfies
\begin{align}
 \left\|\partial_x^kG^0(t-\tau)*\binom{f_1}{f_2}\right\|_\infty
 &\leq C\left\|\binom{f_1}{f_2}\right\|_{W^{k,\infty}},
 &&0\leq k\leq2,\notag\\
 \left\|G^0(t-\tau)*\binom f f\right\|_\infty
 &\leq C\|f\|_\infty,\notag\\
 \left\|\partial_x^kG^0(t-\tau)*\binom f f\right\|_\infty
 &\leq C(t-\tau)^{-1/2}\|\partial_x^{k-1}f\|_\infty,
 &&k=1,2,3,\notag\\
 \left\|\partial_x^kG^0(t-\tau)*\binom{\partial_xf}{\partial_xf}\right\|_\infty
 &\leq C(t-\tau)^{-1/2}\|\partial_x^kf\|_\infty,
 &&0\leq k\leq2.
\label{eq:last-unit-general}
\end{align}
By \eqref{eq:conjugated-kernel}, at a fixed source time the conjugated action is
exactly
\begin{align}
 \int_{\mathbb R}(\widetilde G_{11}f_1+\widetilde G_{12}f_2)\,dy
 &=G^0_{11}*f_1+G^0_{12}*\left(\frac{\phi_1}{\phi_2}f_2\right),\notag\\
 \int_{\mathbb R}(\widetilde G_{21}f_1+\widetilde G_{22}f_2)\,dy
 &=G^0_{21}*\left(\frac{\phi_2}{\phi_1}f_1\right)+G^0_{22}*f_2.
\notag
\end{align}
By \eqref{eq:ratio12-derivatives} and \eqref{eq:ratio21-derivatives}, the
product formulas required for \(0\leq k\leq2\) are
\begin{align}
 \partial_y\left(\frac{\phi_1}{\phi_2}f_2\right)
 &=\frac{\phi_1}{\phi_2}f_{2,y}
 +\partial_y\left(\frac{\phi_1}{\phi_2}\right)f_2,\notag\\
 \partial_y^2\left(\frac{\phi_1}{\phi_2}f_2\right)
 &=\frac{\phi_1}{\phi_2}f_{2,yy}
 +2\partial_y\left(\frac{\phi_1}{\phi_2}\right)f_{2,y}
 +\partial_y^2\left(\frac{\phi_1}{\phi_2}\right)f_2,\notag\\
 \partial_y\left(\frac{\phi_2}{\phi_1}f_1\right)
 &=\frac{\phi_2}{\phi_1}f_{1,y}
 +\partial_y\left(\frac{\phi_2}{\phi_1}\right)f_1,\notag\\
 \partial_y^2\left(\frac{\phi_2}{\phi_1}f_1\right)
 &=\frac{\phi_2}{\phi_1}f_{1,yy}
 +2\partial_y\left(\frac{\phi_2}{\phi_1}\right)f_{1,y}
 +\partial_y^2\left(\frac{\phi_2}{\phi_1}\right)f_1.
 \notag
\end{align}
For a common source column \((\phi_1f,\phi_2f)^\top\), the two row
actions reduce instead to
\begin{align}
 \int_{\mathbb R}(\widetilde G_{11}\phi_1f+\widetilde G_{12}\phi_2f)\,dy
 &=(G^0_{11}+G^0_{12})*(\phi_1f),\notag\\
 \int_{\mathbb R}(\widetilde G_{21}\phi_1f+\widetilde G_{22}\phi_2f)\,dy
 &=(G^0_{21}+G^0_{22})*(\phi_2f).
                                                        \label{eq:last-unit-common}
\end{align}
By \eqref{eq:phi-identities},
\begin{align}
 (\phi_if)_y
 &=\phi_i\left(f_y-\frac{\theta_i}{2d}f\right),\qquad
 (\phi_if)_{yy}
 =\phi_i\left[f_{yy}-\frac{\theta_i}{d}f_y
 +\left(\frac{\theta_i^2}{4d^2}-\frac{\theta_{i,y}}{2d}\right)f\right].
                                                        \notag
\end{align}
Here
\[
 \int_{(t-1)_+}^{t}(t-\tau)^{-1/2}\,d\tau\leq2,
 \qquad P_-(1,1)^\top=0.
\]
Together with \eqref{eq:app-last-gaussian-weight}--
\eqref{eq:app-last-algebraic-weight}, these identities control the last time
unit.

\subsection{The endpoint convolution estimates}

The five-region calculations are given in
Appendix~\ref{app:five-regions}. The unit-mass Gaussian \(\Gamma_i\) is defined
in \eqref{eq:receiving-gaussian}. Fix
\[
 \chi_{\mathrm{time}}\in C^\infty([0,\infty);[0,1]),\qquad
 \chi_{\mathrm{time}}=1\ \hbox{on }[0,1],\qquad
 \chi_{\mathrm{time}}=0\ \hbox{on }[2,\infty).
\]

\begin{lemma}\label{lem:endpoint-convolution}
Let \(i,\alpha\in\{1,2\}\).  In this lemma and
Lemma~\ref{lem:gaussian-extraction}, \(C_f\geq0\)
denotes the size of the displayed source; it is unrelated to the viscosity
coefficient \(B(\cdot)\).  If
\[
 |f_i(y,\tau)|\leq C_f
 \bigl((y-\lambda_i(1+\tau))^2+1+\tau\bigr)^{-3/2},
\]
then
\begin{equation}
 \left|\partial_x\int_0^{(t-1)_+}g_i(t-\tau)*f_i(\tau)\,d\tau\right|
 \leq C\,C_f\log(2+t)\Psi_i(x,t).                 \label{eq:same-cone-map}
\end{equation}
If \(\alpha\neq i\) and
\[
 |f_\alpha(y,\tau)|
 \leq C_f\log(2+\tau)(1+\tau)^{-1/2}
 \bigl((y-\lambda_\alpha(1+\tau))^2+1+\tau\bigr)^{-3/2},
\]
then
\begin{equation}
 \left|\partial_x\int_0^{(t-1)_+}g_i(t-\tau)*f_\alpha(\tau)\,d\tau\right|
 \leq C\,C_f\log(2+t)\Psi_i(x,t).             \label{eq:separated-cone-map}
\end{equation}
The same conclusion holds if this bound is replaced by
\[
 |f_\alpha(y,\tau)|\leq
 C_f\{\log(2+\tau)(1+\tau)^{-1/2}
       +\log^2(2+\tau)(1+\tau)^{-1/2}\}
 \bigl((y-\lambda_\alpha(1+\tau))^2+1+\tau\bigr)^{-3/2}.
\]
In the enlarged case the second logarithm appears only in the secondary term of
\(\Psi_i\) at the \(\alpha\)-cone.
The cited five-region bounds are \eqref{eq:app-same-early}--
\eqref{eq:app-same-late} and \eqref{eq:app-receiving-tube}--
\eqref{eq:app-source-exterior}.
\end{lemma}

\begin{lemma}\label{lem:gaussian-extraction}
Assume
\begin{equation}
 f_\alpha(y,\tau)=\frac{b_\alpha(\tau)}{1+\tau}
       \Gamma_\alpha(y,1+\tau)+\partial_yQ_\alpha(y,\tau),       \label{eq:mass-primitive-form}
\end{equation}
where
\begin{align} \label{eq:primitive-bounds}
 |b_\alpha(\tau)|&\leq C_f,&
 \|Q_\alpha(\tau)\|_1&\leq C_f(1+\tau)^{-1/2},\notag\\
 \|\partial_y^kQ_\alpha(\tau)\|_\infty
 &\leq C_f(1+\tau)^{-1-k/2},&&0\leq k\leq2.         
\end{align}
For \(\alpha=i\), the Gaussian part of its long-time response is exactly
\[
 \Gamma_i(x,1+t)\int_0^{(t-1)_+}\frac{b_i(\tau)}{1+\tau}\,d\tau.
\]
For \(\alpha\neq i\), subtract
\[
 \Gamma_i(x,1+t)\int_0^{(t-1)_+}
 \chi_{\mathrm{time}}\left(\frac{1+\tau}{\sqrt{1+t}}\right)
 \frac{b_\alpha(\tau)}{1+\tau}\,d\tau.
\]
For \(\alpha\neq i\) and \(0\leq k\leq2\), the corrected response satisfies
\begin{multline*}
 \left\|\partial_x^k\left\{
 \int_0^{(t-1)_+}g_i(t-\tau)*
 \frac{b_\alpha(\tau)}{1+\tau}\Gamma_\alpha(\cdot,1+\tau)\,d\tau
 \right.\right.\\[-1mm]
 \left.\left.\qquad-
 \Gamma_i(\cdot,1+t)\int_0^{(t-1)_+}
 \chi_{\mathrm{time}}\left(\frac{1+\tau}{\sqrt{1+t}}\right)
 \frac{b_\alpha(\tau)}{1+\tau}\,d\tau\right\}\right\|_\infty
 \leq C\,C_f(1+t)^{-(k+1)/2}.
\end{multline*}
The constant in this global estimate uses only
\(\|b_\alpha\|_\infty\leq C_f\); no time derivative of \(b_\alpha\) is
assumed.

For every \(i,\alpha\), the primitive bounds in
\eqref{eq:primitive-bounds} also imply the long-time estimate
\[
 \left\|\partial_x^k\int_0^{(t-1)_+}
 g_i(t-\tau)*\partial_yQ_\alpha(\tau)\,d\tau\right\|_\infty
 \leq CC_f(1+t)^{-(k+1)/2},\qquad 0\leq k\leq2.
\]

For the following cone-resolved conclusion, let \(\alpha\neq i\) and assume
in addition that
\begin{equation*}
 (1+\tau)|b_\alpha'(\tau)|\leq C_f.
\end{equation*}
Suppose first that
\begin{align}
 |Q_\alpha(y,\tau)|&\leq
 C_f\bigl((y-\lambda_\alpha(1+\tau))^2+1+\tau\bigr)^{-1},\notag\\
 |Q_{\alpha,y}(y,\tau)|&\leq
 C_f\bigl((y-\lambda_\alpha(1+\tau))^2+1+\tau\bigr)^{-3/2},\notag\\
 |(\partial_\tau+\lambda_\alpha\partial_y-d\partial_y^2)Q_\alpha(y,\tau)|
 &\leq C_f(1+\tau)^{-1/2}
 \bigl((y-\lambda_\alpha(1+\tau))^2+1+\tau\bigr)^{-3/2}\notag\\[-1mm]
 &\quad+C_f(1+\tau)^{-3}
 \exp\left[-\frac{(y-\lambda_\alpha(1+\tau))^2}{C(1+\tau)}\right]\notag\\[-1mm]
 &\quad+C_f(1+\tau)^{-3}
 \mathbf 1_{\{|y-\lambda_\alpha(1+\tau)|\leq c(1+\tau)/2\}}.
 \label{eq:material-primitive-bounds}
\end{align}
Then
\begin{align}
 &\left|\partial_x\left\{
 \int_0^{(t-1)_+}g_i(t-\tau)*f_\alpha(\tau)\,d\tau
 -\Gamma_i(x,1+t)\int_0^{(t-1)_+}
 \chi_{\mathrm{time}}\left(\frac{1+\tau}{\sqrt{1+t}}\right)
 \frac{b_\alpha(\tau)}{1+\tau}\,d\tau
 \right\}\right|\notag\\[-1mm]
 &\leq C\,C_f\Psi_i(x,t).
 \label{eq:material-extraction-pointwise}
\end{align}

There is a second admissible material hypothesis, used below for the explicit
profile sources.  Keep \eqref{eq:primitive-bounds} and the first two pointwise
bounds above, put
\[
 T=1+\tau,\qquad r_\alpha=y-\lambda_\alpha T,
 \qquad A_\alpha=r_\alpha^2+T,
 \qquad P_\alpha=(\partial_\tau+\lambda_\alpha\partial_y
                  -d\partial_y^2)Q_\alpha,
\]
and assume that \(P_\alpha(\cdot,\tau)\) is absolutely continuous, tends to
zero at both spatial infinities, and satisfies
\begin{align}
 |P_\alpha(y,\tau)|&\leq C_fT^{-1}A_\alpha^{-1},\qquad \|P_\alpha(\tau)\|_1\leq C_fT^{-3/2},\notag\\
 |P_{\alpha,y}(y,\tau)|&\leq C_f\left\{
 T^{-1}A_\alpha^{-3/2}
 +T^{-5/2}\exp\left(-\frac{r_\alpha^2}{CT}\right)\right\}.\label{222}
\end{align}
Under this mesoscopic alternative,
\eqref{eq:material-extraction-pointwise} remains valid.
\end{lemma}

\begin{proof}
The same-family identity and the exact normalized cross-family Gaussian
calculation are \eqref{eq:app-semigroup}--\eqref{eq:app-cross-gaussian}.
The estimates immediately following those formulas prove both Gaussian
assertions, using only \(\|b_\alpha\|_\infty\leq C_f\).

We next prove the global primitive estimate.  With
\(H=1+t\), \(T=1+\tau\), and \(s=H-T\), split at \(T=H/2\).
On the early interval, keep all \(k+1\) derivatives on the heat kernel and use
the \(L^1\)-bound for \(Q_\alpha\):
\[
 CC_fH^{-(k+2)/2}\int_1^{H/2}T^{-1/2}\,dT
 \leq CC_fH^{-(k+1)/2}.
\]
On the late interval, transfer derivatives to \(Q_\alpha\).  For
\(k=0,1\), this gives respectively
\[
 CC_f\int_1^{H/2}H^{-3/2}\,ds\leq CC_fH^{-1/2},
 \qquad
 CC_f\int_1^{H/2}H^{-2}\,ds\leq CC_fH^{-1}.
\]
For \(k=2\), leave one derivative on the kernel and transfer the other two:
\[
 CC_f\int_1^{H/2}s^{-1/2}H^{-2}\,ds
 \leq CC_fH^{-3/2}.
\]
Thus no third spatial derivative of \(Q_\alpha\) is required.

The material-derivative identity underlying
Lemma~\ref{lem:gaussian-extraction} is
\begin{align}
 &(\lambda_\alpha-\lambda_i)\partial_x
    [g_i(t-\tau)*Q_\alpha(\tau)]\notag\\
=&g_i(t-\tau)*
   (\partial_\tau+\lambda_\alpha\partial_y-d\partial_y^2)Q_\alpha(\tau)
   -\frac{d}{d\tau}[g_i(t-\tau)*Q_\alpha(\tau)].       \label{eq:material-identity}
\end{align}
After one further derivative and integration in \(\tau\), it gives
\begin{align}
 &(\lambda_\alpha-\lambda_i)\partial_x
 \int_0^{(t-1)_+}g_i(t-\tau)*\partial_yQ_\alpha(\tau)\,d\tau\notag\\
 =&\int_0^{(t-1)_+}\partial_xg_i(t-\tau)*
 (\partial_\tau+\lambda_\alpha\partial_y-d\partial_y^2)Q_\alpha(\tau)\,d\tau-
 [\partial_xg_i(t-\tau)*Q_\alpha(\tau)]_{\tau=0}^{\tau=(t-1)_+}.  \label{eq:integrated-material-identity}
\end{align}
Consequently, the opposite-speed primitive can be estimated without an endpoint logarithmic loss.
The identity is the material-derivative formula of
\cite[Lemma~A.1]{Koike2023}, written with the present characteristic speeds
and heat-kernel normalization.

For the first cone-resolved hypothesis, the material integral in
\eqref{eq:integrated-material-identity} is controlled by
\eqref{eq:app-cutoff-material-map}, while its two boundary terms are
controlled by \eqref{eq:app-upper-material-boundary} and
\eqref{eq:app-lower-material-boundary}.  Combining these estimates with the
Gaussian calculation above proves \eqref{eq:material-extraction-pointwise}.

It remains to prove the mesoscopic alternative.  The boundary terms in
\eqref{eq:integrated-material-identity} are already controlled by
\eqref{eq:app-upper-material-boundary}--
\eqref{eq:app-lower-material-boundary}, so it remains to estimate the material
integral.  
Split the material integral as
\begin{align*}
    I_{E}+I_{L}:=\left(\int_0^{\frac{t-1}{2}}+\int_{\frac{t-1}{2}}^{(t-1)_+}\right)\partial_xg_i(t-\tau)*
 (\partial_\tau+\lambda_\alpha\partial_y-d\partial_y^2)Q_\alpha(\tau)\,d\tau,
\end{align*}
and set
\[
 H=1+t,\qquad T=1+\tau,\qquad s=H-T,
 \qquad r_i=x-\lambda_iH,\qquad a=\lambda_\alpha-\lambda_i.
\]
It suffices to consider \(H\geq4\).

For \(I_E\), since \(\alpha\ne i\), \(|a|=2c>0\). With
\(z=y-\lambda_\alpha T\) and \(R=r_i-aT\), the heat-kernel bound gives $
 |\partial_xg_i(x-y,s)|\leq Cs^{-1}
   \exp\left[-\frac{(R-z)^2}{M_0s}\right],
$
with a fixed \(M_0>0\). By \eqref{222} and the heat-kernel bound, one has
\begin{align*}
    I_E&\lesssim \int_1^{\frac H2} s^{-1}T^{-1}\int_{\mathbb R}
 e^{-(R-z)^2/(M_0s)}\frac{dz}{z^2+T} dT \notag \\
      &= \int_1^{\frac H2} s^{-1}T^{-1}\left(\int_{\mathbb |z|\le|R|/2}+\int_{\mathbb |z|\ge|R|/2}\right)
 e^{-(R-z)^2/(M_0s)}\frac{dz}{z^2+T} dT \notag \\
 &\lesssim\int_1^{\frac H2} \left(s^{-1}T^{-3/2}e^{-R^2/(M_1s)}
   +s^{-1/2}\frac{T^{-1}}{R^2+T}
      \mathbf 1_{\{|R|\geq K\sqrt s\}} \right)dT,
\end{align*}
where \(K\) and \(M_1>M_0\) are fixed. For $1\le T\le\frac{H}{2}$, one has $\frac H2\leq s=H-T\leq H$. Moreover, when $|R|\geq K\sqrt s$, for a fixed constant \(C_K\geq1\), it holds that $C_K^{-1}(R^2+H)\leq R^2+T\leq R^2+H$. Thus, we have \begin{align*}
 I_E \lesssim H^{-1}\int_1^{H/2}T^{-3/2}
 e^{-(r_i-aT)^2/(C H)}\,dT+H^{-1/2}\int_1^{H/2}
 \frac{T^{-1}}{(r_i-aT)^2+H}\,dT.
\end{align*}
Set \(q=\operatorname{sgn}(a)r_i\).  If \(q\leq0\), then
\(|r_i-aT|\geq |r_i|+|a|T\), and retaining this phase bounds both terms by
\(CC_f(r_i^2+H)^{-1}\).  If \(0<q\leq C_0\sqrt H\), then
\begin{align*}
 H^{-1}\int_1^{H/2}T^{-3/2}\,dT
 &\leq CH^{-1},\\
 H^{-1/2}\int_1^{H/2}\frac{T^{-1}}{(r_i-aT)^2+H}\,dT
 &\leq CH^{-3/2}\log H\leq CH^{-1}.
\end{align*}
It remains to consider \(q>C_0\sqrt H\).  Put \(T_0=q/|a|\) and split at
\(T=T_0/2\), intersecting the two pieces with \([1,H/2]\). For $T\in[1,T_0/2]$, \(|r_i-aT|\geq q/2\), so its two contributions are bounded by
\[
 CH^{-1}e^{-cq^2/H}
 +C\frac{H^{-1/2}\log H}{q^2+H}
 \leq\frac{C}{q^2+H}.
\]
For $T\in[T_0/2,H/2]$, \(T^{-3/2}\leq Cq^{-3/2}\) and
\(T^{-1}\leq Cq^{-1}\), while
\[
 \int_{\mathbb R}e^{-(r_i-aT)^2/(C H)}\,dT\leq C\sqrt H,
 \qquad
 \int_{\mathbb R}\frac{dT}{(r_i-aT)^2+H}\leq C H^{-1/2},
\]
so the two contributions are at most
\[
 CH^{-1/2}q^{-3/2},\qquad C(qH)^{-1}.
\]
If this second piece is nonempty, then \(q\leq |a|H\), and both quantities
are bounded by \(Cq^{-2}\).  If it is empty, the phase estimate on the first
piece applies throughout.  Hence, we have
\[
 I_E\leq \frac{CC_f}{r_i^2+H}.
\]

For \(I_L\), integration by parts and \eqref{222} give
\begin{align*}
    I_L=\int_{\frac H2}^{(H-1)_+} g_i *P_{\alpha,y} dT &\lesssim C_f \int_{\frac H2}^{(H-1)_+} g_i * \left\{
 T^{-1}A_\alpha^{-3/2}
 +T^{-5/2}\exp\left(-\frac{r_\alpha^2}{CT}\right)\right\} dT := I_{L,1}+I_{L,2}.
\end{align*}
Write \(\rho=r_i-aH\), so that \(R=\rho+as\). For $I_{L,1}$, since \(H/2\leq T\leq H-1\) and $ |g_i(x-y,s)|\leq Cs^{-1/2}
   \exp\left[-\frac{(R-z)^2}{M_0s}\right]$, with  fixed constant \(M_1>M_0>0\) and $K$, one has
   \begin{align*}
       I_{L,1}&\lesssim \int_{\frac H2}^{(H-1)_+}s^{-1/2}T^{-1}\int_{\mathbb R}
 e^{-(R-z)^2/(M_0s)}\frac{dz}{(z^2+T)^{3/2}} dT\\
 &=\int_{\frac H2}^{(H-1)_+}s^{-1/2}T^{-1}\left(\int_{\mathbb |z|\le|R|/2}+\int_{\mathbb |z|\ge|R|/2}\right)
 e^{-(R-z)^2/(M_0s)}\frac{dz}{(z^2+T)^{3/2}} dT\\
 &\lesssim\int_{\frac H2}^{(H-1)_+} T^{-2}s^{-1/2}e^{-R^2/(M_1s)}
   +T^{-1}(R^2+T)^{-3/2}
      \mathbf 1_{\{|R|\geq K\sqrt s\}}\\
      &\lesssim H^{-2}\int_1^{H/2}s^{-1/2}
 e^{-(\rho+as)^2/(Cs)}\,ds+H^{-1}\int_1^{H/2}
 \bigl((\rho+as)^2+H\bigr)^{-3/2}\,ds.
   \end{align*}
    The first integral is uniformly bounded: if its center
\(s_0=-\rho/a\) lies in the interval, split at
\(s_0/2\) and \(2s_0\), and on the middle piece use
\(s_0^{-1/2}\int e^{-c(s-s_0)^2/s_0}\,ds\leq C\); the two exterior pieces
are exponentially summable.  The second integral is \(O(H^{-1})\), because
\(\int_{\mathbb R}(u^2+H)^{-3/2}du=2H^{-1}\).  Thus
\(I_{L,1}\leq CC_fH^{-2}\) when \(|r_i|\leq K_0H\), where \(K_0\) is fixed.
Choose \(K_0>2|a|\).  When \(|r_i|>K_0H\), retain
\(|r_i-aT|\geq |r_i|/2\) in the same estimates to get
\(I_{L,1}\leq CC_fr_i^{-2}\).

% For \(H/2\leq T\leq H-1\), integration by parts in space is legitimate by the
% assumptions in \(P_\alpha\), and
% \[
%  \partial_xg_i(s)*P_\alpha(\tau)=g_i(s)*P_{\alpha,y}(\tau).
% \]
% Write \(\rho=r_i-aH\), so that \(R=\rho+as\).  Since
% \(H/2\leq T\leq H-1\), the second spatial estimate bounds the algebraic
% part by
% \begin{align*}
%   I_{L,1}\leq{}&CC_fH^{-2}\int_1^{H/2}s^{-1/2}
%  e^{-(\rho+as)^2/(Cs)}\,ds+CC_fH^{-1}\int_1^{H/2}
%  \bigl((\rho+as)^2+H\bigr)^{-3/2}\,ds.
% \end{align*}
% The first integral is uniformly bounded: if its center
% \(s_0=-\rho/a\) lies in the interval, split at
% \(s_0/2\) and \(2s_0\), and on the middle piece use
% \(s_0^{-1/2}\int e^{-c(s-s_0)^2/s_0}\,ds\leq C\); the two exterior pieces
% are exponentially summable.  The second integral is \(O(H^{-1})\), because
% \(\int_{\mathbb R}(u^2+H)^{-3/2}du=2H^{-1}\).  Thus
% \(I_{L,1}\leq CC_fH^{-2}\) when \(|r_i|\leq K_0H\), where \(K_0\) is fixed.
% Choose \(K_0>2|a|\).  When \(|r_i|>K_0H\), retain
% \(|r_i-aT|\geq |r_i|/2\) in the same estimates to get
% \(I_{L,1}\leq CC_fr_i^{-2}\).

For \(I_{L,2}\), since \(R=\rho+as\),
\begin{align*}
    I_{L,2}&\lesssim \int_{\frac H2}^{(H-1)_+}s^{-1/2}T^{-5/2}\int_{\mathbb R}
 e^{-(R-z)^2/(Cs)}e^{-z^2/(CT)}\,dz dT \\
 &\lesssim \int_{\frac H2}^{(H-1)_+} T^{-2} H^{-\frac12}e^{-R^2/(CH)} dT \lesssim H^{-\frac52} \int_{1}^{\frac H2}e^{-(\rho+as)^2/(CH)} ds.
\end{align*}
We choose a fixed constant $K_0$, which satisfies $K_0>2|a|$. For $|r_i|\le K_0 H$, one has 
\begin{align*}
    I_{L,2}\lesssim C_f H^{-2} \lesssim \frac{C_f}{r_i^2+H}.
\end{align*}
For $|r_i|\ge K_0 H$, then we have $|R|\ge \frac{r_i}{2}$ and we obtain
\begin{align*}
    I_{L,2}\lesssim C_f r_i^{-2} \lesssim \frac{C_f}{r_i^2+H}.
\end{align*}
% Finally, the Gaussian part of \(P_{\alpha,y}\) satisfies
% \[
%  s^{-1/2}T^{-5/2}\int_{\mathbb R}
%  e^{-(R-z)^2/(Cs)}e^{-z^2/(CT)}\,dz
%  \leq CT^{-5/2}\sqrt{\frac{T}{H}}e^{-R^2/(CH)}.
% \]
% Its late-time integral is \(O(C_fH^{-2})\), since the phase in \(s\) has
% width \(O(\sqrt H)\); if \(|r_i|>K_0H\), retaining the phase gives
% \(O(C_fr_i^{-2})\).  

Combining the estimates for \(I_E\), \(I_{L,1}\), and \(I_{L,2}\) yields
\[
 \left|\int_0^{(t-1)_+}\partial_xg_i(t-\tau)*P_\alpha(\tau)\,d\tau\right|
 \leq\frac{CC_f}{r_i^2+H}
 \leq CC_f\Psi_i(x,t).
\]
Together with the boundary estimates in
\eqref{eq:integrated-material-identity}, this proves the asserted
mesoscopic alternative without a time logarithm.
\end{proof}

\begin{lemma}
\label{lem:remaining-source-maps}
Fix \(i\in\{1,2\}\).  The constant \(C_f\geq0\) below is the size of the
particular source to which the assertion is applied.

\emph{(i) Divergence source.} Suppose \(h,h_0\) satisfy
\begin{align*}
 |h(y,\tau)|
 &\leq C_f\log(2+\tau)\Bigg[
 \sum_{\alpha=1}^2
 \frac1{([y-\lambda_\alpha(1+\tau)]^2+1+\tau)^{3/2}}\\[-1mm]
 &\hspace{16mm}+(1+\tau)^{-3/4}\sum_{\alpha=1}^2
 \frac1{([y-\lambda_\alpha(1+\tau)]^2+1+\tau)^{3/4}}
 \\[-1mm]
 &\hspace{16mm}+(1+\tau)^{-3/4}\sum_{\alpha=1}^2
 \frac1{|y-\lambda_\alpha(1+\tau)|^{3/2}+1+\tau}\Bigg],\\
 |h_y(y,\tau)|+|h_0(y,\tau)|
 &\leq C_f\log(2+\tau)\Bigg[
 (1+\tau)^{-1/2}\sum_{\alpha=1}^2
 \frac1{([y-\lambda_\alpha(1+\tau)]^2+1+\tau)^{3/2}}\\[-1mm]
 &\hspace{16mm}+(1+\tau)^{-5/4}\sum_{\alpha=1}^2
 \frac1{([y-\lambda_\alpha(1+\tau)]^2+1+\tau)^{3/4}}
 \\[-1mm]
 &\hspace{16mm}+(1+\tau)^{-5/4}\sum_{\alpha=1}^2
 \frac1{|y-\lambda_\alpha(1+\tau)|^{3/2}+1+\tau}\Bigg].
\end{align*}
Then
\begin{equation}
 \left|\partial_x\int_0^{(t-1)_+}
 g_i(t-\tau)*(\partial_yh-h_0)(\tau)\,d\tau\right|
 \leq C\,C_f\log(2+t)\Psi_i(x,t).                 \label{eq:remaining-divergence-map}
\end{equation}

\emph{(ii) Regular remainder.} Fix an entry
\(\mathcal R^r_{i\beta}\) in the receiving row and write
\(e=e^{(0)}+e^{(1)}\).  Suppose
\begin{align}
 |e^{(0)}(y,\tau)|
 &\leq C_f\bigl((y-\lambda_i(1+\tau))^2+1+\tau\bigr)^{-1},\notag\\
 |\partial_ye^{(0)}(y,\tau)|
 &\leq C_f\bigl((y-\lambda_i(1+\tau))^2+1+\tau\bigr)^{-3/2},\notag\\
 |e^{(1)}(y,\tau)|+|\partial_yh(y,\tau)|+|h_0(y,\tau)|
 &\leq C_f\log(2+\tau)\sum_{\alpha=1}^2
 \bigl((y-\lambda_\alpha(1+\tau))^2+1+\tau\bigr)^{-3/2}.
 \label{eq:remaining-structured-hypotheses}
\end{align}
Then
\begin{align}
 &\left|\partial_x\int_0^{(t-1)_+}\int_{\mathbb R}\Bigl[
 \mathcal R^r_{i\beta}(y,\tau;x,t)e(y,\tau)
 +\partial_y\mathcal R^r_{i\beta}(y,\tau;x,t)h(y,\tau)\right.\notag\\[-1mm]
 &\hspace{40mm}\left.
 +\mathcal R^r_{i\beta}(y,\tau;x,t)h_0(y,\tau)
 \Bigr]\,dy\,d\tau\right|
 \leq C\,C_f\log(2+t)\Psi_i(x,t).
 \label{eq:remaining-structured-map}
\end{align}
The same conclusion holds for a fixed finite sum of entries in the receiving
row.  A term bounded only by
\(C_f((y-\lambda_{3-i}(1+\tau))^2+1+\tau)^{-1}\) is excluded from
\(e^{(0)}\).

\emph{(iii) Two-cone profile.} Suppose \(f\) satisfies
\begin{align}
 |f(y,\tau)|
 &\leq C_f\log(2+\tau){(1+\tau)^{-1/2}}
 \sum_{\alpha=1}^2
 \bigl((y-\lambda_\alpha(1+\tau))^2+1+\tau\bigr)^{{-5/4}},\notag\\
 |\partial_yf(y,\tau)|
 &\leq C_f\log(2+\tau)\Bigg[{(1+\tau)^{-1}}
 \sum_{\alpha=1}^2
 \bigl((y-\lambda_\alpha(1+\tau))^2+1+\tau\bigr)^{{-5/4}}\notag\\[-1mm]
 &\hspace{31mm}+(1+\tau)^{-3/2}
 \sum_{\alpha=1}^2
 \bigl((y-\lambda_\alpha(1+\tau))^2+1+\tau\bigr)^{-3/4}\Bigg].
 \label{eq:eta-quarter-hypotheses}
\end{align}
Then
\begin{equation}
 \left|\partial_x\int_0^{(t-1)_+}
 g_i(t-\tau)*f(\tau)\,d\tau\right|
 \leq C C_f\log(2+t)\Psi_i(x,t).
 \label{eq:eta-quarter-map}
\end{equation}

\emph{(iv) Nonlinear source.} If
\[
 |n(y,\tau)|\leq C_f\log^2(2+\tau)(1+\tau)^{-1}
 \{\Psi_1(y,\tau)+\Psi_2(y,\tau)\},
\]
{and, strictly between the two tubes,}
{
\[
 |n(y,\tau)|\leq C_f\log^2(2+\tau)(1+\tau)^{-1}
 \sum_{\alpha=1}^2
 \bigl((y-\lambda_\alpha(1+\tau))^2+1+\tau\bigr)^{-1},
\]}
then
\begin{equation}
 \left|\partial_x\int_0^{(t-1)_+}
 g_i(t-\tau)*n(\tau)\,d\tau\right|
 \leq C\,C_f\log(2+t)\Psi_i(x,t).                 \label{eq:remaining-nonlinear-map}
\end{equation}

\emph{(v) Nonleading Green entries.} If
\begin{align*}
 |f_{\rm rem}(y,\tau)|\leq C_f\Bigg[&
 \sum_{\alpha=1}^2
 \frac1{([y-\lambda_\alpha(1+\tau)]^2+1+\tau)^{3/2}}+(1+\tau)^{-3/4}\sum_{\alpha=1}^2
 \frac1{([y-\lambda_\alpha(1+\tau)]^2+1+\tau)^{3/4}}
 \\[-1mm]
 &+(1+\tau)^{-3/4}\sum_{\alpha=1}^2
 \frac1{|y-\lambda_\alpha(1+\tau)|^{3/2}+1+\tau}\Bigg],
\end{align*}
then every nonleading Gaussian entry in
\eqref{eq:explicit-gaussian-matrix} and every corresponding entry of
\(\mathcal R^r\) in \eqref{eq:regular-green-expansion}, with the output
derivative shown in the Duhamel formula, maps \(f_{\rm rem}\) into
\begin{equation}
 C\,C_f\log(2+t)\Psi_i(x,t).                       \label{eq:remaining-nonleading-map}
\end{equation}
The five assertions remain valid for a fixed finite sum of sources.
Their five-region outputs are \eqref{eq:app-remaining-five-region-output} and
\eqref{eq:app-nonleading-five-region-output-1}.
\end{lemma}

\subsection{Initial response and extraction of the critical Gaussian}

The initial response is organized by
\[
 \int_{\mathbb R}R_i(x,0)\,dx=0
 \ \Longrightarrow\ 
 W_i(x,0)=\int_{-\infty}^{x}R_i(y,0)\,dy
 \ \Longrightarrow\
 Z_i(x,0)=\phi_i(x,0)W_i(x,0).
\]
More precisely,
\[
 \int_{\mathbb R}R_i(x,0)\,dx=0,
 \qquad
\left\|\int_{-\infty}^{x}R_i(y,0)\,dy\right\|_1
 \leq\int_{\mathbb R}|x|\,|R_i(x,0)|\,dx.
\]
Hence, for \(0\leq k\leq2\),
\begin{align}
 &\left\|\partial_x^k\int_{\mathbb R}\widetilde G(y,0;\cdot,t)Z(y,0)\,dy
 \right\|_\infty
 \leq CE_N(1+t)^{-(k+1)/2},                             \label{eq:initial-global}\\
 &\left|\partial_x\int_{\mathbb R}\sum_{\beta=1}^2
 \widetilde G_{i\beta}(y,0;x,t)Z_\beta(y,0)\,dy\right|
 \notag\\[-1mm]
 &\quad+|\theta_i(x,t)|
 \left|\int_{\mathbb R}\sum_{\beta=1}^2
 \widetilde G_{i\beta}(y,0;x,t)Z_\beta(y,0)\,dy\right|
 \leq CE_N\Psi_i(x,t).                         \label{eq:initial-pointwise}
\end{align}
The input bounds are
\begin{equation}
 \|W_i(\cdot,0)\|_1
 \leq\int_{\mathbb R}|x|\,|R_i(x,0)|\,dx .
 \label{eq:initial-antiderivative-L1}
\end{equation}
\[
 W_{i,x}(\cdot,0)=R_i(\cdot,0),\qquad
 W_{i,xx}(\cdot,0)=R_{i,x}(\cdot,0),
\]
and \eqref{eq:initial-size}, \eqref{eq:initial-antiderivative-L1},
\eqref{eq:phi-explicit-identities} yield
\begin{equation}
 \sum_{i,j=1}^2
 \|\phi_i(\cdot,0)W_j(\cdot,0)\|_{W^{2,1}\cap W^{2,\infty}}
 \leq CE_N .
 \label{eq:initial-weighted-W-bound}
\end{equation}
The two conjugated rows are
\begin{align}
 \int_{\mathbb R}[\widetilde G_{11}Z_1+\widetilde G_{12}Z_2](y,0;x,t)\,dy
 &=G^0_{11}(t)*[\phi_1(0)W_1(0)]
 +G^0_{12}(t)*[\phi_1(0)W_2(0)],\notag\\
 \int_{\mathbb R}[\widetilde G_{21}Z_1+\widetilde G_{22}Z_2](y,0;x,t)\,dy
 &=G^0_{21}(t)*[\phi_2(0)W_1(0)]
 +G^0_{22}(t)*[\phi_2(0)W_2(0)].
 \notag
\end{align}
The regular--singular kernel bounds are
\begin{align}
 \|\partial_x^k(G^r(t)*f)\|_\infty
 &\leq C(1+t)^{-(k+1)/2}\|f\|_1,&&0\leq k\leq2,\notag\\
 \|\partial_x^k(G^s(t)*f)\|_\infty
 &\leq Ce^{-\kappa_0t}\|\partial_x^kf\|_\infty,&&0\leq k\leq2.
 \notag
\end{align}
For the cone-resolved estimate, zero mass gives
\begin{equation*}
 |W_i(x,0)|
 \leq\int_{|y|\geq|x|}|R_i(y,0)|\,dy
 \leq CE_N(1+|x|)^{-2}.
\end{equation*}
The corresponding derivative bounds for
\(\phi_j(\cdot,0)W_i(\cdot,0)\), the diagonal and off-diagonal Gaussian
rows, and \eqref{eq:short-time-green-measure}--
\eqref{eq:short-time-green-derivative} give
\eqref{eq:initial-global}--\eqref{eq:initial-pointwise}.

Choose once and for all $\chi_{\mathrm{cone}}\in C^\infty(\mathbb R;[0,1])$ such that
$\chi_{\mathrm{cone}}(z)=1$ for $|z|\leq c/4$ and $\chi_{\mathrm{cone}}(z)=0$ for
$|z|\geq c/2$, and put
\begin{equation}\label{eq:profile-cutoff}
 \chi_i^{\mathrm{cone}}(y,\tau)=
 \chi_{\mathrm{cone}}\left(\frac{y-\lambda_i(1+\tau)}{1+\tau}\right).
\end{equation}

Let \(j=3-i\). Define
\begin{align}
 a_i(t)=&\int_0^{(t-1)_+}
 \chi_{\mathrm{time}}\left(\frac{1+\tau}{\sqrt{1+t}}\right)
 \int_{\mathbb R}\phi_i\theta_jR_i\,dy\,d\tau
 +\frac32\int_0^{(t-1)_+}\int_{\mathbb R}\phi_i\theta_iR_j\,dy\,d\tau\notag\\*
 &-\int_0^{(t-1)_+}
 \chi_{\mathrm{time}}\left(\frac{1+\tau}{\sqrt{1+t}}\right)
 \int_{\mathbb R}\phi_i\theta_jR_j\,dy\,d\tau\notag\\*
 &+\int_0^{(t-1)_+}\int_{\mathbb R}(-1)^{i+1}
 \left(\frac{cp'''(1)}{3p_2^2}+\frac{3}{16c}+\frac{c}{2p_2}\right)
 \phi_i\theta_i^3\,dy\,d\tau\notag\\*
 &+\int_0^{(t-1)_+}
 \chi_{\mathrm{time}}\left(\frac{1+\tau}{\sqrt{1+t}}\right)
 \int_{\mathbb R}(-1)^i
 \left(\frac{cp'''(1)}{3p_2^2}+\frac1{4c}\right)
 \phi_i\theta_j^3\,dy\,d\tau\notag\\*
 &-\frac12\int_0^{(t-1)_+}\int_{\mathbb R}\phi_i\chi_i^{\mathrm{cone}}\Xi_i^2\,dy\,d\tau
 -\frac12\int_0^{(t-1)_+}
 \chi_{\mathrm{time}}\left(\frac{1+\tau}{\sqrt{1+t}}\right)
 \int_{\mathbb R}\phi_i\chi_j^{\mathrm{cone}}\Xi_j^2\,dy\,d\tau .       \label{eq:ai-definition}
\end{align}
The unlocalized tails
$-\phi_i(1-\chi_\alpha^{\mathrm{cone}})\Xi_\alpha^2/2$ remain in $Z_i^{\mathrm{rem}}$.
Since the upper limit is $(t-1)_+$, $a_i$ is locally Lipschitz, and
suprema involving $a_i'$ are understood as essential suprema. Set
\begin{equation}
 Z_i(x,t)=a_i(t)\Gamma_i(x,1+t)+Z_i^{\mathrm{rem}}(x,t).        \label{eq:Z-decomposition}
\end{equation}
In general \(a_i\neq\int_{\mathbb R}Z_i\,dx\). By
construction, \(a_i\) contains the same-family Gaussian coefficients and the
cutoff opposite-family coefficients extracted from the ordinary,
exceptional, and critical profile sources. Thus the logarithmic mode is
contained in \(a_i\Gamma_i\), whereas \(Z_i^{\mathrm{rem}}\) has no
logarithmic loss.
The complete estimate, including the remaining source and Green-kernel
components, is proved in Lemma~\ref{lem:consolidated-source-ledger}.

The Cole--Hopf identity
\[
 \phi_i\theta_i=2d\bigl(1-e^{-M_i/(2d)}\bigr)\Gamma_i
\]
and \eqref{eq:Z-decomposition} give the exact recovery formula
\begin{equation}
 R_i=\frac{a_i(t)}{2d(1-e^{-M_i/(2d)})}\theta_{i,x}
 +\phi_i^{-1}\left(\partial_x Z_i^{\mathrm{rem}}+\frac{\theta_i}{2d}Z_i^{\mathrm{rem}}\right).  \label{eq:modal-recovery}
\end{equation}
When \(M_i=0\), the quotient is understood by continuous extension through
\[
 \frac{\theta_{i,x}}{2d(1-e^{-M_i/(2d)})}
 =\phi_i^{-1}\left(\Gamma_{i,x}+\frac{\theta_i}{2d}\Gamma_i\right).
\]

For fixed \(t^*>0\) and \(N\geq5\), define
\begin{align}
 X_N(t^*)=\sup_{0\leq t\leq t^*}\max_{i=1,2}\Bigg\{&
 \frac1{\log(2+t)}\sup_{x\in\mathbb R}
       \frac{|R_i(x,t)|}{\Psi_i(x,t)}
 +\frac{|a_i(t)|}{\log(2+t)}+(1+t)|a_i'(t)|\notag\\
 &+\sum_{k=0}^2(1+t)^{(k+1)/2}\|\partial_x^k Z_i^{\mathrm{rem}}(t)\|_\infty
 +\|R(t)\|_{H^N}\Bigg\}.                                  \label{eq:pointwise-continuity-norm}
\end{align}
Below, \(X_N\) means \(X_N(t^*)\). In particular,
\begin{align}
 |a_i(t)|&\leq X_N\log(2+t),& |a_i'(t)|&\leq X_N(1+t)^{-1},\notag\\
 \|\partial_x^k Z_i^{\mathrm{rem}}(t)\|_\infty
 &\leq X_N(1+t)^{-(k+1)/2},&&0\leq k\leq2.              \label{eq:no-prelog-input}
\end{align}
Equations \eqref{eq:Z-decomposition}, \eqref{eq:modal-recovery}, and
\eqref{eq:no-prelog-input} give
\begin{align}
 \|\partial_x^kZ_i(t)\|_\infty
 &\leq CX_N\log(2+t)(1+t)^{-(k+1)/2},&&0\leq k\leq2,\notag\\
 \|\partial_x^kR_i(t)\|_\infty
 &\leq CX_N\log(2+t)(1+t)^{-1-k/2},&&0\leq k\leq1.       \label{eq:derived-low-rates}
\end{align}
Differentiating \(Z_j=\phi_jW_j\) gives
\begin{equation}
 \partial_x^{k+1}Z_j
 =\sum_{a=0}^k\binom{k}{a}(\partial_x^a\phi_j)
 \partial_x^{k-a}R_j
 -\frac1{2d}\sum_{a=0}^k\binom{k}{a}(\partial_x^a\theta_j)
 \partial_x^{k-a}Z_j .
 \label{eq:exact-recovery-derivative}
\end{equation}
Equations \eqref{eq:pointwise-continuity-norm}, \eqref{eq:modal-recovery},
\eqref{eq:no-prelog-input}, and \eqref{eq:burgers-pointwise-bound} imply
\begin{equation}
 \left\|\partial_x^N\left[\phi_j^{-1}
 \left(\partial_x Z_j^{\mathrm{rem}}+\frac{\theta_j}{2d}Z_j^{\mathrm{rem}}\right)\right](t)\right\|_2
 +\left\|\partial_x^{N-1}\left[\phi_j^{-1}
 \left(\partial_x Z_j^{\mathrm{rem}}+\frac{\theta_j}{2d}Z_j^{\mathrm{rem}}\right)\right](t)\right\|_\infty
 \leq CX_N .
 \label{eq:recovered-hat-top}
\end{equation}
Applying the one-dimensional Gagliardo--Nirenberg inequality to
\eqref{eq:derived-low-rates} and \eqref{eq:recovered-hat-top} gives, for
the indicated derivative ranges,
\begin{align}
 \|\partial_x^kR_j(t)\|_\infty
 \leq CX_N\log^{\frac{2N-2k-1}{2N-3}}(2+t)
       (1+t)^{-\frac32\frac{2N-2k-1}{2N-3}},
       &\qquad 2\leq k\leq N-1,\notag\\
 \|\partial_x^kZ_j(t)\|_\infty
 \leq CX_N\log^{\frac{2N-2k+1}{2N-3}}(2+t)
       (1+t)^{-\frac32\frac{2N-2k+1}{2N-3}},
       &\qquad 3\leq k\leq N.                            \label{eq:intermediate-rates}
\end{align}
The explicit modal bounds are
\begin{align*}
 \left\|\partial_x^k\left(
 \frac{a_j}{2d(1-e^{-M_j/(2d)})}\theta_{j,x}\right)\right\|_\infty
 &\leq CX_N\log(2+t)(1+t)^{-(k+2)/2},\\
 \|\partial_x^k(a_j\Gamma_j)\|_\infty
 &\leq CX_N\log(2+t)(1+t)^{-(k+1)/2}.
\end{align*}
For \(N\ge5\), define
\[
 \beta_N=\frac{2N-5}{2N-3},\qquad \frac32\beta_N>1.
\]
Then \eqref{eq:derived-low-rates}--\eqref{eq:intermediate-rates} imply
\begin{equation}
 \|R_{j,xx}(t)\|_\infty+\|Z_{j,xxx}(t)\|_\infty
 \leq CX_N\log^{\beta_N}(2+t)(1+t)^{-\frac32\beta_N}.
 \label{eq:second-third-derivative-rate}
\end{equation}

For \(i\ne j\), the exact modal decompositions are
\begin{align}
 \theta_iR_j
 &=\frac{a_j}{2d(1-e^{-M_j/(2d)})}\theta_i\theta_{j,y}
 +\theta_i\phi_j^{-1}\left(\partial_y Z_j^{\mathrm{rem}}
 +\frac{\theta_j}{2d}Z_j^{\mathrm{rem}}\right),\notag\\
 \theta_i Z_j&=a_j\theta_i\Gamma_j+\theta_i Z_j^{\mathrm{rem}} .
 \label{eq:theta-cross-modal-split}
\end{align}
By \eqref{eq:theta-cross-modal-split}, \eqref{eq:burgers-pointwise-bound},
\eqref{eq:no-prelog-input}, \eqref{eq:derived-low-rates}, and
\eqref{eq:phi-explicit-identities}, the modal terms are exponentially small,
while the terms containing \(Z_j^{\mathrm{rem}}\) give
\begin{align}
 \|\theta_iR_j(\tau)\|_1
 &\leq C\epsilon X_N(1+\tau)^{-1}
 +C\epsilon X_N\log(2+\tau)e^{-\kappa_0(1+\tau)},\notag\\
 \|\theta_iR_j(\tau)\|_\infty
 &\leq C\epsilon X_N(1+\tau)^{-3/2}
 +C\epsilon X_N\log(2+\tau)e^{-\kappa_0(1+\tau)},\notag\\
 \|\partial_y(\theta_iR_j)(\tau)\|_\infty
 &\leq C\epsilon X_N(1+\tau)^{-2}
 +C\epsilon X_N\log(2+\tau)e^{-\kappa_0(1+\tau)},\notag\\
 \|\theta_iZ_j(\tau)\|_1
 &\leq C\epsilon X_N(1+\tau)^{-1/2}
 +C\epsilon X_N\log(2+\tau)e^{-\kappa_0(1+\tau)},\notag\\
 \|\partial_y^k(\theta_iZ_j)(\tau)\|_\infty
 &\leq C\epsilon X_N(1+\tau)^{-1-k/2}
 +C\epsilon X_N\log(2+\tau)e^{-\kappa_0(1+\tau)},\quad0\leq k\leq2.
 \notag
\end{align}
The same split, with \eqref{eq:burgers-pointwise-bound} and
\eqref{eq:no-prelog-input}, gives, for \(0\leq q\leq2\),
\begin{equation*}
 \sum_{k=1}^2|\partial_y^q(\phi_k\theta_iZ_j)(y,\tau)|
 \leq C\epsilon X_N
 \begin{cases}
 ([y-\lambda_i(1+\tau)]^2+1+\tau)^{-1},&q=0,\\
 ([y-\lambda_i(1+\tau)]^2+1+\tau)^{-3/2},&q=1,\\
 (1+\tau)^{-1/2}([y-\lambda_i(1+\tau)]^2+1+\tau)^{-3/2},&q=2.
 \end{cases}
\end{equation*}

\subsection{The exceptional opposite-family feedback}\label{exceptional opposite-family feedback}

The terms requiring a separate endpoint argument are
\(\phi_i\theta_jR_j\), \(j=3-i\), and
\((\frac12\theta_{i,y}-\frac{\theta_i^2}{4d})Z_i\); see
\eqref{eq:expanded-commutator-12}--\eqref{eq:expanded-commutator-21},
\eqref{eq:complete-duhamel}, and
\eqref{eq:first-leading-row}--\eqref{eq:second-leading-row}.
The first is the unique exceptional opposite-family feedback.  Indeed,
\[
 \|\theta_jR_j(\tau)\|_1
 \leq C\epsilon X_N\log(2+\tau)(1+\tau)^{-1}
\]
would give
\[
 (1+t)^{-1}\int_1^{t/2}\frac{\log(2+\tau)}{1+\tau}\,d\tau
 \geq c_0(1+t)^{-1}\log^2(2+t)\qquad(t\geq4),
\]
and hence lose one logarithm.

Substituting \eqref{eq:modal-recovery} for \(R_j\), differentiating
\(\phi_i\), and then adding and subtracting the two spatial masses gives the exact
identity
\begin{align}
 \phi_i\theta_jR_j=&
 \partial_y\left\{
 \frac{a_j\phi_i\theta_j^2}{4d(1-e^{-M_j/(2d)})}\right\} +\left(\int_{\mathbb R}\phi_i\theta_jR_j\,dz\right)
       \Gamma_j(y,1+\tau)                   \notag\\
 &+\frac{a_j}{8d^2(1-e^{-M_j/(2d)})}
 \left\{\phi_i\theta_i\theta_j^2
 -\left(\int_{\mathbb R}\phi_i\theta_i\theta_j^2\,dz\right)
       \Gamma_j(y,1+\tau)\right\}                                          \notag\\
 &+\left\{
 \phi_i\theta_j\phi_j^{-1}
       \left(\partial_y Z_j^{\mathrm{rem}}+\frac{\theta_j}{2d}Z_j^{\mathrm{rem}}\right)-\Bigg[\int_{\mathbb R}\phi_i\theta_j\phi_j^{-1}
       \left(\partial_y Z_j^{\mathrm{rem}}
       +\frac{\theta_j}{2d}Z_j^{\mathrm{rem}}\right)dz\Bigg]
       \Gamma_j(y,1+\tau)\right\}\notag\\
       :=&I_1+I_2+I_3+I_4.   \label{eq:exceptional-decomposition}
\end{align}
Its four components satisfy
\begin{equation*}
 \begin{aligned}
 I_1&=\partial_y(\cdots),
 &&\text{opposite-speed divergence},\\
 I_2&=\left(\int_{\mathbb R}\phi_i\theta_jR_j\,dz\right)
       \Gamma_j(y,1+\tau),
 &&\text{Gaussian mass contained in \eqref{eq:ai-definition}},\\
 \int_{\mathbb R}I_3\,dy&=0,
 &\qquad \int_{\mathbb R}I_4\,dy&=0.
 \end{aligned}
\end{equation*}
Thus the bad feedback is reduced to a divergence, an extracted Gaussian, and
two zero-mass terms, without introducing a coupled Volterra system for the
modal coefficients.

Before applying Lemma~\ref{lem:gaussian-extraction}, we verify its coefficient
hypothesis from the complete equations.  Multiplying
\eqref{eq:modal-recovery} by \(\phi_i\theta_j\), integrating over
\(\mathbb R\), and integrating the term containing
\(\partial_yZ_j^{\mathrm{rem}}\) by parts give the exact identity
\begin{align}
 (1+\tau)\int_{\mathbb R}\phi_i\theta_jR_j\,dy&=\frac{(1+\tau)a_j}{8d^2(1-e^{-M_j/(2d)})}
       \int_{\mathbb R}\phi_i\theta_i\theta_j^2\,dy
  +(1+\tau)\int_{\mathbb R}
 \frac{\phi_i}{\phi_j}
 \left(\frac{\theta_i\theta_j}{2d}-\theta_{j,y}\right)
 Z_j^{\mathrm{rem}}\,dy .                                      \notag
\end{align}
Differentiating this formula and integrating the derivatives of \(Z_j^{\mathrm{rem}}\)
by parts yields
\begin{align}
 &\frac d{d\tau}\left[(1+\tau)
       \int_{\mathbb R}\phi_i\theta_jR_j\,dy\right]\notag\\*
=&\frac{a_j+(1+\tau)a_j'}{8d^2(1-e^{-M_j/(2d)})}
       \int_{\mathbb R}\phi_i\theta_i\theta_j^2\,dy\notag\\*
 &+\frac{(1+\tau)a_j}{8d^2(1-e^{-M_j/(2d)})}
       \frac d{d\tau}\int_{\mathbb R}\phi_i\theta_i\theta_j^2\,dy\notag+\int_{\mathbb R}\frac{\phi_i}{\phi_j}
       \left(\frac{\theta_i\theta_j}{2d}-\theta_{j,y}\right)
       Z_j^{\mathrm{rem}}\,dy\notag\\*
 &+(1+\tau)\int_{\mathbb R}\Bigg\{
 \left(\partial_\tau+\lambda_j\partial_y+d\partial_y^2\right)
 \left[\frac{\phi_i}{\phi_j}
       \left(\frac{\theta_i\theta_j}{2d}-\theta_{j,y}\right)\right]
       Z_j^{\mathrm{rem}}\notag\\*[-1mm]
 &\hspace{35mm}+\frac{\phi_i}{\phi_j}
       \left(\frac{\theta_i\theta_j}{2d}-\theta_{j,y}\right)
       (\partial_\tau+\lambda_j\partial_y-d\partial_y^2)
       Z_j^{\mathrm{rem}}\Bigg\}\,dy .                       \label{eq:exceptional-mass-derivative}
\end{align}
There is no estimate of \(\partial_\tau Z_j^{\mathrm{rem}}\) in this formula.
Instead the last factor is taken from the complete \(j\)-equation.  By the
preceding coefficient identity, the sum of the two zero-mass brackets in
\eqref{eq:exceptional-decomposition} has the primitive
\begin{align}
\int_{-\infty}^y(I_3+I_4)ds=&\frac{\phi_i}{\phi_j}\theta_j Z_j^{\mathrm{rem}}
+\int_{-\infty}^y\Bigg[
 \frac{a_j}{8d^2(1-e^{-M_j/(2d)})}\phi_i\theta_i\theta_j^2
 +\frac{\phi_i}{\phi_j}
 \left(\frac{\theta_i\theta_j}{2d}-\theta_{j,z}\right)Z_j^{\mathrm{rem}}\notag\\
&\qquad\qquad\qquad\qquad\quad -\left(\int_{\mathbb R}\phi_i\theta_jR_j\,dw\right)
 \Gamma_j(z,1+\tau)\Bigg]dz .
\label{111}
\end{align}
Applying \(\partial_\tau+\lambda_j\partial_y-d\partial_y^2\) for \eqref{111} directly gives
\begin{align}
&\left(\partial_\tau+\lambda_j\partial_y-d\partial_y^2\right)
 \left(\frac{\phi_i}{\phi_j}\theta_j Z_j^{\mathrm{rem}}\right)\notag\\
&+\int_{-\infty}^y\Bigg[
 \left(\partial_\tau+\lambda_j\partial_z-d\partial_z^2\right)
 \left(\frac{a_j\phi_i\theta_i\theta_j^2}
 {8d^2(1-e^{-M_j/(2d)})}\right) -\frac d{d\tau}\left(\int_{\mathbb R}\phi_i\theta_jR_j\,dw\right)
 \Gamma_j(z,1+\tau)\notag\\[-1mm]
&\hspace{20mm}
 +\left(\partial_\tau+\lambda_j\partial_z-d\partial_z^2\right)
 \left\{\frac{\phi_i}{\phi_j}
 \left(\frac{\theta_i\theta_j}{2d}-\theta_{j,z}\right)Z_j^{\mathrm{rem}}\right\}\Bigg]dz .
\label{eq:exceptional-primitive-material-exact}
\end{align}
The two products in \eqref{eq:exceptional-primitive-material-exact} expand exactly as
\begin{align}
&\left(\partial_\tau+\lambda_j\partial_y-d\partial_y^2\right)
 \left(\frac{\phi_i}{\phi_j}\theta_j Z_j^{\mathrm{rem}}\right)\notag\\
=&\frac{\phi_i}{\phi_j}\theta_j
 \left(\partial_\tau+\lambda_j\partial_y-d\partial_y^2\right)Z_j^{\mathrm{rem}}
 +\left(\partial_\tau+\lambda_j\partial_y-d\partial_y^2\right)
 \left(\frac{\phi_i}{\phi_j}\theta_j\right)Z_j^{\mathrm{rem}}-2d\,\partial_y\left(\frac{\phi_i}{\phi_j}\theta_j\right)
 \partial_y Z_j^{\mathrm{rem}},\notag\\
&\left(\partial_\tau+\lambda_j\partial_y-d\partial_y^2\right)
 \left\{\frac{\phi_i}{\phi_j}
 \left(\frac{\theta_i\theta_j}{2d}-\theta_{j,y}\right)Z_j^{\mathrm{rem}}\right\}
 \notag\\
=&\frac{\phi_i}{\phi_j}
 \left(\frac{\theta_i\theta_j}{2d}-\theta_{j,y}\right)
 \left(\partial_\tau+\lambda_j\partial_y-d\partial_y^2\right)Z_j^{\mathrm{rem}}-2d\,\partial_y\left\{\frac{\phi_i}{\phi_j}
 \left(\frac{\theta_i\theta_j}{2d}-\theta_{j,y}\right)\right\}
 \partial_y Z_j^{\mathrm{rem}}
 \notag\\[-1mm]
&\qquad+\left(\partial_\tau+\lambda_j\partial_y-d\partial_y^2\right)
 \left\{\frac{\phi_i}{\phi_j}
 \left(\frac{\theta_i\theta_j}{2d}-\theta_{j,y}\right)\right\}
 Z_j^{\mathrm{rem}}.
\label{eq:exceptional-primitive-product-rules}
\end{align}
No third derivative of \(Z_j^{\mathrm{rem}}\) occurs.  Put
\[
 \mathcal L_j:=\partial_\tau+\lambda_j\partial_y-d\partial_y^2.
\]
Together with \eqref{eq:def-phi}, \eqref{eq:explicit-burgers-wave}, and
\eqref{eq:no-prelog-input}, the complete equation reduces the coefficient
estimates to \(\mathcal L_jZ_j^{\mathrm{rem}}\).  Since
\[
 \mathcal L_jZ_j^{\mathrm{rem}}=\mathcal L_jZ_j-a_j'\Gamma_j,
\]
the term \(-a_j'\Gamma_j\) cancels the first occurrence of \(a_j'\) by
\begin{align}
 \int_{\mathbb R}\frac{\phi_i}{\phi_j}
 \left(\frac{\theta_i\theta_j}{2d}-\theta_{j,y}\right)\Gamma_j\,dy
 &=\frac1{8d^2(1-e^{-M_j/(2d)})}
   \int_{\mathbb R}\phi_i\theta_i\theta_j^2\,dy,\notag\\
 \partial_y\left\{\frac{\phi_i\theta_j^2}
 {4d(1-e^{-M_j/(2d)})}\right\}
 &=\frac{\phi_i\theta_i\theta_j^2}
 {8d^2(1-e^{-M_j/(2d)})}
 -\frac{\phi_i}{\phi_j}
 \left(\frac{\theta_i\theta_j}{2d}-\theta_{j,y}\right)\Gamma_j .
\label{eq:exceptional-modal-cancellation}
\end{align}
The cancellation in the mass derivative does not remove the complete
\(a_j'\)-contribution from the material derivative of the zero-mass primitive.
The remaining term is
\begin{equation}
 -\frac{a_j'\phi_i\theta_j^2}
 {4d(1-e^{-M_j/(2d)})}.
 \label{eq:exceptional-aj-prime-remainder}
\end{equation}

The cross-diffusion term
\(d(\phi_j/\phi_i)Z_{i,yy}\) in \(\mathcal L_jZ_j\), cf.\
\eqref{eq:complete-Z-system}, contributes through
\begin{align}
 \frac{\phi_i}{\phi_j}\theta_j
       \left(d\frac{\phi_j}{\phi_i}Z_{i,yy}\right)
 &=d\theta_jZ_{i,yy},\notag\\
 \frac{\phi_i}{\phi_j}
 \left(\frac{\theta_i\theta_j}{2d}-\theta_{j,y}\right)
       \left(d\frac{\phi_j}{\phi_i}Z_{i,yy}\right)
 &=\left(\frac12\theta_i\theta_j-d\theta_{j,y}\right)Z_{i,yy}.
 \label{eq:exceptional-cross-diffusion}
\end{align}
By \eqref{eq:exceptional-cross-diffusion},
\eqref{eq:burgers-pointwise-bound}, and \eqref{eq:no-prelog-input},
\begin{align}
 |d\theta_jZ_{i,yy}|
 &\leq C\epsilon X_N(1+\tau)^{-1/2}
 \bigl((y-\lambda_j(1+\tau))^2+1+\tau\bigr)^{-3/2},\notag\\
 \left|\left(\frac12\theta_i\theta_j-d\theta_{j,y}\right)Z_{i,yy}\right|
 &\leq C\epsilon X_N(1+\tau)^{-1}
 \bigl((y-\lambda_j(1+\tau))^2+1+\tau\bigr)^{-3/2}.        \label{eq:exceptional-cross-diffusion-bound}
\end{align}
The corresponding \(L^1\)-bounds and first centered moments are
\begin{equation*}
 \begin{aligned}
 \|d\theta_jZ_{i,yy}\|_1
 &\leq C\epsilon X_N(1+\tau)^{-3/2},\\
 \int_{\mathbb R}|y-\lambda_j(1+\tau)|\,|d\theta_jZ_{i,yy}|\,dy
 &\leq C\epsilon X_N(1+\tau)^{-1},\\
 \left\|\left(\tfrac12\theta_i\theta_j-d\theta_{j,y}\right)Z_{i,yy}\right\|_1
 &\leq C\epsilon X_N(1+\tau)^{-2},\\
 \int_{\mathbb R}|y-\lambda_j(1+\tau)|
 \left|\left(\tfrac12\theta_i\theta_j-d\theta_{j,y}\right)Z_{i,yy}\right|\,dy
 &\leq C\epsilon X_N(1+\tau)^{-3/2}.
 \end{aligned}
\end{equation*}
Consequently, these estimates show that the terms induced by the cross-diffusion term \(d(\phi_j/\phi_i)Z_{i,yy}\) in \(\mathcal L_j Z_j\) satisfy the material-derivative bound in \eqref{eq:material-primitive-bounds}.

The remaining linear coefficient sources in
\(\mathcal L_jZ_j\), cf.\ \eqref{eq:complete-Z-system}, are, in the first row,
\begin{align}
 &\theta_2Z_{1,y}+\frac{\phi_1}{\phi_2}\theta_1Z_{2,y}
 +\frac{\theta_1\theta_2}{2d}Z_1 
 +\frac{\phi_1}{\phi_2}
 \left(\frac12\theta_{2,y}-\frac{\theta_2^2}{4d}
 +\frac{\theta_1\theta_2}{2d}\right)Z_2,
 \label{eq:exceptional-linear-row-one}
\end{align}
and, in the second row,
\begin{align}
 &\frac{\phi_2}{\phi_1}\theta_2Z_{1,y}+\theta_1Z_{2,y}
 +\frac{\phi_2}{\phi_1}
 \left(\frac12\theta_{1,y}-\frac{\theta_1^2}{4d}
 +\frac{\theta_1\theta_2}{2d}\right)Z_1 
 +\frac{\theta_1\theta_2}{2d}Z_2.
 \label{eq:exceptional-linear-row-two}
\end{align}
Using \eqref{eq:exceptional-linear-row-one}--
\eqref{eq:exceptional-linear-row-two}, \eqref{eq:burgers-pointwise-bound},
\eqref{eq:no-prelog-input}, and \eqref{eq:phi-explicit-identities}, for each
of the four summands in the appropriate row, multiplication first by
\(\frac{\phi_i}{\phi_j}\theta_j\) and then by
\(\frac{\phi_i}{\phi_j}
(\frac{\theta_i\theta_j}{2d}-\theta_{j,y})\) gives, respectively,
\begin{align}
 &C(\epsilon+\epsilon^2)X_N(1+\tau)^{-1/2}
 \bigl((y-\lambda_j(1+\tau))^2+1+\tau\bigr)^{-3/2},\notag\\
 &C(\epsilon+\epsilon^2)X_N
(1+\tau)^{-\frac52}
\exp\left(-\frac{(y-\lambda_j(1+\tau))^2}{C(1+\tau)}\right).
 \label{eq:exceptional-linear-row-bound}
\end{align}
For the second estimation in \eqref{eq:exceptional-linear-row-bound}, we have
\begin{align*}
 &\int_{\mathbb R}C(\epsilon+\epsilon^2)X_N(1+\tau)^{-5/2}
 e^{-\frac{(y-\lambda_j(1+\tau))^2}{C(1+\tau)}}\,dy
 \leq C(\epsilon+\epsilon^2)X_N(1+\tau)^{-2},\\
 &\int_{\mathbb R}|y-\lambda_j(1+\tau)|
 C(\epsilon+\epsilon^2)X_N(1+\tau)^{-5/2}
 e^{-\frac{(y-\lambda_j(1+\tau))^2}{C(1+\tau)}}\,dy
 \leq C(\epsilon+\epsilon^2)X_N(1+\tau)^{-3/2}.
\end{align*}
Substitution in
\eqref{eq:exceptional-mass-derivative} and
\eqref{eq:exceptional-primitive-material-exact} therefore verifies
\eqref{eq:primitive-bounds} and \eqref{eq:material-primitive-bounds} for every
linear summand in Lemma~\ref{lem:gaussian-extraction}.

The remaining coefficient and nonlinear sources in
\(\mathcal L_jZ_j\) are paired before estimation through
\begin{equation*}
 \frac{\phi_i}{\phi_j}
 \left(\frac{\theta_i\theta_j}{2d}-\theta_{j,y}\right)\phi_j
 =-\partial_y(\phi_i\theta_j).
\end{equation*}
For every scalar source \(f\) in the complete \(j\)-row, the local contribution
and its zero-mass primitive must be combined before either is estimated.  The
combination is the exact identity
\begin{align}
 &\phi_i\theta_j f(y,\tau)
 +\int_{-\infty}^{y}\Bigg\{-\partial_z(\phi_i\theta_j)(z,\tau)f(z,\tau)
 \notag\\*[-1mm]
 &\qquad-\left[\int_{\mathbb R}
   -\partial_w(\phi_i\theta_j)(w,\tau)f(w,\tau)\,dw\right]
   \Gamma_j(z,1+\tau)\Bigg\}\,dz
 \notag\\*
 &=\int_{-\infty}^{y}\Bigg\{\phi_i\theta_j f_z(z,\tau)
 -\left[\int_{\mathbb R}\phi_i\theta_j f_w(w,\tau)\,dw\right]
   \Gamma_j(z,1+\tau)\Bigg\}\,dz .
 \notag
\end{align}
The complete background residual contains the linear term
\(d\gamma_j\phi_j\Xi_{j,yy}\).  Its contribution to the differentiated mass is
\begin{align}
 &(1+\tau)d\gamma_j\int_{\mathbb R}
 \frac{\phi_i}{\phi_j}
 \left(\frac{\theta_i\theta_j}{2d}-\theta_{j,y}\right)
 \phi_j\Xi_{j,yy}\,dy=(1+\tau)d\gamma_j\int_{\mathbb R}
 \phi_i\theta_j\Xi_{j,yyy}\,dy .
 \label{eq:exceptional-Xi-mass}
\end{align}
Its complete contribution to the material derivative of the zero-mass
primitive, including the local term, primitive term, and Gaussian mass, is
\begin{align}
&d\gamma_j\phi_i\theta_j\Xi_{j,yy}+d\gamma_j\int_{-\infty}^y\Bigg[
 \frac{\phi_i}{\phi_j}
 \left(\frac{\theta_i\theta_j}{2d}-\theta_{j,z}\right)
 \phi_j\Xi_{j,zz}\notag\\[-1mm]
&\hspace{28mm}
 -\left(\int_{\mathbb R}\frac{\phi_i}{\phi_j}
 \left(\frac{\theta_i\theta_j}{2d}-\theta_{j,w}\right)
 \phi_j\Xi_{j,ww}\,dw\right)\Gamma_j(z,1+\tau)\Bigg]dz\notag\\
&=d\gamma_j\int_{-\infty}^y\Bigg[
 \phi_i\theta_j\Xi_{j,zzz}
 -\left(\int_{\mathbb R}\phi_i\theta_j\Xi_{j,www}\,dw\right)
 \Gamma_j(z,1+\tau)\Bigg]dz .
 \label{eq:exceptional-Xi-cancellation}
\end{align}
In the last line the unlabelled factors in the first spatial integral are
evaluated at \((z,\tau)\), and those under \(dw\) at \((w,\tau)\).  Thus the
apparent local term cancels exactly with the boundary term.  The profile
estimates \eqref{eq:burgers-pointwise-bound},
\eqref{eq:resummed-pointwise-bound}, and \eqref{eq:phi-explicit-identities} give
\begin{align}
 |\phi_i\theta_j\Xi_{j,yyy}(y,\tau)|
 &\leq C\epsilon^3(1+\tau)^{-11/4}
 e^{-[y-\lambda_j(1+\tau)]^2/[C(1+\tau)]},\notag\\
 \left|\int_{\mathbb R}\phi_i\theta_j\Xi_{j,yyy}\,dy\right|
 &\leq C\epsilon^3(1+\tau)^{-9/4}.
 \label{eq:exceptional-Xi-profile-bounds}
\end{align}
Consequently \eqref{eq:exceptional-Xi-cancellation} and
\eqref{eq:exceptional-Xi-profile-bounds} give
\begin{equation}
 \eqref{eq:exceptional-Xi-cancellation} \le C\epsilon^3(1+\tau)^{-1/2}
 \bigl((y-\lambda_j(1+\tau))^2+1+\tau\bigr)^{-3/2},
 \label{eq:exceptional-Xi-bounds}
\end{equation}
and, by \eqref{eq:exceptional-Xi-mass} and
\eqref{eq:exceptional-Xi-profile-bounds}, \((1+\tau)\) times the corresponding
contribution to the differentiated mass is bounded by \(C\epsilon^3\).

Put
\[
 T=1+\tau,\qquad r_j=y-\lambda_jT,\qquad
 A_j=r_j^2+T,\qquad K_{ij}=\phi_i\theta_j.
\]
For a scalar source \(f\) in the complete \(j\)-row define its differentiated
mass and the local--primitive combination by
\begin{align}
 m_j[f]&:=-\int_{\mathbb R}K_{ij,y}f\,dy
          =\int_{\mathbb R}K_{ij}f_y\,dy,\notag\\
 \mathcal P_j[f](y)&:=K_{ij}f(y)
 +\int_{-\infty}^y
 \{-K_{ij,z}f(z)-m_j[f]\Gamma_j(z,T)\}\,dz.              \label{eq:exceptional-source-functional}
\end{align}
The second equality in the first line and the identity
\begin{equation}
 \mathcal P_j[f](y)=\int_{-\infty}^y
 \{K_{ij}f_z-m_j[f]\Gamma_j(z,T)\}\,dz                  \label{eq:exceptional-source-primitive}
\end{equation}
are integrations by parts; all boundary terms vanish by the displayed source
envelopes.  We repeatedly use the following centered-Gaussian fact.  If
\[
 |q(y)|\leq\Lambda T^{-5/2}e^{-r_j^2/(CT)},\qquad
 m=\int_{\mathbb R}q(y)\,dy,
\]
then one has
\begin{align}
 T^2|m|\leq C\Lambda,\qquad\quad
 T^{1/2}\sup_y A_j^{3/2}
 \left|\int_{-\infty}^y\{q(z)-m\Gamma_j(z,T)\}\,dz\right|
 &\leq C\Lambda .                                        \label{eq:exceptional-centered-gaussian}
\end{align}
Indeed, on \(|r_j|\leq\sqrt T\) this follows from the \(L^1\) bound, while
outside that set it follows from the two centered Gaussian tails.

For the coefficient source put
\[
 f_\eta=\eta(R_2-R_1).
\]
The exact product identity
\begin{equation*}
 K_{ij}(f_\eta)_y
 =\theta_j(\phi_i f_\eta)_y
  +\frac{\theta_i\theta_j}{2d}\phi_i f_\eta
\end{equation*}
and \eqref{eq:eta-envelope}, together with
\eqref{eq:burgers-pointwise-bound}, give
\begin{align}
 |K_{ij}(f_\eta)_y|
 &\leq C(\epsilon+\epsilon^2)X_N\log(2+\tau)T^{-11/4}
 e^{-r_j^2/(CT)}
 \leq C(\epsilon+\epsilon^2)X_N T^{-5/2}e^{-r_j^2/(CT)}. \notag
\end{align}
The terms centered on the opposite cone are exponentially separated, and the
last inequality uses \(\log(2+\tau)\leq CT^{1/4}\).  Thus
\eqref{eq:exceptional-source-primitive} and
\eqref{eq:exceptional-centered-gaussian} give
\begin{equation*}
 T^2|m_j[f_\eta]|+T^{1/2}\sup_yA_j^{3/2}
 |\mathcal P_j[f_\eta](y)|\leq C(\epsilon+\epsilon^2)X_N .
\end{equation*}

For the variable-viscosity source put
\[
 f_B=B(u_2^{\mathrm{app}}-u_1^{\mathrm{app}})(R_{1,y}+R_{2,y}).
\]
Here one must not estimate \((f_B)_y\), which would introduce the weak
\(R_{yy}\) rate.  Instead,
\eqref{eq:B-explicit-profile}, \eqref{eq:derived-low-rates}, and the Gaussian
bound for \(\theta_j\) give directly
\begin{align}
 |K_{ij}f_B|
 &\leq C(\epsilon+\epsilon^2)X_N\log(2+\tau)T^{-5/2}e^{-r_j^2/(CT)},\notag\\
 |K_{ij,y}f_B|
 &\leq C(\epsilon+\epsilon^2)X_N\log(2+\tau)T^{-3}e^{-r_j^2/(CT)}.
\label{eq:exceptional-viscosity-source-bound}
\end{align}
Using the first representation in
\eqref{eq:exceptional-source-functional},
\[
 \mathcal P_j[f_B]=K_{ij}f_B+
 \int_{-\infty}^y\{-K_{ij,z}f_B-m_j[f_B]\Gamma_j(z,T)\}\,dz.
\]
The Gaussian tails in \eqref{eq:exceptional-viscosity-source-bound}, with
\(\log(2+\tau)\leq CT^{1/2}\), imply
\begin{equation*}
 T^2|m_j[f_B]|+T^{1/2}\sup_yA_j^{3/2}
 |\mathcal P_j[f_B](y)|\leq C(\epsilon+\epsilon^2)X_N .
\end{equation*}
This is precisely the organization in
\eqref{eq:variable-viscosity-divergence}; no estimate of \(R_{yy}\) is used.

For \(f=N(R)\), retain the interpolation exponent \(\beta_N\) introduced
above.
The first, sharper line of
\eqref{eq:modal-nonlinear-derivative-bound}, obtained from
\eqref{eq:derived-low-rates} and
\eqref{eq:second-third-derivative-rate}, gives
\[
 \|N_y(\tau)\|_\infty\leq CX_N^2
 \{\log^2(2+\tau)T^{-5/2}
   +\log^{1+\beta_N}(2+\tau)T^{-1-\frac32\beta_N}\}.
\]
Consequently,
\begin{align}
 |K_{ij}N_y|
 &\leq C\epsilon X_N^2
 \{\log^2(2+\tau)T^{-3}
   +\log^{1+\beta_N}(2+\tau)T^{-\frac32-\frac32\beta_N}\}e^{-r_j^2/(CT)}
 \notag\\
 &\leq C\epsilon X_N^2T^{-5/2}e^{-r_j^2/(CT)}.            \notag
\end{align}
Since
\[
 \frac32\beta_N-1=\frac{2N-9}{2(2N-3)}>0
 \qquad (N\geq5),
\]
the logarithmic factor is absorbed, and
\eqref{eq:exceptional-source-primitive}--
\eqref{eq:exceptional-centered-gaussian} give the nonlinear bound recorded
in \eqref{eq:exceptional-common-source-ledger}.

Finally let
\[
 \phi_jf_{\mathrm{bg}}
 =-\phi_j\mathcal E_j^{\mathrm{app}}-d\gamma_j\phi_j\Xi_{j,yy}.
\]
Use \eqref{eq:exact-background-decomposition} and
\eqref{eq:complete-integrable-remainder} term by term, the mixed-profile
bound \eqref{eq:refined-product-bounds}, and
\eqref{eq:burgers-pointwise-bound}--\eqref{eq:resummed-pointwise-bound}.
The linear \(\Xi_{j,yy}\) contribution is first reduced by
\eqref{eq:exceptional-Xi-cancellation}--\eqref{eq:exceptional-Xi-bounds}.
These faster estimates yield
\begin{equation*}
 |K_{ij}(f_{\mathrm{bg}})_y|
 \leq C(\epsilon^2+\epsilon^3+\epsilon^4)
 T^{-5/2}e^{-r_j^2/(CT)}.
\end{equation*}
The \(\epsilon^2\) term in this line is necessary.  Indeed, the linear part
\[
 V_{j,\mathrm{lin}}^{\mathrm{bg}}
 =(-1)^j\frac{d^2}{2c}\phi_j\theta_{j,y}
\]
of \eqref{eq:V-bg} contributes
\[
 K_{ij}\partial_y\left[
 \phi_j^{-1}\partial_yV_{j,\mathrm{lin}}^{\mathrm{bg}}\right]
 =(-1)^j\frac{d^2}{2c}\phi_i\theta_j
 \partial_y\left[\phi_j^{-1}\partial_y(\phi_j\theta_{j,y})\right],
\]
which is generally of order
\(\epsilon^2T^{-5/2}e^{-r_j^2/(CT)}\); its centered primitive is of order
\(\epsilon^2T^{-2}\).  No cancellation removing this term occurs in the
present equation.

Collecting the four source bounds,
\begin{align}
 &T^2|m_j[f_\eta]|+T^{1/2}\sup_yA_j^{3/2}|\mathcal P_j[f_\eta]|
 \leq C(\epsilon+\epsilon^2)X_N,\notag\\
 &T^2|m_j[f_B]|+T^{1/2}\sup_yA_j^{3/2}|\mathcal P_j[f_B]|
 \leq C(\epsilon+\epsilon^2)X_N,\notag\\
 &T^2|m_j[N]|+T^{1/2}\sup_yA_j^{3/2}|\mathcal P_j[N]|
 \leq C\epsilon X_N^2,\notag\\
 &T^2|m_j[f_{\mathrm{bg}}]|+
 T^{1/2}\sup_yA_j^{3/2}|\mathcal P_j[f_{\mathrm{bg}}]|
 \leq C(\epsilon^2+\epsilon^3+\epsilon^4).
 \label{eq:exceptional-common-source-ledger}
\end{align}
Set
\begin{equation*}
 \mathcal S_N=(\epsilon+\epsilon^2)X_N+\epsilon X_N^2
 +\epsilon^2+\epsilon^3+\epsilon^4.
\end{equation*}
Combining \eqref{eq:exceptional-common-source-ledger} with
\eqref{eq:exceptional-cross-diffusion-bound},
\eqref{eq:exceptional-linear-row-bound},
\eqref{eq:exceptional-aj-prime-remainder}, and the exact modal cancellation
\eqref{eq:exceptional-modal-cancellation} gives
\begin{equation}
 \left|T\int_{\mathbb R}\phi_i\theta_jR_j\,dy\right|
 +T\left|\frac d{d\tau}\left[
 T\int_{\mathbb R}\phi_i\theta_jR_j\,dy\right]\right|
 \leq C\mathcal S_N.
 \label{eq:exceptional-mass-C1}
\end{equation}
For the primitive in \eqref{111}, substitution of the complete \(j\)-row into
\eqref{eq:exceptional-primitive-material-exact}, followed by
\eqref{eq:exceptional-primitive-product-rules}, gives
\begin{align}
 &\left|\int_{-\infty}^y(I_3+I_4)\,dz\right|
 \leq C\mathcal S_NA_j^{-1},\notag\\
 &\left|\partial_y\int_{-\infty}^y(I_3+I_4)\,dz\right|
 \leq C\mathcal S_NA_j^{-3/2},\notag\\
 &|(\partial_\tau+\lambda_j\partial_y-d\partial_y^2)
 \int_{-\infty}^y(I_3+I_4)\,dz|
 \leq C\mathcal S_NT^{-1/2}A_j^{-3/2}.
 \label{eq:exceptional-full-primitive}
\end{align}
For the material line the cross diffusion is
\eqref{eq:exceptional-cross-diffusion}, the four linear terms are bounded by
\eqref{eq:exceptional-linear-row-bound}, every common scalar source is reduced
by \eqref{eq:exceptional-source-functional}, and the modal contribution uses
\eqref{eq:exceptional-modal-cancellation}.  Thus all terms in the complete row
have been accounted for.  Neither \(\partial_\tau Z_j^{\mathrm{rem}}\) nor a
third spatial derivative of \(Z_j^{\mathrm{rem}}\) is used.

The term \(I_1\) in \eqref{eq:exceptional-decomposition} is estimated by the
material identity, rather than by its
absolute time integral.  Under \eqref{eq:no-prelog-input},
\eqref{eq:burgers-pointwise-bound}, and \eqref{eq:phi-explicit-identities},
\begin{align}
 &\left\|\frac{a_j\phi_i\theta_j^2}
 {4d(1-e^{-M_j/(2d)})}\right\|_1
 \leq C\epsilon |a_j|(1+\tau)^{-1/2},\notag\\
 &\left\|(\partial_\tau+\lambda_j\partial_y-d\partial_y^2)
 \frac{a_j\phi_i\theta_j^2}{4d(1-e^{-M_j/(2d)})}\right\|_1\notag\\[-1mm]
 \leq&C\epsilon\bigl\{|a_j'|(1+\tau)^{-1/2}
             +|a_j|(1+\tau)^{-3/2}\bigr\}
   +C\epsilon^2|a_j|e^{-(1+\tau)/C},\notag\\
 &\left\|\partial_y(\partial_\tau+\lambda_j\partial_y-d\partial_y^2)
 \frac{a_j\phi_i\theta_j^2}{4d(1-e^{-M_j/(2d)})}\right\|_1\notag\\[-1mm]
 \leq&C\epsilon\bigl\{|a_j'|(1+\tau)^{-1}
             +|a_j|(1+\tau)^{-2}\bigr\}
   +C\epsilon^2|a_j|e^{-(1+\tau)/C}.
   \label{eq:exceptional-material}
\end{align}
Applying \eqref{eq:material-identity} with \(\alpha=j\) therefore yields, for
\(0\leq k\leq2\),
\begin{equation}
 \left\|\partial_x^k\int_0^{(t-1)_+}g_i(t-\tau)*
 \partial_y\left\{\frac{a_j\phi_i\theta_j^2}
 {4d(1-e^{-M_j/(2d)})}\right\}d\tau\right\|_\infty
 \leq C\epsilon X_N(1+t)^{-(k+1)/2}.                    \label{eq:exceptional-divergence-map}
\end{equation}
Estimate \eqref{eq:exceptional-divergence-map} controls the unweighted
derivatives of \(Z_i^{\mathrm{rem}}\).  Its early- and late-time
contributions, including the two material boundary terms, are estimated
separately in \eqref{eq:app-exceptional-early-material}--
\eqref{eq:app-exceptional-late-material}.  It does not by itself imply a
cone-resolved bound.  For the weighted recovery of \(R_i\), with \(j=3-i\),
we also use the
following spatially resolved estimate:
\begin{equation}
 \left|\partial_x\int_0^{(t-1)_+}g_i(t-\tau)*
 \partial_y\left\{\frac{a_j\phi_i\theta_j^2}
 {4d(1-e^{-M_j/(2d)})}\right\}d\tau\right|
 \leq C\epsilon X_N\log(2+t)\Psi_i(x,t).                                      \label{eq:exceptional-divergence-pointwise}
\end{equation}
Its proof uses the exact material expansion and the spatial estimates
\eqref{eq:app-exceptional-material-spatial} in
Appendix~\ref{app:five-regions}.  Multiplying the \(k=0\) case of
\eqref{eq:exceptional-divergence-map} by \(|\theta_i(x,t)|\), and using
\eqref{eq:burgers-pointwise-bound}, gives the remaining weighted
contribution:
\begin{equation*}
 \left|\theta_i(x,t)\int_0^{(t-1)_+}g_i(t-\tau)*
 \partial_y\left\{\frac{a_j\phi_i\theta_j^2}
 {4d(1-e^{-M_j/(2d)})}\right\}d\tau\right|
 \leq C\epsilon^2X_N\Psi_i(x,t).
\end{equation*}
For \(k=2\), one derivative is moved to the material derivative in the late
integral.  The only borderline factor is
\[
 (1+t)^{-2}\log(2+t)\int_1^{t/2}\frac{d\tau}{1+\tau}
 \leq C(1+t)^{-3/2};
\]
it does not create a logarithm in \(Z_i^{\mathrm{rem}}\).

For the zero-mass part \(I_4\), the primitive in
\eqref{eq:exceptional-decomposition} satisfies.
\begin{equation}
\|\int_{-\infty}^y I_4 dz \|_{1} \le C\epsilon X_N(1+\tau)^{-1/2}.
 \label{eq:exceptional-zero-mass-L1}
\end{equation}
For \(0\leq k\leq2\), the \(L^\infty(\mathbb R_y)\)-norm of its \(k\)-th
spatial derivative is at most
\begin{equation}
\| \partial_y^k \int_{-\infty}^y I_4 dz\|_{\infty} \le C\epsilon X_N(1+\tau)^{-1-k/2}.
 \label{eq:exceptional-zero-mass-Linf}
\end{equation}
The corresponding \(I_3\)-bracket containing
\(\theta_i\theta_j^2\) is exponentially small. Thus, $\int_{-\infty}^y (I_3+I_4)dz$ satisfies \eqref{eq:primitive-bounds}. Hence
\eqref{eq:exceptional-mass-C1}, \eqref{eq:exceptional-full-primitive},
\eqref{eq:exceptional-zero-mass-L1}--\eqref{eq:exceptional-zero-mass-Linf},
Lemma~\ref{lem:gaussian-extraction}, together with
\eqref{eq:exceptional-divergence-map} and
\eqref{eq:exceptional-divergence-pointwise}, gives
\begin{align*}
 &\Bigl\|\partial_x^k\Biggl[
 -\int_0^{(t-1)_+}g_i(t-\tau)*(\phi_i\theta_jR_j)(\tau)\,d\tau+\Gamma_i(\cdot,1+t)\int_0^{(t-1)_+}
 \chi_{\mathrm{time}}\left(\frac{1+\tau}{\sqrt{1+t}}\right)
 \int_{\mathbb R}\phi_i\theta_jR_j\,dy\,d\tau
 \Biggr]\Bigr\|_\infty\\
 \leq& C\mathcal S_N(1+t)^{-(k+1)/2},
       \qquad\qquad\qquad\qquad\qquad\qquad\qquad\qquad\qquad\qquad\qquad\qquad\qquad\qquad\quad 0\leq k\leq2,
\end{align*}
\begin{align}
& \Bigl|\partial_x\Biggl[
 -\int_0^{(t-1)_+}g_i(t-\tau)*(\phi_i\theta_jR_j)(\tau)\,d\tau+\Gamma_i(x,1+t)\int_0^{(t-1)_+}
 \chi_{\mathrm{time}}\left(\frac{1+\tau}{\sqrt{1+t}}\right)
 \int_{\mathbb R}\phi_i\theta_jR_j\,dy\,d\tau
 \Biggr]\Bigr|\notag\\
 \leq& C\mathcal S_N\log(2+t)
          \Psi_i(x,t),                         \label{eq:exceptional-pointwise}
\end{align}
The Gaussian subtracted in these two estimates is precisely the \(I_2\)-mass
already included in \(a_i\) by \eqref{eq:ai-definition}.
The time cutoff in that definition extracts only the opposite-family mass
feeding the \(i\)-th receiving Gaussian, while \(I_1\), \(I_3\), and \(I_4\)
remain in the cone-resolved maps.  Thus the \(j\)-family emitting cone and the
\(i\)-family receiving cone remain distinct throughout the cancellation.
On the last time unit the original product is used.  By
\eqref{eq:burgers-pointwise-bound}, \eqref{eq:derived-low-rates},
\eqref{eq:second-third-derivative-rate}, and \eqref{eq:last-unit-general}, its
first two derivatives contain at most \(R_{j,yy}\), and give the same bounds.
% No third derivative of
% \(Z_j^{\mathrm{rem}}\) is required.

It remains to treat the secondary derivative mode generated by
\[
 \int_0^{(t-1)_+}\int_{\mathbb R}(g_{1,y}-g_{2,y})
 \left(\frac12\theta_{j,y}-\frac{\theta_j^2}{4d}\right)Z_j\,dy\,d\tau .
\]
Direct integration by parts gives
\begin{equation*}
 \int_{\mathbb R}\left(\frac12\theta_{j,y}-\frac{\theta_j^2}{4d}\right)
       \Gamma_j(y,1+\tau)\,dy
 =-\frac{1}{16d^2(1-e^{-M_j/(2d)})}
       \int_{\mathbb R}\phi_j\theta_j^3\,dy.
\end{equation*}
By \eqref{eq:burgers-pointwise-bound} and
\eqref{eq:phi-explicit-identities}, the right-hand side equals a constant of
size \(C\epsilon^2\), divided by \(1+\tau\).  Its same-family heat response is
controlled by \eqref{eq:app-semigroup}.
In the opposite output, the five-region calculation of
Appendix~\ref{app:five-regions} gives exactly the secondary
\(\log(2+t)(|x-\lambda_{3-j}(1+t)|^2+(1+t)^{3/2})^{-1}\) term already present in
\(\Psi_j\).  Since it is a derivative mode, it is not included in \(a_i\).

By \eqref{eq:Z-decomposition}, \eqref{eq:no-prelog-input}, and
\eqref{eq:burgers-pointwise-bound}, the two same-family coefficients occurring
only below an off-diagonal Green derivative satisfy
\begin{align}
 \left\|\left(\frac12\theta_{j,y}-\frac{\theta_j^2}{4d}\right)Z_j\right\|_1
 &\leq C\epsilon X_N\log(2+\tau)(1+\tau)^{-1},\notag\\
 \left\|\left(\frac12\theta_{j,y}-\frac{\theta_j^2}{4d}\right)Z_j\right\|_\infty
 &\leq C\epsilon X_N\log(2+\tau)(1+\tau)^{-3/2}.
 \notag
\end{align}
One source derivative adds \((1+\tau)^{-1/2}\) to the displayed
\(L^\infty\) power.  The zero-mass primitive of the Gaussian modal product is
\begin{multline}
 \int_{-\infty}^{y}\Bigg\{
 \left(\frac12\theta_{j,z}-\frac{\theta_j^2}{4d}\right)(z,\tau)
 \Gamma_j(z,1+\tau)\\[-1mm]
 -\Gamma_j(z,1+\tau)\int_{\mathbb R}
 \left(\frac12\theta_{j,w}-\frac{\theta_j^2}{4d}\right)(w,\tau)
 \Gamma_j(w,1+\tau)\,dw\Bigg\}dz .
 \label{eq:secondary-mode-primitive}
\end{multline}
By \eqref{eq:burgers-pointwise-bound}, the \(L^1(\mathbb R_y)\)-norm of the primitive in
\eqref{eq:secondary-mode-primitive} is bounded by
\(C\epsilon(1+\tau)^{-1/2}\).  Moreover, for \(0\leq q\leq2\), the
absolute value of its \(q\)-th spatial derivative is bounded by
\begin{equation}
 \eqref{eq:secondary-mode-primitive}\leq C\epsilon(1+\tau)^{-1-q/2}
 e^{-[y-\lambda_j(1+\tau)]^2/[C(1+\tau)]}.
 \label{eq:secondary-mode-primitive-bounds}
\end{equation}
Consequently, by \eqref{eq:Z-decomposition} and
\eqref{eq:secondary-mode-primitive}, the exact modal decomposition is
\begin{align}
 \left(\frac12\theta_{j,y}-\frac{\theta_j^2}{4d}\right)Z_j
 =&a_j(\tau)\left[\int_{\mathbb R}
 \left(\frac12\theta_{j,y}-\frac{\theta_j^2}{4d}\right)
 \Gamma_j(y,1+\tau)\,dy\right]\times\Gamma_j(\cdot,1+\tau)+\left(\frac12\theta_{j,y}-\frac{\theta_j^2}{4d}\right)Z_j^{\mathrm{rem}}\notag\\
 &+\partial_y\Bigg\{a_j(\tau)\int_{-\infty}^{y}\Bigg[
 \left(\frac12\theta_{j,z}-\frac{\theta_j^2}{4d}\right)(z,\tau)
 \Gamma_j(z,1+\tau)\notag\\[-1mm]
 &\hspace{12mm}-\Gamma_j(z,1+\tau)\int_{\mathbb R}
 \left(\frac12\theta_{j,w}-\frac{\theta_j^2}{4d}\right)(w,\tau)
 \Gamma_j(w,1+\tau)\,dw\Bigg]dz\Bigg\}.
 \label{eq:secondary-mode-decomposition}
\end{align}
Substitute \eqref{eq:secondary-mode-decomposition} into
\[
 (-1)^{j+1}\frac d{2c}(g_{j,y}-g_{3-j,y})
 \left[\left(\frac12\theta_{j,y}-\frac{\theta_j^2}{4d}\right)Z_j\right],
\]
whose same-family part satisfies
\begin{align*}
 &\left|\partial_x\int_0^{(t-1)_+}g_{j,y}(t-\tau)*
 \left\{a_j(\tau)\left[\int_{\mathbb R}
 \left(\frac12\theta_{j,y}-\frac{\theta_j^2}{4d}\right)
 \Gamma_j\,dy\right]\Gamma_j\right\}d\tau\right|\\
 &\leq C\epsilon^2X_N\log^2(2+t)(1+t)^{-3/2}
 e^{-[x-\lambda_j(1+t)]^2/[C(1+t)]}
 \leq C\epsilon^2X_N\log(2+t)\Psi_j(x,t).
\end{align*}
After the final output derivative, the opposite-family part, with
\(S=1+\tau\), is
\begin{multline*}
 (-1)^j\frac d{2c}\partial_x^2\int_1^{1+t}\frac{a_j(S-1)}S
 \Bigg[S\int_{\mathbb R}
 \left(\frac12\theta_{j,y}-\frac{\theta_j^2}{4d}\right)(y,S-1)
 \Gamma_j(y,S)\,dy\Bigg]\\*[-1mm]
 \times\frac{1}{\sqrt{4\pi d(1+t)}}
 \exp\left[-\frac{[x-\lambda_{3-j}(1+t)
 -(\lambda_j-\lambda_{3-j})S]^2}{4d(1+t)}\right]dS .
\end{multline*}
Changing the upper endpoint by one produces only a last-unit term.  By
\eqref{eq:no-prelog-input}, in the
\(K\sqrt{1+t}\)-tube about the \(3-j\) cone,
\begin{equation}
 \int_1^{\sqrt{1+t}}\frac{|a_j(S-1)|}{S}\,dS
 \leq CX_N\log^2(2+t).
 \label{eq:secondary-mode-source-tube}
\end{equation}
In the \(j\)-tube the Gaussian localizes \(S\) to an interval of length
\(C\sqrt{1+t}\) next to \(1+t\).  Strictly between the two tubes,
\begin{align*}
 &(1+t)^{-3/2}\int_1^{1+t}\frac{\log(2+S)}S
 \exp\left[-\frac{c[S-(x-\lambda_{3-j}(1+t))/
 (\lambda_j-\lambda_{3-j})]^2}{1+t}\right]dS\\
 \leq& C\log(2+t)(1+t)^{-1}
 \left|\frac{x-\lambda_{3-j}(1+t)}{\lambda_j-\lambda_{3-j}}\right|^{-1}
 \leq \frac{C\log(2+t)}
 {\langle x-\lambda_j(1+t)\rangle^{5/4}
  \langle x-\lambda_{3-j}(1+t)\rangle^{1/2}} .
\end{align*}
On the two exterior half-lines the phase is monotone and gives, respectively,
\begin{equation}
 \frac{C\epsilon^2X_N\log(2+t)}{[x-\lambda_j(1+t)]^2+1+t},
 \qquad
 \frac{C\epsilon^2X_N\log^2(2+t)}
 {[x-\lambda_{3-j}(1+t)]^2+(1+t)^{3/2}} .
 \label{eq:secondary-mode-exteriors}
\end{equation}
Thus \eqref{eq:secondary-mode-source-tube}--
\eqref{eq:secondary-mode-exteriors} give
\begin{align}
 &\left|\partial_x\int_0^{(t-1)_+}(g_{j,y}-g_{3-j,y})(t-\tau)*
 \left\{a_j(\tau)\left[\int_{\mathbb R}
 \left(\frac12\theta_{j,y}-\frac{\theta_j^2}{4d}\right)
 \Gamma_j\,dy\right]\Gamma_j\right\}d\tau\right|\notag\\
 \leq& C\epsilon^2X_N\log(2+t)\Psi_j(x,t).
 \label{eq:secondary-mode-five-region}
\end{align}
For the primitive in \eqref{eq:secondary-mode-decomposition}, keep three
Gaussian derivatives on the early kernel and transfer two derivatives to the
displayed primitive on the late interval.  The last term uses the no-prelog
bounds for \(Z_j^{\mathrm{rem}}\).  Combining
\eqref{eq:secondary-mode-primitive-bounds},
\eqref{eq:secondary-mode-five-region}, and \eqref{eq:no-prelog-input} gives
\begin{align}
 &\left|\partial_x\int_0^{(t-1)_+}(g_{j,y}-g_{3-j,y})(t-\tau)*
 \left[\left(\frac12\theta_{j,y}-\frac{\theta_j^2}{4d}\right)Z_j\right](\tau)
 d\tau\right|\notag\\
 \leq &C(\epsilon+\epsilon^2)X_N\log(2+t)\Psi_j(x,t).
 \label{eq:secondary-mode-complete}
\end{align}

The off-diagonal Gaussian carries the same-speed product
\(\phi_i\theta_jR_j\) into the opposite output.  Its mass--primitive
decomposition is
\begin{align}
 \phi_i\theta_jR_j=&
 \left[\int_{\mathbb R}\phi_i\theta_jR_j\,dy\right]
 \Gamma_j(\cdot,1+\tau)\notag\\
 &+\partial_y\int_{-\infty}^{y}\Bigg\{
 \phi_i\theta_jR_j(z,\tau)
 -\Gamma_j(z,1+\tau)\int_{\mathbb R}\phi_i\theta_jR_j(w,\tau)\,dw
 \Bigg\}dz .
 \label{eq:offdiagonal-same-speed-decomposition}
\end{align}
By 
\eqref{eq:exceptional-mass-C1},  its coefficient and primitive satisfy
\begin{equation}
 \left|(1+\tau)\int_{\mathbb R}\phi_i\theta_jR_j\,dy\right|
 \leq C(\epsilon+\epsilon^2)X_N.
 \label{eq:offdiagonal-same-speed-bounds}
\end{equation}
If \(Q_j^{\mathrm{off}}\) denotes the primitive displayed in
\eqref{eq:offdiagonal-same-speed-decomposition}, then
\begin{equation*}
 \|Q_j^{\mathrm{off}}(\tau)\|_1
 \leq C(\epsilon+\epsilon^2)X_N\log(2+\tau)(1+\tau)^{-1/2},
 \qquad
 \|\partial_y^2Q_j^{\mathrm{off}}(\tau)\|_\infty
 \leq C(\epsilon+\epsilon^2)X_N\log(2+\tau)(1+\tau)^{-2}.
\end{equation*}
Moreover,
\begin{equation}
 C(\epsilon+\epsilon^2)X_N\log(2+\tau)(1+\tau)^{-1}
 e^{-[y-\lambda_j(1+\tau)]^2/[C(1+\tau)]}.
 \label{eq:offdiagonal-same-speed-pointwise}
\end{equation}
Thus, \eqref{eq:offdiagonal-same-speed-bounds} and
\eqref{eq:app-semigroup} give for the Gaussian part
\begin{equation}
\begin{split}
& \left|\partial_x\int_0^{(t-1)_+}g_{j,y}(t-\tau)*
 \left\{\left[\int_{\mathbb R}\phi_i\theta_jR_j\,dy\right]
 \Gamma_j(\cdot,1+\tau)\right\}d\tau\right|\\
& \leq C(\epsilon+\epsilon^2)X_N\log(2+t)(1+t)^{-3/2}
 e^{-[x-\lambda_j(1+t)]^2/[C(1+t)]}.
 \label{eq:offdiagonal-same-speed-gaussian}
 \end{split}
\end{equation}
For the primitive, \eqref{eq:offdiagonal-same-speed-pointwise} and
\eqref{eq:app-early-convolution}--\eqref{eq:app-late-convolution-one}, with
all three Green derivatives on the early kernel and transferring two to the
primitive on the late interval, give
\begin{align}
 &C(\epsilon+\epsilon^2)X_N(1+t)^{-2}
 \int_0^{t/2}\log(2+\tau)(1+\tau)^{-1/2}\,d\tau\notag\\
 &\quad+C(\epsilon+\epsilon^2)X_N\log(2+t)(1+t)^{-2}
 \int_1^{(1+t)/2}\sigma^{-1/2}\,d\sigma\notag\\
 &\leq C(\epsilon+\epsilon^2)X_N\log(2+t)(1+t)^{-3/2}.
 \label{eq:offdiagonal-same-speed-primitive}
\end{align}
Therefore \eqref{eq:offdiagonal-same-speed-gaussian}--
\eqref{eq:offdiagonal-same-speed-primitive} give
\begin{align}
 &\left|\partial_x\int_0^{(t-1)_+}g_{j,y}(t-\tau)*
 [\phi_i\theta_jR_j](\tau)\,d\tau\right|\notag\\
 \leq& C(\epsilon+\epsilon^2)X_N\log(2+t)(1+t)^{-3/2}
 e^{-[x-\lambda_j(1+t)]^2/[C(1+t)]}\notag\\
 \leq& C(\epsilon+\epsilon^2)X_N\log(2+t)\Psi_i(x,t).
 \label{eq:offdiagonal-same-speed-complete}
\end{align}

\subsection{The explicit profile sources}\label{section 5.4}
Up to constant coefficients, the critical spatial profiles in
\eqref{eq:ai-definition} and \eqref{eq:remaining-source-column} are
\begin{align}
 f=\phi_i\theta_\alpha^3
 \quad\hbox{or}\quad
 f=\phi_i\chi_\alpha^{\mathrm{cone}}\Xi_\alpha^2,
 \qquad \alpha\in\{1,2\},\label{cvb}
\end{align}
where the cone cutoff is fixed by \eqref{eq:profile-cutoff}.  Set
\[
 T=1+\tau,\qquad r_\alpha=y-\lambda_\alpha T,
 \qquad A_\alpha=r_\alpha^2+T,
 \qquad L_\alpha=\partial_\tau+\lambda_\alpha\partial_y-d\partial_y^2.
\]
By \eqref{eq:phi-identities} and \eqref{eq:burgers-wave},
\begin{align*}
 L_\alpha\phi_i&=(\lambda_\alpha-\lambda_i)\phi_{i,y}, \qquad
 L_\alpha\theta_\alpha=-\theta_\alpha\theta_{\alpha,y}.
\end{align*}
Hence, we have
\begin{align}
 &L_\alpha(\phi_i\theta_\alpha^3)=(\lambda_\alpha-\lambda_i)\phi_{i,y}\theta_\alpha^3
       -3\phi_i\theta_\alpha^3\theta_{\alpha,y}
       -6d\phi_i\theta_\alpha\theta_{\alpha,y}^2-6d\phi_{i,y}\theta_\alpha^2\theta_{\alpha,y}.
\label{eq:critical-cubic-product}
\end{align}
For the cutoff profile, it holds
\begin{align}
 &L_\alpha(\phi_i\chi_\alpha^{\mathrm{cone}}\Xi_\alpha^2)\notag\\
 =&\chi_\alpha^{\mathrm{cone}}\Xi_\alpha^2
       L_\alpha\phi_i
       +\phi_i\Xi_\alpha^2
       L_\alpha\chi_\alpha^{\mathrm{cone}}
       -2d\phi_{i,y}\partial_y\chi_\alpha^{\mathrm{cone}}\Xi_\alpha^2\notag\\
&-2d\phi_i\chi_\alpha^{\mathrm{cone}}\Xi_{\alpha,y}^2-4d\phi_{i,y}\chi_\alpha^{\mathrm{cone}}\Xi_\alpha\Xi_{\alpha,y}
-4d\phi_i\partial_y\chi_\alpha^{\mathrm{cone}}\Xi_\alpha\Xi_{\alpha,y}+2\phi_i\chi_\alpha^{\mathrm{cone}}\Xi_\alpha L_\alpha\Xi_\alpha.
\label{eq:critical-cutoff-product}
\end{align}
On \(\operatorname{supp}\chi_\alpha^{\mathrm{cone}}\),
\eqref{eq:resummed-pointwise-bound} and \eqref{eq:profile-cutoff} give
\begin{gather*}
 |\partial_y^k\Xi_\alpha(y,\tau)|
 \leq C_k\epsilon^2T^{-k/2}
 \bigl(A_\alpha^{-3/4}+T^{-3/2}\bigr),
 \qquad 0\leq k\leq2,\\
 |\partial_y^k\chi_\alpha^{\mathrm{cone}}|\leq C_kT^{-k},\qquad
 |L_\alpha\chi_\alpha^{\mathrm{cone}}|
 \leq CT^{-1}\mathbf 1_{\{c/4\leq |r_\alpha|/T\leq c/2\}}.
\end{gather*}
The only term in \eqref{eq:critical-cutoff-product} without a purely
Gaussian derivative gain satisfies
\[
 |\chi_\alpha^{\mathrm{cone}}\Xi_{\alpha,y}^2|
 \leq C\epsilon^4T^{-1}A_\alpha^{-3/2}.
\]
The material identity \eqref{eq:resummed-profile} reads
\[
 L_\alpha\Xi_\alpha
 =-\partial_y(\theta_\alpha\Xi_\alpha)
  -\partial_y\left[
    \theta_{3-\alpha}\left(\frac12\theta_{3-\alpha}
   +\Xi_{3-\alpha}\right)\right].
\]
Cone separation, \eqref{eq:burgers-pointwise-bound}, and
\eqref{eq:resummed-pointwise-bound} yield
\[
 \mathbf 1_{\operatorname{supp}\chi_\alpha^{\mathrm{cone}}}
 |L_\alpha\Xi_\alpha|
 \leq C\epsilon^2A_\alpha^{-7/4}.
\]
\begin{align*}
 &|\phi_i\chi_\alpha^{\mathrm{cone}}\Xi_\alpha L_\alpha\Xi_\alpha|
 +|\phi_{i,y}\chi_\alpha^{\mathrm{cone}}\Xi_\alpha\Xi_{\alpha,y}|
 +|\chi_\alpha^{\mathrm{cone}}\Xi_\alpha^2L_\alpha\phi_i|\leq C\epsilon^4A_\alpha^{-5/2}.
\end{align*}
Indeed, \(L_\alpha\phi_i=0\) for \(i=\alpha\), whereas
\(\phi_{i,y}\) is exponentially small on the \(\alpha\)-cone for
\(i\ne\alpha\).  On the cutoff annulus, it holds
\[
 \frac{c^2}{16}T^2\leq A_\alpha=r_\alpha^2+T
 \leq\left(1+\frac{c^2}{4}\right)T^2.
\]
Thus the cutoff-material and mixed cutoff-gradient terms satisfy
\begin{align*}
 &|\phi_i\Xi_\alpha^2L_\alpha\chi_\alpha^{\mathrm{cone}}|
 +|\phi_i(\partial_y\chi_\alpha^{\mathrm{cone}})
       \Xi_\alpha\Xi_{\alpha,y}|\leq C\epsilon^4T^{-1}A_\alpha^{-3/2}.
\end{align*}
Combining these bounds with
\eqref{eq:critical-cubic-product}--\eqref{eq:critical-cutoff-product}
gives, for either choice of \(f\) \eqref{cvb}
\begin{align}
 |\partial_y^k f(y,\tau)|
 &\leq C_k(\epsilon^3+\epsilon^4)T^{-k/2}A_\alpha^{-3/2},
 \qquad 0\leq k\leq2,\notag\\
 |L_\alpha f(y,\tau)|
 &\leq C(\epsilon^3+\epsilon^4)
 \left\{A_\alpha^{-5/2}+T^{-1}A_\alpha^{-3/2}\right\}.
 \label{eq:critical-profile-material-bound}
\end{align}
The mass--primitive decomposition \eqref{eq:mass-primitive-form} is
\begin{align}
 b_\alpha(\tau)
 &=T\int_{\mathbb R}f(y,\tau)\,dy,\notag\\
 Q_\alpha(y,\tau)
 &=\int_{-\infty}^y\left\{
 f(z,\tau)-\frac{b_\alpha(\tau)}T
 \Gamma_\alpha(z,T)\right\}dz.                       \label{eq:critical-mass-and-primitive}
\end{align}
Equivalently, $f$ can be written as 
\[
 f=\frac{b_\alpha}T\Gamma_\alpha+Q_{\alpha,y},
 \qquad \int_{\mathbb R}Q_{\alpha,y}(y,\tau)\,dy=0.
\]
Since the spatial derivative terms have zero mass, it is easy to see that
\begin{align}
 b_\alpha'(\tau)
 &=\int_{\mathbb R}f(y,\tau)\,dy
   +T\int_{\mathbb R}L_\alpha f(y,\tau)\,dy.             \label{eq:critical-mass-derivative}
\end{align}
\[
 \int_{\mathbb R}A_\alpha^{-3/2}\,dy=CT^{-1},\qquad
 \int_{\mathbb R}\left(A_\alpha^{-5/2}
       +T^{-1}A_\alpha^{-3/2}\right)dy=CT^{-2}.
\]
Therefore \eqref{eq:critical-profile-material-bound}--
\eqref{eq:critical-mass-derivative} imply
\[
 |b_\alpha(\tau)|+T|b_\alpha'(\tau)|
 \leq C(\epsilon^3+\epsilon^4).
\]
Moreover, we have
\[
 Q_{\alpha,y}=f-\frac{b_\alpha}{T}\Gamma_\alpha,\qquad
 Q_{\alpha,yy}=f_y-\frac{b_\alpha}{T}\Gamma_{\alpha,y}.
\]
For \(P_\alpha=L_\alpha Q_\alpha\), the identity
\(L_\alpha\Gamma_\alpha(\cdot,T)=0\) gives
\[
 P_\alpha(y,\tau)=\int_{-\infty}^y\left[
 L_\alpha f(z,\tau)
 -\left(\frac{b_\alpha'}T-\frac{b_\alpha}{T^2}\right)
 \Gamma_\alpha(z,T)\right]dz,
\]
and 
\[
 P_{\alpha,y}
 =L_\alpha f-\left(\frac{b_\alpha'}T-\frac{b_\alpha}{T^2}\right)
 \Gamma_\alpha(y,T).
\]
By \eqref{eq:critical-mass-derivative}, we have
\[
 \int_{\mathbb R}P_{\alpha,y}(y,\tau)\,dy=0,
 \qquad P_\alpha(\pm\infty,\tau)=0.
\]
Integrating \(Q_{\alpha,y}\) and \(P_{\alpha,y}\) from the nearer
spatial infinity yields
\begin{align} \label{eq:critical-primitive-material}
 &|b_\alpha|+T|b_\alpha'|\leq C(\epsilon^3+\epsilon^4),\qquad
 \|Q_\alpha(\tau)\|_1\leq C(\epsilon^3+\epsilon^4)T^{-1/2},\notag\\
 &|Q_\alpha(y,\tau)|\leq C(\epsilon^3+\epsilon^4)A_\alpha^{-1},\qquad
 |Q_{\alpha,y}(y,\tau)|
 \leq C(\epsilon^3+\epsilon^4)A_\alpha^{-3/2},\notag\\
 &|Q_{\alpha,yy}(y,\tau)|
 \leq C(\epsilon^3+\epsilon^4)T^{-1/2}A_\alpha^{-3/2},\qquad
 \|P_\alpha(\tau)\|_1
 \leq C(\epsilon^3+\epsilon^4)T^{-3/2},\\
 &|P_\alpha(y,\tau)|
 \leq C(\epsilon^3+\epsilon^4)T^{-1}A_\alpha^{-1},\qquad
 |P_{\alpha,y}(y,\tau)|
 \leq C(\epsilon^3+\epsilon^4)\left\{
 T^{-1}A_\alpha^{-3/2}
 +T^{-5/2}\exp\left(-\frac{r_\alpha^2}{CT}\right)\right\}.  \notag 
\end{align}
Thus \eqref{eq:critical-profile-material-bound} and
\eqref{eq:critical-primitive-material} verify the hypotheses of
Lemma~\ref{lem:gaussian-extraction}, including the last-unit bounds
\eqref{eq:last-unit-general}.  Let \(Z_i^{\mathrm{prof}}\) denote the
complete responses after removal of the Gaussian masses in
\eqref{eq:ai-definition}.  Lemmas~\ref{lem:endpoint-convolution} and
\ref{lem:gaussian-extraction}, together with \eqref{eq:Rr-norms} and
\eqref{eq:singular-action}, yield
\begin{equation}
\begin{gathered}
 \sum_{k=0}^2(1+t)^{(k+1)/2}
 \|\partial_x^k Z_i^{\mathrm{prof}}(t)\|_\infty
 \leq C(\epsilon^3+\epsilon^4),\\
 |\partial_x Z_i^{\mathrm{prof}}(x,t)|
 +|\theta_i(x,t)Z_i^{\mathrm{prof}}(x,t)|
 \leq C(\epsilon^3+\epsilon^4)\log(2+t)\Psi_i(x,t).
\end{gathered}
\label{eq:critical-profile-response}
\end{equation}

\subsection{The remaining terms in the Duhamel formula}

For \(0\leq\tau\leq(t-1)_+\), the four \(\theta\)-terms and the
commutator reduce to the following two rows:
\begin{align}
 \int_0^{(t-1)_+}\int_{\mathbb R}\Bigg\{&
 g_1\phi_1\theta_2R_1
 +g_1\phi_1\left(\frac32\theta_1-\theta_2\right)R_2\notag\\
 &+\frac d{2c}(g_{1,y}-g_{2,y})\phi_1
       \{\theta_2R_1+(2\theta_1-\theta_2)R_2\}\notag\\
 &+\frac d{2c}(g_{1,y}-g_{2,y})
       \left(\frac12\theta_{1,y}-\frac{\theta_1^2}{4d}\right)Z_1
 -\frac12\frac{\phi_1}{\phi_2}\theta_1g_{2,y}Z_2
 \Bigg\}\,dy\,d\tau,                                      \label{eq:first-leading-row}
\end{align}
and
\begin{align}
 \int_0^{(t-1)_+}\int_{\mathbb R}\Bigg\{&
 g_2\phi_2\theta_1R_2
 +g_2\phi_2\left(\frac32\theta_2-\theta_1\right)R_1\notag\\
 &+\frac d{2c}(g_{1,y}-g_{2,y})\phi_2
       \{(2\theta_2-\theta_1)R_1+\theta_1R_2\}\notag\\
 &+\frac d{2c}(g_{1,y}-g_{2,y})
       \left(\frac12\theta_{2,y}-\frac{\theta_2^2}{4d}\right)Z_2
 -\frac12\frac{\phi_2}{\phi_1}\theta_2g_{1,y}Z_1
 \Bigg\}\,dy\,d\tau.                                      \label{eq:second-leading-row}
\end{align}
The identities \eqref{eq:expanded-commutator-12}--
\eqref{eq:expanded-commutator-21} and
\eqref{eq:source-row-one}--\eqref{eq:source-row-two} cancel the two
undifferentiated cross terms
\[
 \frac{\phi_1}{\phi_2}
 \left(\frac12\theta_{2,y}-\frac{\theta_2^2}{4d}\right)g_1Z_2,
 \qquad
 \frac{\phi_2}{\phi_1}
 \left(\frac12\theta_{1,y}-\frac{\theta_1^2}{4d}\right)g_2Z_1.
\]
Thus
\[
 \mathcal F_i^{\mathrm{exc}}=-\phi_i\theta_{3-i}R_{3-i},
 \qquad i=1,2,
\]
are the only nonordinary terms without an additional Green derivative; they
were treated in Subsection~\ref{exceptional opposite-family feedback} by
\eqref{eq:exceptional-decomposition}.  All other long-time components remain
in \eqref{eq:complete-duhamel}.  On \(t-1\leq\tau\leq t\), the source rows act
through \(G^0\), and the full commutator \(E\) acts on the opposite component.

Let \(Z_i^{\mathrm{ord}}\) denote the sum of the ordinary profile and
coefficient responses.  By \eqref{eq:modal-recovery},
\eqref{eq:no-prelog-input}, and \eqref{eq:burgers-pointwise-bound}, their source
products are
\[
 \phi_1\theta_2R_1,\qquad \frac32\phi_1\theta_1R_2,
 \qquad \phi_2\theta_1R_2,\qquad \frac32\phi_2\theta_2R_1.
\]
If \(\alpha\) is the profile family in \(f\), define \(b_\alpha,Q_\alpha\)
as in \eqref{eq:critical-mass-and-primitive}.  Then
\[
 f=\frac{b_\alpha(\tau)}{1+\tau}\Gamma_\alpha(\cdot,1+\tau)
   +\partial_yQ_\alpha,
\]
and the modal decomposition gives
\begin{align}
 &|b_\alpha(\tau)|\leq C\epsilon X_N,\qquad\qquad
 \|Q_\alpha(\tau)\|_1\leq C\epsilon X_N(1+\tau)^{-1/2},\notag\\
 &\|\partial_y^kQ_\alpha(\tau)\|_\infty
 \leq C\epsilon X_N(1+\tau)^{-1-k/2},
 \qquad 0\leq k\leq2.
 \label{eq:ordinary-cross-primitive-bounds}
\end{align}
The corresponding same-family envelope is
\begin{equation}
 |\phi_i\theta_iR_{3-i}|(y,\tau)
 \leq C\epsilon X_N
 \bigl((y-\lambda_i(1+\tau))^2+1+\tau\bigr)^{-3/2}.
 \label{eq:ordinary-same-family-pointwise}
\end{equation}
For \(\alpha=3-i\), the opposite-family envelope is
\begin{align}
 |\phi_i\theta_\alpha R_i|(y,\tau)
 \leq& C\epsilon X_N\log(2+\tau)(1+\tau)^{-1/2}
 [1
   +\log(2+\tau)]\times
 \bigl((y-\lambda_\alpha(1+\tau))^2+1+\tau\bigr)^{-3/2}\notag\\
 &+C\epsilon X_N\log(2+\tau)(1+\tau)^{-7/4}
\exp\!\left(
-\frac{(y-\lambda_\alpha (1+\tau))^2}{C(1+\tau)}
\right).
 \label{eq:ordinary-opposite-family-pointwise}
\end{align}
The last Gaussian term satisfies the cone-preserving bound
\[
 \int_0^{(t-1)_+}
 \left|\partial_xg_i(\cdot,t-\tau)\right|*
 \left[(1+\tau)^{-7/4}
 e^{-\frac{[\,\cdot-\lambda_\alpha(1+\tau)\,]^2}{C(1+\tau)}}\right](x)
 \,d\tau
 \leq C\Psi_i(x,t).
\]
Hence \eqref{eq:ordinary-cross-primitive-bounds}--
\eqref{eq:ordinary-opposite-family-pointwise} meet the hypotheses of
Lemmas~\ref{lem:endpoint-convolution} and \ref{lem:gaussian-extraction}; the
latter is used in its global nonmaterial form, so no bound for
\(b_\alpha'\) is needed.

For the last time unit, \eqref{eq:first-leading-row}--
\eqref{eq:second-leading-row} are combined before absolute values are taken.
The non-common terms require at most \(R_{yy}\) and \(Z_{yy}\), while every
term containing \(Z_{j,y}\) belongs to a common vector.  Thus
\eqref{eq:last-unit-general}--\eqref{eq:last-unit-common} stop the derivative
count at \(Z_{j,yy}\), with no derivative loss.  Together with
\eqref{eq:secondary-mode-complete},
\eqref{eq:offdiagonal-same-speed-complete}, \eqref{eq:Rr-norms}, and
\eqref{eq:singular-action}, this yields
\begin{equation}
 |\partial_x Z_i^{\mathrm{ord}}|
 \leq C\epsilon X_N\log(2+t)\Psi_i(x,t).       \label{eq:ordinary-profile-coefficient}
\end{equation}

For the remaining source column in \eqref{eq:remaining-source-column} and
\eqref{eq:complete-duhamel}, write
\[
 Z^{\#}(x,t)
 :=\int_0^t\int_{\mathbb R}
 \widetilde G(y,\tau;x,t)F_{\mathrm{src}}^{\#}(y,\tau)\,dy\,d\tau,
 \qquad
 \#\in\{\mathrm{bg},\mathrm{coef},\mathrm{visc},\mathrm{nl}\},
\]
with \(Z^{\#}=(Z_1^{\#},Z_2^{\#})^\top\).

For \(F_{\mathrm{src}}^{\mathrm{coef}}\),
\eqref{eq:eta-explicit-profile},
\eqref{eq:burgers-pointwise-bound}, \eqref{eq:resummed-pointwise-bound},
\eqref{eq:phi-explicit-identities}, \eqref{eq:pointwise-continuity-norm},
\eqref{eq:derived-low-rates}, and \eqref{eq:endpoint-weight} give
\begin{align}
 &|\phi_i\eta(R_2-R_1)|(y,\tau)
 \leq C(\epsilon+\epsilon^2)X_N\log(2+\tau)(1+\tau)^{-1/2}\times
 \sum_{\alpha=1}^2
 \bigl((y-\lambda_\alpha(1+\tau))^2+1+\tau\bigr)^{-5/4},\notag\\
 &|\partial_y[\phi_i\eta(R_2-R_1)]|(y,\tau)
 \notag\\
 &\leq C(\epsilon+\epsilon^2)X_N\log(2+\tau)\times\Bigg\{(1+\tau)^{-1}
 \sum_{\alpha=1}^2
 \bigl((y-\lambda_\alpha(1+\tau))^2+1+\tau\bigr)^{-5/4}\notag\\[-1mm]
 &\hspace{55mm}+(1+\tau)^{-3/2}
 \sum_{\alpha=1}^2
 \bigl((y-\lambda_\alpha(1+\tau))^2+1+\tau\bigr)^{-3/4}
 \Bigg\}.
\label{eq:eta-envelope}
\end{align}
In particular, the algebraic derivative tail in \eqref{eq:eta-envelope} is
retained, and \eqref{eq:eta-quarter-hypotheses} holds.  Hence
\eqref{eq:eta-quarter-map}, \eqref{eq:remaining-structured-map},
\eqref{eq:remaining-nonleading-map}, and \eqref{eq:last-unit-common} yield
\begin{equation}
 |\partial_x Z_i^{\mathrm{coef}}|
 \leq C(\epsilon+\epsilon^2)X_N\log(2+t)\Psi_i(x,t).
                                                               \label{eq:eta-response}
\end{equation}

For \(F_{\mathrm{src}}^{\mathrm{visc}}\), use
\eqref{eq:variable-viscosity-divergence} in the form
\[
 F_{\mathrm{src}}^{\mathrm{visc}}:\quad
 \begin{cases}
 \partial_yh+h_0,
   &0\leq\tau\leq(t-1)_+,\\
 \text{the original common-vector source},
   &t-1\leq\tau\leq t.
 \end{cases}
\]
The divergence is transferred only to the long-time regular heat kernel.  By
\eqref{eq:B-explicit-profile},
\eqref{eq:burgers-pointwise-bound}, \eqref{eq:resummed-pointwise-bound},
\eqref{eq:phi-explicit-identities}, \eqref{eq:pointwise-continuity-norm},
\eqref{eq:derived-low-rates}, and \eqref{eq:endpoint-weight}, the two parts in
\eqref{eq:variable-viscosity-divergence} satisfy the \(h\)- and
\(h_0\)-bounds in
Lemma~\ref{lem:remaining-source-maps}, with
\(C_f=C(\epsilon+\epsilon^2)X_N\).  Therefore
\begin{equation}
 |\partial_x Z_i^{\mathrm{visc}}|
 \leq C(\epsilon+\epsilon^2)X_N\log(2+t)\Psi_i(x,t).  \label{eq:viscosity-response}
\end{equation}

For \(F_{\mathrm{src}}^{\mathrm{nl}}\), Taylor's formula
\eqref{eq:exact-N} is applied to the common vector
\[
 F_{\mathrm{src}}^{\mathrm{nl}}
 =\binom{\phi_1N(R)}{\phi_2N(R)}.
\]
Together, \eqref{eq:pointwise-continuity-norm},
\eqref{eq:derived-low-rates}, \eqref{eq:phi-explicit-identities}, and
\eqref{eq:endpoint-weight} give
\begin{equation}
 |\phi_iN(R)|(y,\tau)
 \leq CX_N^2\log^2(2+\tau)(1+\tau)^{-1}
       \{\Psi_1(y,\tau)+\Psi_2(y,\tau)\}.       \label{eq:nonlinear-envelope}
\end{equation}
Strictly between the two tubes,
\begin{align*}
 &\{\Psi_1(y,\tau)+\Psi_2(y,\tau)\}^2
 \leq C(1+\tau)^{-1}\sum_{\alpha=1}^2
 \bigl((y-\lambda_\alpha(1+\tau))^2+1+\tau\bigr)^{-1},\\
& |\phi_iN(R)|(y,\tau)
 \leq CX_N^2\log^2(2+\tau)(1+\tau)^{-1}
 \sum_{\alpha=1}^2
 \bigl((y-\lambda_\alpha(1+\tau))^2+1+\tau\bigr)^{-1}.
\end{align*}
By \eqref{eq:exact-N}, \eqref{eq:derived-low-rates}, and
\eqref{eq:second-third-derivative-rate}, the derivative required on the last
time unit satisfies
\begin{align}
 \|\partial_yN(R)(\tau)\|_\infty
 &\leq C(1+\epsilon)X_N^2\Bigl[
 \log^2(2+\tau)(1+\tau)^{-5/2}
 +\log^{1+\beta_N}(2+\tau)(1+\tau)^{-1-\frac32\beta_N}\Bigr]\notag\\
 &\leq C(1+\epsilon)X_N^2
 \log^2(2+\tau)(1+\tau)^{-2}.
 \label{eq:modal-nonlinear-derivative-bound}
\end{align}
The weak contribution to \(N_y\) is
\(\|R\|_\infty\|R_{xx}\|_\infty\).  Since the two transformed rows form a
common vector, their principal damped Dirac projection is
\begin{equation}
 \begin{pmatrix}\phi_1&0\\0&\phi_2\end{pmatrix}
 P_-
 \begin{pmatrix}\phi_1^{-1}&0\\0&\phi_2^{-1}\end{pmatrix}
 \binom{\phi_1N(R)}{\phi_2N(R)}
 =\begin{pmatrix}\phi_1&0\\0&\phi_2\end{pmatrix}
 P_-\binom{N(R)}{N(R)}=0.
 \label{eq:nonlinear-dirac-cancellation}
\end{equation}
Thus \eqref{eq:nonlinear-envelope},
\eqref{eq:nonlinear-dirac-cancellation},
\eqref{eq:modal-nonlinear-derivative-bound}, and
\eqref{eq:last-unit-common}, together with
Lemma~\ref{lem:remaining-source-maps}, yield
\begin{equation}
 |\partial_x Z_i^{\mathrm{nl}}|
 \leq CX_N^2\log(2+t)\Psi_i(x,t).              \label{eq:nonlinear-response}
\end{equation}

For \(F_{\mathrm{src}}^{\mathrm{bg}}\),
\eqref{eq:exact-background-decomposition} gives the cancellations
\begin{align*}
 &(\theta_2-\theta_1)(\theta_{1,y}+\theta_{2,y})
       +\theta_1\theta_{1,y}-\theta_2\theta_{2,y}
   =\theta_2\theta_{1,y}-\theta_1\theta_{2,y},\\
 &(\theta_2-\theta_1)^3+\theta_1^3-\theta_2^3
   =3\theta_1\theta_2(\theta_1-\theta_2).
\end{align*}
The cubic terms and
\(\phi_i\chi_{\alpha}^{\mathrm{cone}}\Xi_\alpha^2\) were estimated in
Subsection~\ref{section 5.4}; the terms \eqref{eq:V-bg} and
\(d\gamma_i\phi_i\Xi_{i,yy}\) are covered by
Subsection~\ref{appendix 4}.  The remaining pieces are
\(\phi_i(1-\chi_{\alpha}^{\mathrm{cone}})\Xi_\alpha^2\) and
\begin{align}
 S_i^{\mathrm{rem}}:=&-\phi_i\mathcal E_i^{\mathrm{app}}
 -(-1)^{i+1}\mu_{\mathrm{cub}}\phi_i\theta_i^3
 -(-1)^i\left(\frac{cp'''(1)}{3p_2^2}+\frac1{4c}\right)
       \phi_i\theta_{3-i}^3
 -\partial_yV_i^{\mathrm{bg}}-d\gamma_i\phi_i\Xi_{i,yy}\notag\\
 &+\frac12\phi_i\sum_{\alpha=1}^2
 \left\{\chi_\alpha^{\mathrm{cone}}\Xi_\alpha^2
 +(1-\chi_\alpha^{\mathrm{cone}})\Xi_\alpha^2\right\}\notag\\
 =&\phi_i\left[
 \theta_1\left(\Xi_2-\frac{\theta_1^2}{4c}\right)
 +\theta_2\left(\Xi_1+\frac{\theta_2^2}{4c}\right)
 \right]
 +\mathcal R_i^{\mathrm{int}}-d\gamma_i\phi_i\Xi_{i,yy}.
 \notag
\end{align}
By \eqref{eq:burgers-pointwise-bound},
\eqref{eq:resummed-pointwise-bound}, \eqref{eq:profile-cutoff}, and
\eqref{eq:complete-integrable-remainder}, for \(0\leq k\leq2\),
\begin{align}
 \left|\partial_y^kS_i^{\mathrm{rem}}(y,\tau)\right|
 \leq &C_k(\epsilon^2+\epsilon^3+\epsilon^4)
 \Bigg[(1+\tau)^{-1/4-k/2}
 \sum_{\alpha=1}^2
 \frac1{([y-\lambda_\alpha(1+\tau)]^2+1+\tau)^{3/2}}\notag\\
 &+(1+\tau)^{-k/2}
 \prod_{\alpha=1}^2\left(
 \frac1{([y-\lambda_\alpha(1+\tau)]^2+1+\tau)^{3/4}}+\frac1{|y-\lambda_{3-\alpha}(1+\tau)|^{3/2}+1+\tau}\right)
\Bigg].                            \notag
\end{align}
The complementary higher-profile squares satisfy, for \(0\leq k\leq2\),
\begin{align}
 &|\partial_y^k[\phi_i(1-\chi_\alpha^{\mathrm{cone}})\Xi_\alpha^2](y,\tau)|\notag\\
 &\leq C_k\epsilon^4(1+\tau)^{-k/2}
 \left\{
 \frac1{([y-\lambda_\alpha(1+\tau)]^2+1+\tau)^{3/4}}
 +\frac1{|y-\lambda_{3-\alpha}(1+\tau)|^{3/2}+1+\tau}
 \right\}^2.                                             \notag
\end{align}
For \(k=2\), these estimates follow term by term from the fixed fourth-order
profile bounds \eqref{eq:burgers-pointwise-bound}--
\eqref{eq:resummed-pointwise-bound}.  Since
\((S_1^{\mathrm{rem}},S_2^{\mathrm{rem}})^\top\) is not a common vector,
because \(\mathcal R_i^{\mathrm{int}}\) contains \(i\)-dependent terms, no
\(P_-\) cancellation is used on the last time unit.  Hence
\eqref{eq:app-Xi-tail},
\eqref{eq:remaining-divergence-map}--\eqref{eq:remaining-nonleading-map}, and
\eqref{eq:last-unit-general} give
\begin{equation}
 |\partial_x Z_i^{\mathrm{bg}}|
 \leq C(\epsilon+\epsilon^2)\log(2+t)\Psi_i(x,t). \label{eq:background-response}
\end{equation}

\subsection{Consolidated modal and pointwise response}

The source ledger has the four canonical forms
\[
 f_\alpha
 =\frac{b_\alpha(\tau)}{1+\tau}\Gamma_\alpha(\cdot,1+\tau)
   +\partial_yQ_\alpha,
 \qquad
 F_{\mathrm{src}}=\partial_yh+h_0,
 \qquad
 \int_{\mathbb R}\partial_yQ_\alpha\,dy=0,
\]
corresponding, respectively, to a Gaussian mass, a zero-mass primitive, a
divergence, and an integrable remainder.

\begin{lemma}\label{lem:consolidated-source-ledger}
Let \(N\geq5\), let \(0<t\leq t^*\), and assume the bootstrap quantity
\(X_N(t^*)\) \eqref{eq:pointwise-continuity-norm} is finite.  Then
\begin{align}
 &\frac{|a_i(t)|}{\log(2+t)}+(1+t)|a_i'(t)|
 +\sum_{k=0}^2(1+t)^{(k+1)/2}
       \|\partial_x^k Z_i^{\mathrm{rem}}(t)\|_\infty\notag\\
 &\quad\leq C\{E_N+(\epsilon+\epsilon^2)X_N(t^*)+X_N(t^*)^2\},
 \label{eq:modal-output}
\end{align}
where \(a_i'\) \eqref{eq:ai-definition} is understood almost everywhere.  Moreover,
\begin{equation}
 |Z_{i,x}(x,t)|+|\theta_i(x,t)Z_i(x,t)|
 \leq C\{E_N+(\epsilon+\epsilon^2)X_N(t^*)+X_N(t^*)^2\}
       \log(2+t)\Psi_i(x,t).
 \label{eq:weighted-Z-output}
\end{equation}
\end{lemma}

\begin{proof}
For every source in \eqref{eq:mass-primitive-form},
\[
 \frac{b_\alpha(\tau)}{1+\tau}
 =\int_{\mathbb R}f_\alpha(y,\tau)\,dy.
\]
With the signs in \eqref{eq:complete-duhamel}, the Gaussian coefficients
removed by Lemma~\ref{lem:gaussian-extraction} agree term by term with
\eqref{eq:ai-definition}; hence, we have
\[
 Z_i=a_i(t)\Gamma_i(\cdot,1+t)+Z_i^{\mathrm{rem}},
 \qquad
 Z_i^{\mathrm{rem}}\ \text{contains no principal Gaussian mass}.
\]
The source-by-source sizes are
\begingroup
\small
\renewcommand{\arraystretch}{1.25}
\[
\begin{array}{@{}l@{\;}c@{\qquad}l@{\;}c@{\qquad}l@{\;}c@{}}
\text{response} & \text{size}
 & \text{response} & \text{size}
 & \text{response} & \text{size}\\ \hline
\text{initial} & E_N
 & \text{ordinary} & \epsilon X_N
 & \text{exceptional} & \mathcal S_N\\
\text{secondary and off-diagonal} & (\epsilon+\epsilon^2)X_N
 & \text{explicit profile} & \epsilon^3+\epsilon^4
 & \text{coefficient} & (\epsilon+\epsilon^2)X_N\\
\text{variable viscosity} & (\epsilon+\epsilon^2)X_N
 & \text{nonlinear} & X_N^2
 & \text{background} & \epsilon+\epsilon^2
\end{array}
\]
\endgroup
The first four entries follow from \eqref{eq:initial-global}--\eqref{eq:initial-pointwise},
\eqref{eq:ordinary-cross-primitive-bounds}--\eqref{eq:ordinary-profile-coefficient},
\eqref{eq:exceptional-mass-C1}--\eqref{eq:exceptional-pointwise}, and
\eqref{eq:secondary-mode-primitive-bounds}--\eqref{eq:secondary-mode-complete}.
The remaining rows follow from \eqref{eq:critical-mass-and-primitive}--
\eqref{eq:critical-profile-response}, \eqref{eq:eta-envelope}--
\eqref{eq:eta-response}, \eqref{eq:variable-viscosity-divergence}--
\eqref{eq:viscosity-response}, \eqref{eq:nonlinear-envelope}--
\eqref{eq:nonlinear-response}, and \eqref{eq:background-response}.
All propagations are covered by Lemmas~\ref{lem:endpoint-convolution}--
\ref{lem:remaining-source-maps}, including the mesoscopic alternative after
\eqref{eq:integrated-material-identity}, together with
\eqref{eq:singular-action}, \eqref{eq:last-unit-general}, and
\eqref{eq:last-unit-common}.  Thus, with \(X_N=X_N(t^*)\) \eqref{eq:pointwise-continuity-norm}, the total size is
\[
 E_N+\mathcal S_N+(\epsilon+\epsilon^2)
 +C(\epsilon+\epsilon^2)X_N+X_N^2.
\]

{We first estimate \(Z_i^{\mathrm{rem}}\). Using the source sizes listed in the table above, Young’s  inequality, and the kernel and source estimates employed in Lemmas \ref{lem:gaussian-extraction}–\ref{lem:remaining-source-maps}, we reduce the \(L^\infty\) estimates of \(\partial_x^k Z_i^{\mathrm{rem}}\) to the following early- and late-time integrals. Let \(I_k^{\mathrm{early}}\) and \(I_k^{\mathrm{late}}\) denote the corresponding contributions, respectively. For \(0\leq k\leq2\), we have}
\[
\begin{aligned}
I_0^{\rm late}(t)
&\leq
C\int_{t/2}^{(t-1)_+}
(t-\tau)^{-1/2}\log(2+\tau)(1+\tau)^{-3/2}\,d\tau 
\leq C(1+t)^{-1/2},\\
I_1^{\rm late}(t)
&\leq C\log^2(2+t)(1+t)^{-3/2}
\leq C(1+t)^{-1},\\
I_2^{\rm late}(t)
&\leq C\log^2(2+t)(1+t)^{-2}
\leq C(1+t)^{-3/2},
\end{aligned}
\]
while
\[
I_k^{\rm early}(t)
\leq
C(1+t)^{-(k+2)/2}
\int_0^{t/2}(1+\tau)^{-1/2}\,d\tau
\leq C(1+t)^{-(k+1)/2}.
\]
Consequently, it holds
\[
I_k^{\rm early}(t)+I_k^{\rm late}(t)
\leq C(1+t)^{-(k+1)/2},
\qquad 0\leq k\leq2.
\]
Thus the \(Z_i^{\mathrm{rem}}\)-part of \eqref{eq:modal-output} follows, since
\[
 \mathcal S_N
 \leq C\{(\epsilon+\epsilon^2)X_N+X_N^2+E_N\},
\]
\[
 \epsilon X_N^2\leq X_N^2,\qquad
 \epsilon\leq E_N\leq1,\qquad
 \epsilon^2+\epsilon^3+\epsilon^4\leq CE_N.
\]

% The late and early time integrals needed for
% \(0\leq k\leq2\) satisfy, with \(I_k^{\mathrm{late}}\) denoting the
% corresponding \(k\)-th derivative integral,
% \begin{align*}
%  \int_{t/2}^{(t-1)_+}(t-\tau)^{-1/2}
%  \log(2+\tau)(1+\tau)^{-3/2}\,d\tau
%  &\leq C\frac{\log(2+t)}{1+t}
%  \leq C(1+t)^{-1/2},\\
%  I_1^{\mathrm{late}}(t)
%  &\leq C\log^2(2+t)(1+t)^{-3/2}
%  \leq C(1+t)^{-1},\\
%  I_2^{\mathrm{late}}(t)
%  &\leq C\log^2(2+t)(1+t)^{-2}
%  \leq C(1+t)^{-3/2},\\
%  (1+t)^{-(k+2)/2}\int_0^{t/2}(1+\tau)^{-1/2}\,d\tau
%  &\leq C(1+t)^{-(k+1)/2}.
% \end{align*}
% Thus the \(Z_i^{\mathrm{rem}}\)-part of \eqref{eq:modal-output} follows, since
% \[
%  \mathcal S_N
%  \leq C\{(\epsilon+\epsilon^2)X_N+X_N^2+E_N\},
% \]
% \[
%  \epsilon X_N^2\leq X_N^2,\qquad
%  \epsilon\leq E_N\leq1,\qquad
%  \epsilon^2+\epsilon^3+\epsilon^4\leq CE_N.
% \]

For the coefficient \(a_i\), the same-family boundary term from
\eqref{eq:ai-definition} is \(O((1+t)^{-1})\).  For a cutoff cross-family
mass, one has
\begin{align*}
 &\frac{d}{dt}\int_0^{t-1}
 \chi_{\mathrm{time}}\left(\frac{1+\tau}{\sqrt{1+t}}\right)
 \frac{b(\tau)}{1+\tau}\,d\tau
 =O((1+t)^{-1}),\\
 &\operatorname{supp}_{\tau}\partial_t
 \chi_{\mathrm{time}}\left(\frac{1+\tau}{\sqrt{1+t}}\right)
 \subset
 \left\{\sqrt{1+t}\leq1+\tau\leq2\sqrt{1+t}\right\}.
\end{align*}
Together with \eqref{eq:exceptional-mass-C1}, this gives
\[
 (1+t)|a_i'(t)|+\frac{|a_i(t)|}{\log(2+t)}
 \leq C\{E_N+(\epsilon+\epsilon^2)X_N+X_N^2\},
\]
which is the \(a_i\)-part of \eqref{eq:modal-output}.  Finally,
\eqref{eq:receiving-gaussian} and
\eqref{eq:burgers-pointwise-bound} imply
\[
 |\Gamma_{i,x}(x,1+t)|+|\theta_i(x,t)\Gamma_i(x,1+t)|
 \leq C(1+t)^{-1}
 e^{-[x-\lambda_i(1+t)]^2/[C(1+t)]}
 \leq C\Psi_i(x,t).
\]
Therefore
\[
 |\partial_x[a_i(t)\Gamma_i(x,1+t)]|
 +|\theta_i(x,t)a_i(t)\Gamma_i(x,1+t)|
 \leq C\{E_N+(\epsilon+\epsilon^2)X_N+X_N^2\}
 \log(2+t)\Psi_i(x,t).
\]

Summing \eqref{eq:ordinary-profile-coefficient},
\eqref{eq:exceptional-pointwise}, \eqref{eq:secondary-mode-complete},
\eqref{eq:critical-profile-response}, \eqref{eq:background-response},
\eqref{eq:eta-response},
\eqref{eq:viscosity-response}, \eqref{eq:nonlinear-response}, and the
coefficient--Gaussian estimate above
gives \eqref{eq:weighted-Z-output}.
\end{proof}

\subsection{Pointwise mapping}

\begin{proposition}\label{prop:pointwise-mapping}
Let \(N\geq5\), and let a smooth solution on \([0,t^*]\) remain in the fixed
compact state neighborhood.  If \(\epsilon+X_N(t^*)\) is sufficiently small,
then, with every supremum involving \(a_i'\) interpreted as an essential
supremum,
\begin{align}
 &\sup_{0\leq t\leq t^*}\max_{i=1,2}\Bigg\{
 \frac1{\log(2+t)}\sup_{x\in\mathbb R}
       \frac{|Z_{i,x}(x,t)|+|\theta_i(x,t)Z_i(x,t)|}{\Psi_i(x,t)}\notag\\
 &\hspace{25mm}+\frac{|a_i(t)|}{\log(2+t)}+(1+t)|a_i'(t)|
 +\sum_{k=0}^2(1+t)^{(k+1)/2}\|\partial_x^k Z_i^{\mathrm{rem}}(t)\|_\infty\Bigg\}\notag\\
 &\leq
 C_0E_N+C_1(\epsilon+\epsilon^2)X_N(t^*)+C_2X_N(t^*)^2,               \label{eq:pointwise-map-Z}\\
 &\sup_{0\leq t\leq t^*}\max_{i=1,2}
 \frac1{\log(2+t)}\sup_{x\in\mathbb R}
       \frac{|R_i(x,t)|}{\Psi_i(x,t)}
 \leq C_0E_N+C_1(\epsilon+\epsilon^2)X_N(t^*)+C_2X_N(t^*)^2.          \label{eq:pointwise-map-R}
\end{align}
By \eqref{eq:initial-size}, \(\epsilon\leq E_N\), so the right-hand side involves only the initial size
and the linear and quadratic bootstrap terms.
\end{proposition}

\begin{proof}
Estimates \eqref{eq:weighted-Z-output} and \eqref{eq:modal-output} prove
\eqref{eq:pointwise-map-Z}.  The exact recovery identity gives
\[
 R_i=\phi_i^{-1}\left(Z_{i,x}+\frac{\theta_i}{2d}Z_i\right),
 \qquad
 |R_i|
 \leq C\bigl(|Z_{i,x}|+|\theta_iZ_i|\bigr),
 \qquad
 \|\phi_i^{-1}\|_\infty\leq C.
\]
Combining this with \eqref{eq:weighted-Z-output} yields
\eqref{eq:pointwise-map-R}.
\end{proof}

\section{Uniform energy estimate and completion of the proof}
\label{sec:energy}
The uniform top-order norm follows from the original remainder equations.  Set
\begin{align}
 A_R=&c+2F'(u_2^{\mathrm{app}}-u_1^{\mathrm{app}})
 +2B'(u_2^{\mathrm{app}}-u_1^{\mathrm{app}})(u_{1,x}^{\mathrm{app}}+u_{2,x}^{\mathrm{app}}),
 \notag\\
 D_R=&d+B\big((u_2^{\mathrm{app}}-u_1^{\mathrm{app}})+(R_2-R_1)\big)=\frac{d}{v},
 \label{eq:energy-coefficients}
\end{align}
and
\begin{align}
 f=&2\big\{
 F\big((u_2^{\mathrm{app}}-u_1^{\mathrm{app}})+(R_2-R_1)\big)
 -F(u_2^{\mathrm{app}}-u_1^{\mathrm{app}})-F'(u_2^{\mathrm{app}}-u_1^{\mathrm{app}})(R_2-R_1)
 \notag\\
 &\quad+\big[B\big((u_2^{\mathrm{app}}-u_1^{\mathrm{app}})+(R_2-R_1)\big)-B(u_2^{\mathrm{app}}-u_1^{\mathrm{app}})
 \notag\\
 &\hspace{24mm}-B'(u_2^{\mathrm{app}}-u_1^{\mathrm{app}})(R_2-R_1)\big]
 (u_{1,x}^{\mathrm{app}}+u_{2,x}^{\mathrm{app}})\big\}
 -(\mathcal E_1^{\mathrm{app}}+\mathcal E_2^{\mathrm{app}}),
 \notag\\
 g=&\mathcal E_1^{\mathrm{app}}-\mathcal E_2^{\mathrm{app}}.
 \label{eq:energy-sources}
\end{align}
By \eqref{eq:exact-R-equation},
\begin{equation}
 \left\{
 \begin{aligned}
 (R_1+R_2)_t-\big[A_R(R_2-R_1)\big]_x
 -2\big[D_R(R_1+R_2)_x\big]_x&=f_x,\\
 (R_2-R_1)_t-c(R_1+R_2)_x&=g_x.
 \end{aligned}
 \right.
 \label{eq:energy-system}
\end{equation}
The coefficient $D_R=d/v$ is positive as long as $v$ stays in a compact
 subset of $(0,\infty)$.
\subsection{The differential energy inequality}
For \(0\leq k\leq N\), set
\begin{equation*}
 K_k(t)=\left\langle
 (1-\partial_x^2)^{-1}\partial_x^{k+1}(R_2-R_1),
 \partial_x^k(R_1+R_2)
 \right\rangle_{L^2},
\end{equation*}
and, with \(\delta\) chosen in \eqref{eq:delta-choice},
\begin{align}
 \mathcal E_N(t)=&\frac12\sum_{k=0}^N\int_{\mathbb R}
 \left\{c\big[\partial_x^k(R_1+R_2)\big]^2
 +A_R\big[\partial_x^k(R_2-R_1)\big]^2\right\}dx
 -\delta\sum_{k=0}^NK_k(t),\notag\\
 \mathcal D_N(t)=&\sum_{k=0}^N
 \big\|\partial_x^{k+1}(R_1+R_2)(t)\big\|_2^2
 +\sum_{k=0}^N
 \big\|(1-\partial_x^2)^{-1/2}
 \partial_x^{k+1}(R_2-R_1)(t)\big\|_2^2.\notag
\end{align}
For convenience, we denote
\[
 S=R_1+R_2,\qquad D=R_2-R_1,\qquad L=1-\partial_x^2.
\]
Due to $\|L^{-1/2}\partial_x^{k+1}D\|_2\leq\|\partial_x^kD\|_2$, one has $
 \sum_{k=0}^N|K_k(t)|
 \leq C_K\|(R_1,R_2)(t)\|_{H^N}^2$,                        
thus we have
\begin{equation}
 c_E\|(R_1,R_2)(t)\|_{H^N}^2
 \leq\mathcal E_N(t)
 \leq C_E\|(R_1,R_2)(t)\|_{H^N}^2,
 \qquad
 \sum_{k=1}^N\|\partial_x^k(R_2-R_1)(t)\|_2^2
 \leq C\mathcal D_N(t).
 \label{eq:energy-equivalence}
\end{equation}
Choose \(C_G\) larger than all profile Gaussian widths and satisfying
\(16d/C_G\leq1/8\), and set
\begin{equation*}
 \mathcal G_0(t)=\frac1{1+t}\sum_{i=1}^2\int_{\mathbb R}
 \big[(R_1+R_2)^2+(R_2-R_1)^2\big](x,t)
 \exp\left[-\frac{[x-\lambda_i(1+t)]^2}{C_G(1+t)}\right]dx.
\end{equation*}

\begin{lemma}\label{qwe}
Suppose \(N\geq5\), \(v(\mathbb R\times[0,t^*])\subset I_0\), and
\(\epsilon+X_N(t^*)\) is sufficiently small.  Then there exists
\(c_{\mathrm{diss}}>0\), depending only on the fixed coefficients and
\(I_0\), such that
\begin{align}
 \frac{d}{dt}\mathcal E_N+\frac{c_{\mathrm{diss}}}{2}\mathcal D_N
 \leq&C\epsilon\mathcal G_0
 +C[\epsilon+X_N(t^*)]\mathcal D_N
 +C\sum_{i=1}^2\|\mathcal E_i^{\mathrm{app}}(t)\|_{H^{N+1}}^2\notag
 \\
 &+C\left[\epsilon(1+t)^{-3/2}
 +\epsilon^2(1+t)^{-5/4}
 +X_N(t^*)\log(2+t)(1+t)^{-3/2}
 \right.\notag\\[-2mm]
 &\hspace{42mm}\left.
 +X_N(t^*)^2\log^2(2+t)(1+t)^{-2}\right]\mathcal E_N.
 \label{eq:closed-differential-energy}
\end{align}
\end{lemma}

\subsection{Proof of the differential energy inequality.}
In this subsection, we prove  Lemma \ref{qwe}.
\begin{proof}
Applying \(\partial_x^k\) to \eqref{eq:energy-system}, testing the two
equations by \(c\partial_x^kS\) and \(A_R\partial_x^kD\), and integrating
by parts yields
\begin{align}
 &\frac12\frac{d}{dt}\int_{\mathbb R}
 \left\{c\big[\partial_x^k(R_1+R_2)\big]^2
 +A_R\big[\partial_x^k(R_2-R_1)\big]^2\right\}dx
 +2c\int_{\mathbb R}D_R\big[\partial_x^{k+1}(R_1+R_2)\big]^2dx
 \notag\\
 &=I_{1,k}+I_{2,k}+I_{3,k}+I_{4,k},
 \label{eq:basic-energy-identity}
\end{align}
where
\begin{align*}
 I_{1,k}:=&\frac12\int_{\mathbb R}(A_R)_t(\partial_x^kD)^2\,dx
 +kc\int_{\mathbb R}(A_R)_x\partial_x^kS\,\partial_x^kD\,dx+c\sum_{\ell=2}^{k+1}\binom{k+1}{\ell}
 \int_{\mathbb R}(\partial_x^\ell A_R)
 \partial_x^kS\,\partial_x^{k+1-\ell}D\,dx,\\
 I_{2,k}:=&-2c\sum_{\ell=1}^{k}\binom{k}{\ell}
 \int_{\mathbb R}(\partial_x^\ell D_R)
 \partial_x^{k+1-\ell}S\,\partial_x^{k+1}S\,dx,\qquad
 I_{3,k}:=-c\langle\partial_x^{k+1}S,\partial_x^kf\rangle,\\
 I_{4,k}:=&\langle A_R\partial_x^kD,\partial_x^{k+1}g\rangle.
\end{align*}
Differentiating \(K_k\) and using \eqref{eq:energy-system} yields
\begin{align}
 K_k'=I_{5,k}+I_{6,k}+I_{7,k}+I_{8,k},
 \label{eq:kawashima-identity}
\end{align}
where
\begin{align*}
 I_{5,k}:=&-c\|L^{-1/2}\partial_x^{k+1}S\|_2^2
 +\langle L^{-1}\partial_x^{k+1}D,
 A_R\partial_x^{k+1}D\rangle,\\
 I_{6,k}:=&\sum_{\ell=1}^{k+1}\binom{k+1}{\ell}
 \left\langle L^{-1}\partial_x^{k+1}D,
 (\partial_x^\ell A_R)\partial_x^{k+1-\ell}D\right\rangle,\\
 I_{7,k}:=&-2\sum_{\ell=0}^{k}\binom{k}{\ell}
 \left\langle\partial_xL^{-1}\partial_x^{k+1}D,
 (\partial_x^\ell D_R)\partial_x^{k+1-\ell}S\right\rangle,\\
 I_{8,k}:=&-\langle L^{-1}\partial_x^{k+1}g,
 \partial_x^{k+1}S\rangle
 +\langle L^{-1}\partial_x^{k+1}D,\partial_x^{k+1}f\rangle.
\end{align*}
The second term of \(I_{5,k}\) decomposes as
\begin{align}
 &\left\langle (1-\partial_x^2)^{-1}\partial_x^{k+1}(R_2-R_1),
 A_R\partial_x^{k+1}(R_2-R_1)\right\rangle
 \notag\\
=&c\big\|(1-\partial_x^2)^{-1/2}
 \partial_x^{k+1}(R_2-R_1)\big\|_2^2
 +\left\langle (1-\partial_x^2)^{-1}
 \partial_x^{k+1}(R_2-R_1),
 (A_R-c)\partial_x^{k+1}(R_2-R_1)\right\rangle.
 \label{eq:kawashima-positive-term}
\end{align}
\begin{align}
 &\big\langle (1-\partial_x^2)^{-1}\partial_x^{k+1}(R_2-R_1),
 (A_R-c)\partial_x^{k+1}(R_2-R_1)\big\rangle\notag\\
 =&\int_{\mathbb R}(A_R-c)
 \left[(1-\partial_x^2)^{-1}\partial_x^{k+1}(R_2-R_1)\right]^2dx+\int_{\mathbb R}(A_R-c)
 \left[\partial_x(1-\partial_x^2)^{-1}
 \partial_x^{k+1}(R_2-R_1)\right]^2dx\notag\\
 &\quad+\int_{\mathbb R}(A_R)_x
 (1-\partial_x^2)^{-1}\partial_x^{k+1}(R_2-R_1)\times\partial_x(1-\partial_x^2)^{-1}
 \partial_x^{k+1}(R_2-R_1)\,dx.
 \label{eq:kawashima-variable-coefficient}
\end{align}
Summing \eqref{eq:basic-energy-identity} minus  \(\delta\) times \eqref{eq:kawashima-identity} over $k=0,\cdots,N$, we obtain
\begin{align*}
 \frac{d}{dt}\mathcal E_N
&+2c\sum_{k=0}^N\int_{\mathbb R}
 D_R(\partial_x^{k+1}S)^2\,dx
 +\delta c\sum_{k=0}^N
 \|L^{-1/2}\partial_x^{k+1}D\|_2^2\\
&=J_1(t)+I_{\mathrm{coeff}}(t)+I_f(t)+I_{\mathrm{bg}}(t),
\end{align*}
where we split the contributions of $I_{3,k}$ and $I_{8,k}$ between $I_f$ and $I_{\mathrm{bg}}$, and that of $I_{7,k}$ between $J_1$ and $I_{\mathrm{coef}}$. As for $I_{5,k}$, one part contributes to the dissipation term $D_N$, while the other is incorporated into $J_1$, thus we have
\begin{align*}
    &J_1(t):=\delta c\sum_{k=0}^N
 \|L^{-1/2}\partial_x^{k+1}S\|_2^2-\delta\sum_{k=0}^N
 \left\langle L^{-1}\partial_x^{k+1}D,
 (A_R-c)\partial_x^{k+1}D\right\rangle\\
 &\qquad\qquad\;+2\delta\sum_{k=0}^N
 \left\langle\partial_xL^{-1}\partial_x^{k+1}D,
 D_R\partial_x^{k+1}S\right\rangle,\\
 &I_{\mathrm{coeff}}(t):=
 \sum_{k=0}^N\big(I_{1,k}+I_{2,k}-\delta I_{6,k}\big)+2\delta\sum_{k=0}^N\sum_{\ell=1}^{k}\binom{k}{\ell}
 \left\langle\partial_xL^{-1}\partial_x^{k+1}D,
 (\partial_x^\ell D_R)\partial_x^{k+1-\ell}S\right\rangle,\\
  &I_f(t):=-c\sum_{k=0}^N
 \left\langle\partial_x^{k+1}S,
 \partial_x^k(f+\mathcal E_1^{\mathrm{app}}
 +\mathcal E_2^{\mathrm{app}})\right\rangle-\delta\sum_{k=0}^N
 \left\langle L^{-1}\partial_x^{k+1}D,
 \partial_x^{k+1}(f+\mathcal E_1^{\mathrm{app}}
 +\mathcal E_2^{\mathrm{app}})\right\rangle ,\\
  &I_{\mathrm{bg}}(t):=\sum_{k=0}^N I_{4,k}+c\sum_{k=0}^N
 \left\langle\partial_x^{k+1}S,
 \partial_x^k(\mathcal E_1^{\mathrm{app}}
 +\mathcal E_2^{\mathrm{app}})\right\rangle+\delta\sum_{k=0}^N
 \left\langle L^{-1}\partial_x^{k+1}
 (\mathcal E_1^{\mathrm{app}}-\mathcal E_2^{\mathrm{app}}),
 \partial_x^{k+1}S\right\rangle\\
 &\qquad\qquad\;+\delta\sum_{k=0}^N
 \left\langle L^{-1}\partial_x^{k+1}D,
 \partial_x^{k+1}(\mathcal E_1^{\mathrm{app}}
 +\mathcal E_2^{\mathrm{app}})\right\rangle,
\end{align*}
The profile bounds give
\begin{equation}
 \|A_R(t)-c\|_\infty+\|(A_R)_x(t)\|_\infty
 \leq C\left\{\epsilon(1+t)^{-1/2}
       +\epsilon^2(1+t)^{-3/4}\right\}\leq\frac c4.  \label{eq:AR-positivity}
\end{equation}
Let \(I_0=[\underline v,\overline v]\) and
\begin{equation*}
 D_{\min}:=\frac d{\overline v},\qquad D_{\max}:=\frac d{\underline v}.
\end{equation*}
Choose
\begin{equation}
 \delta\leq\min\left\{1,\frac{c}{8C_K},
 \frac{D_{\min}}{4},\frac{c^2D_{\min}}{16(D_{\max})^2}\right\}.             \label{eq:delta-choice}
\end{equation}
By \eqref{eq:AR-positivity}--\eqref{eq:delta-choice},
\begin{align*}
 J_1(t)\leq\frac12\Bigg[
 2cD_{\min}\sum_{k=0}^N\|\partial_x^{k+1}S\|_2^2
 +\delta c\sum_{k=0}^N
 \|L^{-1/2}\partial_x^{k+1}D\|_2^2\Bigg].
\end{align*}
Hence, for some fixed \(c_{\mathrm{diss}}>0\),
\begin{align}
 &2c\sum_{k=0}^N\int_{\mathbb R}
 D_R(\partial_x^{k+1}S)^2\,dx
 +\delta c\sum_{k=0}^N
 \|L^{-1/2}\partial_x^{k+1}D\|_2^2-J_1(t)
 \notag\\
 &\hspace{45mm}\geq c_{\mathrm{diss}}\mathcal D_N .
 \label{eq:fixed-principal-dissipation}
\end{align}
Thus
\begin{equation}\label{qqq}
 \frac{d}{dt}\mathcal E_N+c_{\mathrm{diss}}\mathcal D_N
 \leq |I_{\mathrm{coeff}}(t)|+|I_f(t)|+|I_{\mathrm{bg}}(t)|.
\end{equation}

\subsubsection{Estimate of $I_{\mathrm{coeff}}$} {\bf The first step is to consider
the derivatives of \(A_R\): separation of the Gaussian core and the integrable tail.}
The \(A_R\)-terms of \(I_{\mathrm{coeff}}\) are
\(\sum_{k=0}^N(I_{1,k}-\delta I_{6,k})\).
For \(0\leq k\leq N\), put
\begin{align*}
 \mathcal A_k(x,t):=&\epsilon(1+t)^{-1-k/2}
 \sum_{\alpha=1}^2
 \exp\left[-\frac{[x-\lambda_\alpha(1+t)]^2}{C(1+t)}\right]\\
 &+\epsilon^2(1+t)^{-(k+1)/2}
 \sum_{\alpha=1}^2
 \frac1{([x-\lambda_\alpha(1+t)]^2+1+t)^{3/4}}\\
 &+\epsilon^2(1+t)^{-(k+1)/2}
 \sum_{\alpha=1}^2
 \frac1{|x-\lambda_\alpha(1+t)|^{3/2}+1+t}.
\end{align*}
Then
\begin{equation}
 |\partial_x^k(A_R)_x(x,t)|\leq C_k\mathcal A_k(x,t),
 \qquad 0\leq k\leq N.                                  \label{eq:AR-derivative-bound}
\end{equation}
The coefficient in the first term of \(I_{1,k}\) satisfies
\begin{align}
 (A_R)_t=&2F''(u_2^{\mathrm{app}}-u_1^{\mathrm{app}})(u_{2,t}^{\mathrm{app}}-u_{1,t}^{\mathrm{app}})\notag\\
 &+2B''(u_2^{\mathrm{app}}-u_1^{\mathrm{app}})(u_{2,t}^{\mathrm{app}}-u_{1,t}^{\mathrm{app}})
       (u_{1,x}^{\mathrm{app}}+u_{2,x}^{\mathrm{app}})\notag\\
 &+2B'(u_2^{\mathrm{app}}-u_1^{\mathrm{app}})(u_{1,xt}^{\mathrm{app}}+u_{2,xt}^{\mathrm{app}}).       \label{eq:AR-time-chain}
\end{align}
The profile bounds imply
\begin{equation*}
 |(A_R)_t(x,t)|\leq C\mathcal A_0(x,t).
\end{equation*}
For \(k=0\), the corresponding term is
\(I_{1,0}=\frac12\int_{\mathbb R}(A_R)_tD^2\,dx\), which satisfies
\begin{equation*}
 |I_{1,0}|
 \leq C\epsilon\mathcal G_0
 +C\epsilon^2(1+t)^{-5/4}\mathcal E_N.
\end{equation*}
The endpoint terms are the \(\ell=N+1\) parts of \(I_{1,N}\) and
\(-\delta I_{6,N}\):
\begin{align*}
 &\left|\int_{\mathbb R}(\partial_x^{N+1}A_R)
       \partial_x^N(R_1+R_2)(R_2-R_1)\,dx\right|
 \leq \|\partial_x^{N+1}A_R\|_\infty
       \|\partial_x^N(R_1+R_2)\|_2\|R_2-R_1\|_2,\\
 &\left|
 \left\langle (1-\partial_x^2)^{-1}\partial_x^{N+1}(R_2-R_1),(\partial_x^{N+1}A_R)(R_2-R_1)\right\rangle
\right| \\
 &\leq \|(1-\partial_x^2)^{-1/2}\partial_x^{N+1}(R_2-R_1)\|_2\|(\partial_x^{N+1}A_R)(R_2-R_1)\|_2.
\end{align*}
All other \(A_R\)-terms in \(\sum_{k=0}^N(I_{1,k}-\delta I_{6,k})\) follow from \eqref{eq:AR-derivative-bound} and
\eqref{eq:energy-equivalence} and require at most \(N\)
derivatives of \(S\) and \(D\).

{\bf The second step is to consider the derivatives of \(D_R\) and the viscous compensator.}
The \(D_R\)-part of \(I_{\mathrm{coeff}}\) consists of
\(\sum_{k=0}^NI_{2,k}\) and the
\(\ell\geq1\) part of \(-\delta\sum_{k=0}^NI_{7,k}\). To estimate these terms in $I_{\mathrm{coeff}}$, we first establish the necessary estimates for $D_R$, together with the commutator estimates involving $R_1$ and $R_2$.
By $D_R=\frac d v$, it holds
\begin{equation*}
 D_R-d=d\left(\frac1v-1\right),\qquad
 (D_R)_x=-d\frac{v_x}{v^2},\qquad
 (D_R)_{xx}=2d\frac{v_x^2}{v^3}-d\frac{v_{xx}}{v^2},
\end{equation*}
then we have
\begin{align}
 \|D_R-d\|_\infty
 &\leq C\left[\epsilon(1+t)^{-1/2}
 +\epsilon^2(1+t)^{-3/4}
 +X_N(t^*)\log(2+t)(1+t)^{-1}\right],
 \notag\\
 \|(D_R)_x\|_\infty
 &\leq C\left[\epsilon(1+t)^{-1}
 +\epsilon^2(1+t)^{-5/4}
 +X_N(t^*)\log(2+t)(1+t)^{-3/2}\right],
 \label{eq:viscosity-coefficient-bounds}
\end{align}
and
\begin{align*}
 \|(D_R)_{xx}\|_\infty
 \leq C\Big[&
 \epsilon(1+t)^{-3/2}
 +\epsilon^2(1+t)^{-7/4}\\
 &+X_N(t^*)\log^{\beta_N}(2+t)(1+t)^{-\frac32\beta_N}
 +X_N(t^*)^2\log^2(2+t)(1+t)^{-3}\Big].
\end{align*}
Moreover,
\begin{align}
 &\|D_R-d\|_{H^N}
 \leq C_N\{\epsilon+\|R_2-R_1\|_{H^N}\},                   \label{eq:viscosity-HN}\\
 &[\partial_x^k,D_R](R_1+R_2)_x
 =\sum_{\ell=1}^k\binom{k}{\ell}
 (\partial_x^\ell D_R)\partial_x^{k+1-\ell}(R_1+R_2),
 \notag\\
 &\|[\partial_x^k,D_R](R_1+R_2)_x\|_2
 \leq C_N\Big\{\|(D_R)_x\|_\infty\|R_1+R_2\|_{H^k}
 +\|D_R-d\|_{H^k}\|(R_1+R_2)_x\|_\infty\Big\}.
\label{eq:viscosity-commutator}
\end{align}
For convenience, we fix
$
 \eta_0:=\frac{c_{\mathrm{diss}}}{64}
$. The \(\ell=1\) term of \(I_{2,k}\) satisfies
\begin{equation*}
 \left|2ck\int_{\mathbb R}(D_R)_x
 \partial_x^kS\,\partial_x^{k+1}S\,dx\right|
 \leq\eta_0\mathcal D_N
 +C\|(D_R)_x\|_\infty^2\mathcal E_N.
\end{equation*}
Next, we estimate the terms in $I_{7,k}$ corresponding to $l\geq 1$. It follows from \eqref{eq:viscosity-commutator} that
\begin{align}\label{eq:Kawashima-viscosity-bound}
 &\left|2\delta\sum_{k=0}^N\left\langle
 \partial_xL^{-1}\partial_x^{k+1}D,
 [\partial_x^k,D_R]S_x\right\rangle\right| \notag\\
 \leq&\eta_0\mathcal D_N
 +C\Big[\epsilon^2(1+t)^{-2}
 +\epsilon^4(1+t)^{-5/2}
 +X_N(t^*)^2\log^2(2+t)(1+t)^{-3}\Big]\mathcal E_N .
\end{align}
The remaining \(D_R\)-terms, the \(2\leq\ell\leq k\) parts of
\(I_{2,k}\), follow from
\eqref{eq:viscosity-coefficient-bounds}--\eqref{eq:viscosity-commutator}. Thus, we obtain
\begin{align}
 |I_{\mathrm{coeff}}(t)|&\leq
 8\eta_0\mathcal D_N+C\epsilon\mathcal G_0
 +C[\epsilon+X_N(t^*)]\mathcal D_N+CX_N(t^*)^2\log^2(2+t)(1+t)^{-2}\mathcal E_N\notag\\
 &\qquad+C\Big[\epsilon(1+t)^{-3/2}+\epsilon^2(1+t)^{-5/4}
 +X_N(t^*)\log(2+t)(1+t)^{-3/2}\Big]\mathcal E_N
  .
 \label{eq:coefficient-energy-action}
\end{align}

\subsubsection{Nonlinear sources and background residuals}
We first consider  the estimate of $I_f$.  The quadratic source satisfies
\begin{equation}
 \|f+\mathcal E_1^{\mathrm{app}}+\mathcal E_2^{\mathrm{app}}\|_{H^N}
 \leq C X_N(t^*)\log(2+t)(1+t)^{-1}\mathcal E_N(t)^{1/2}.
 \label{eq:energy-Moser-bound}
\end{equation}
Since
\begin{equation*}
 \|L^{-1/2}\partial_x^{k+1}(f+\mathcal E_1^{\mathrm{app}}
 +\mathcal E_2^{\mathrm{app}})\|_2
 \leq\|\partial_x^k(f+\mathcal E_1^{\mathrm{app}}
 +\mathcal E_2^{\mathrm{app}})\|_2,
\end{equation*}
the compensated source term loses no derivative.  Hence, we have
\begin{align}
 |I_f(t)|=&\Bigg|-c\sum_{k=0}^N\left\langle\partial_x^{k+1}(R_1+R_2),
 \partial_x^k(f+\mathcal E_1^{\mathrm{app}}+\mathcal E_2^{\mathrm{app}})\right\rangle\notag\\
 &\quad-\delta\sum_{k=0}^N\left\langle
 (1-\partial_x^2)^{-1}\partial_x^{k+1}(R_2-R_1),
 \partial_x^{k+1}(f+\mathcal E_1^{\mathrm{app}}
 +\mathcal E_2^{\mathrm{app}})\right\rangle\Bigg|\notag\\
 \leq&\eta_0\mathcal D_N
 +CX_N(t^*)^2\log^2(2+t)(1+t)^{-2}\mathcal E_N.
 \label{eq:full-quadratic-action}
\end{align}

Now, we consider the estimate \(I_{\mathrm{bg}}\).   Only the $I_{4,k}$-part of \(I_{\mathrm{bg}}\) requires a redistribution of derivatives
\begin{align}
 &\sum_{k=0}^N I_{4,k}=-\sum_{k=0}^{N-1}\Big\{
 \left\langle A_R\partial_x^{k+1}(R_2-R_1),
 \partial_x^k(\mathcal E_1^{\mathrm{app}}-\mathcal E_2^{\mathrm{app}})\right\rangle+\left\langle(A_R)_x\partial_x^k(R_2-R_1),
 \partial_x^k(\mathcal E_1^{\mathrm{app}}-\mathcal E_2^{\mathrm{app}})\right\rangle\Big\}\notag\\
 &\qquad\qquad\qquad+\left\langle A_R\partial_x^N(R_2-R_1),
 \partial_x^{N+1}(\mathcal E_1^{\mathrm{app}}-\mathcal E_2^{\mathrm{app}})\right\rangle.
 \label{eq:top-background-pair}
\end{align}
The other three terms follow from Young's inequality and
\eqref{eq:energy-equivalence}.  Hence
\begin{align}
 |I_{\mathrm{bg}}(t)|\leq&
 \eta_0\mathcal D_N+C\epsilon\mathcal G_0
 +C\sum_{i=1}^2\|\mathcal E_i^{\mathrm{app}}\|_{H^{N+1}}^2 \notag\\
& +C\Big[
 \epsilon(1+t)^{-3/2}+\epsilon^2(1+t)^{-5/4}
 +X_N(t^*)\log(2+t)(1+t)^{-3/2}\Big]\mathcal E_N
  .
 \label{eq:background-energy-action}
\end{align}

\subsubsection{Conclusion}
  Combining \eqref{qqq},
\eqref{eq:coefficient-energy-action}, \eqref{eq:full-quadratic-action}, and
\eqref{eq:background-energy-action}, we prove  Lemma \ref{qwe}.
\end{proof}

\subsection{The cone Gaussian estimate}  In this subsection, we deal with $\mathcal G_0(\tau)$ in \eqref{eq:closed-differential-energy}.

\begin{lemma}
Under the assumptions of Lemma \ref{qwe}, it holds
\begin{align}
 \int_0^t\mathcal G_0(\tau)\,d\tau
 \leq&C\mathcal E_N(0)+C\int_0^t\mathcal D_N(\tau)\,d\tau
 +C\int_0^t\sum_{i=1}^2
 \|\mathcal E_i^{\mathrm{app}}(\tau)\|_{H^{N+1}}^2d\tau
 \notag\\
 &+C\int_0^t\left[\epsilon(1+\tau)^{-3/2}
 +\epsilon^2(1+\tau)^{-5/4}
 +X_N(t^*)\log(2+\tau)(1+\tau)^{-3/2}
 \right.\notag\\[-2mm]
 &\hspace{37mm}\left.
 +X_N(t^*)^2\log^2(2+\tau)(1+\tau)^{-2}\right]
 \mathcal E_N(\tau)\,d\tau.
\label{eq:gaussian-energy-bound}
\end{align}
\end{lemma}

\begin{proof}
With the \(C_G\) fixed above, define
\begin{align*}
 \omega_1(x,t)&=(1+t)^{-1/2}
 \int_{-\infty}^{x-c(1+t)}
 \exp\left[-\frac{y^2}{2C_G(1+t)}\right]dy,\\
 \omega_2(x,t)&=-(1+t)^{-1/2}
 \int_{x+c(1+t)}^\infty
 \exp\left[-\frac{y^2}{2C_G(1+t)}\right]dy.
\end{align*}
Then
\begin{equation*}
 (\omega_i)_x=(1+t)^{-1/2}
 \exp\left[-\frac{[x-\lambda_i(1+t)]^2}{2C_G(1+t)}\right],
 \qquad
 \frac{2}{C_G}(\partial_t+\lambda_i\partial_x)\omega_i=(\omega_i)_{xx}.
\end{equation*}
The Gaussian-weight inequality of Huang-Li-Matsumura \cite{HuangLiMatsumura2010} gives
\begin{align}
 &\int_0^t\frac1{1+\tau}\int_{\mathbb R}h(x,\tau)^2
 \exp\left[-\frac{[x-\lambda_i(1+\tau)]^2}{C_G(1+\tau)}\right]dx\,d\tau\notag
 \\
 &\leq4\pi\|h(0)\|_2^2
 +8\pi C_G\int_0^t\|h_x(\tau)\|_2^2\,d\tau
 +\frac4{C_G}\int_0^t
 \left\langle(\partial_\tau+\lambda_i\partial_x)h,
 h(\omega_i)^2\right\rangle_{H^{-1}\times H^1}d\tau.
\label{eq:HLM-inequality}
\end{align}
By \eqref{eq:energy-system}, one has
\begin{align}
 (\partial_t+\lambda_i\partial_x)(R_1+R_2)
 &=\partial_x\big[A_R(R_2-R_1)+2D_R(R_1+R_2)_x
 +f+\lambda_i(R_1+R_2)\big],\label{aaaa}\\
 (\partial_t+\lambda_i\partial_x)(R_2-R_1)
 &=\partial_x\big[c(R_1+R_2)+g+\lambda_i(R_2-R_1)\big].\label{zzzz}
\end{align}
For \(h=R_1+R_2\) and \(h=R_2-R_1\), taking \eqref{aaaa} and \eqref{zzzz} into \eqref{eq:HLM-inequality}, the duality term 
$$\left\langle(\partial_\tau+\lambda_i\partial_x)h,
 h(\omega_i)^2\right\rangle_{H^{-1}\times H^1}$$ equals
\begin{align}
 -\int_{\mathbb R}&
 \big[A_R(R_2-R_1)+2D_R(R_1+R_2)_x+f+\lambda_i(R_1+R_2)\big]\notag\\[-1mm]
 &\quad\times
 \big[(R_1+R_2)_x(\omega_i)^2
      +2(R_1+R_2)\omega_i(\omega_i)_x\big]\,dx\label{ppp}\\
 -\int_{\mathbb R}&
 \big[c(R_1+R_2)+g+\lambda_i(R_2-R_1)\big]\notag\\[-1mm]
 &\quad\times
 \big[(R_2-R_1)_x(\omega_i)^2
      +2(R_2-R_1)\omega_i(\omega_i)_x\big]\,dx.
\label{eq:HLM-complete-duality}
\end{align}
Writing \(A_R=c+(A_R-c)\) and \(D_R=d+(D_R-d)\), then \eqref{ppp}+\eqref{eq:HLM-complete-duality} equals
\begin{align}
 &-\int_{\mathbb R}\{2c(R_1+R_2)(R_2-R_1)
 +\lambda_i[(R_1+R_2)^2+(R_2-R_1)^2]\}\omega_i(\omega_i)_x\,dx\notag\\
 &-2d\int_{\mathbb R}(R_1+R_2)_x
 [(R_1+R_2)_x(\omega_i)^2+2(R_1+R_2)\omega_i(\omega_i)_x]dx\notag\\
 &-\int_{\mathbb R}(A_R-c)(R_2-R_1)
 [(R_1+R_2)_x(\omega_i)^2+2(R_1+R_2)\omega_i(\omega_i)_x]dx\notag\\
 &-2\int_{\mathbb R}(D_R-d)(R_1+R_2)_x
 [(R_1+R_2)_x(\omega_i)^2+2(R_1+R_2)\omega_i(\omega_i)_x]dx\notag\\
 &-\int_{\mathbb R}f[(R_1+R_2)_x(\omega_i)^2
 +2(R_1+R_2)\omega_i(\omega_i)_x]dx\notag\\
 &-\int_{\mathbb R}g[(R_2-R_1)_x(\omega_i)^2
 +2(R_2-R_1)\omega_i(\omega_i)_x]dx.
\label{eq:HLM-full-split}
\end{align}
For \(i=1,2\), the first line is
\begin{equation}
 -c\int_{\mathbb R}[2R_2]^2\omega_1(\omega_1)_x\,dx\leq0,
 \qquad
 c\int_{\mathbb R}[2R_1]^2\omega_2(\omega_2)_x\,dx\leq0.
\label{eq:signed-transport-two-families}
\end{equation}
The constant-viscosity term satisfies
\begin{align}
 &-2d\int_{\mathbb R}(R_1+R_2)_x
 \big\{(R_1+R_2)_x(\omega_i)^2
 +2(R_1+R_2)\omega_i(\omega_i)_x\big\}dx \notag
 \\
 &\leq-d\int_{\mathbb R}(R_1+R_2)_x^2(\omega_i)^2dx
 +\frac{4d}{1+t}\int_{\mathbb R}(R_1+R_2)^2
 \exp\left[-\frac{[x-\lambda_i(1+t)]^2}{C_G(1+t)}\right]dx.
\label{eq:HLM-viscosity}
\end{align}
Recall the definition of $\eta$ \eqref{eq:def-eta}, since \(A_R-c=2(\theta_1-\theta_2+\eta)\), and 
\begin{equation*}
 |\theta_\alpha(x,t)|\leq C\epsilon(1+t)^{-1/2}
 \exp\left[-\frac{[x-\lambda_\alpha(1+t)]^2}{C_G(1+t)}\right],
 \qquad \alpha=1,2,
\end{equation*}
then the two \(\theta_\alpha\)-products satisfy
\begin{align}
 &2\sum_{i,\alpha=1}^2\left|\int_{\mathbb R}
 \theta_\alpha(R_2-R_1)(R_1+R_2)_x(\omega_i)^2dx\right|
 \leq\frac d{16}\sum_{i=1}^2\int_{\mathbb R}
 (R_1+R_2)_x^2(\omega_i)^2dx+C\epsilon^2\mathcal G_0,\notag\\
 &4\sum_{i,\alpha=1}^2\left|\int_{\mathbb R}
 \theta_\alpha(R_2-R_1)(R_1+R_2)\omega_i(\omega_i)_xdx\right|
 \leq C\epsilon\mathcal G_0.
 \label{eq:HLM-theta-products}
\end{align}
The two \(\eta\)-products satisfy
\begin{align}
 &2\sum_{i=1}^2\left|\int_{\mathbb R}
 \eta(R_2-R_1)(R_1+R_2)_x(\omega_i)^2dx\right|
 \leq\frac d{16}\sum_{i=1}^2\int_{\mathbb R}
 (R_1+R_2)_x^2(\omega_i)^2dx+C\|\eta\|_\infty^2\mathcal E_N,\notag\\
 &4\sum_{i=1}^2\left|\int_{\mathbb R}
 \eta(R_2-R_1)(R_1+R_2)\omega_i(\omega_i)_xdx\right|
 \leq C(1+t)^{-1/2}\|\eta\|_\infty\mathcal E_N,\notag\\
 &C\{\|\eta\|_\infty^2+(1+t)^{-1/2}\|\eta\|_\infty\}\mathcal E_N
 \leq C\{\epsilon(1+t)^{-3/2}
 +\epsilon^2(1+t)^{-5/4}\}\mathcal E_N.
 \label{eq:HLM-eta-products}
\end{align}
Consequently, one has
\begin{align}
 &\sum_{i=1}^2\left|\int_{\mathbb R}(A_R-c)(R_2-R_1)
 \big[(R_1+R_2)_x(\omega_i)^2
 +2(R_1+R_2)\omega_i(\omega_i)_x\big]dx\right|\notag\\
 &\leq\frac d8\sum_{i=1}^2
       \int_{\mathbb R}(R_1+R_2)_x^2(\omega_i)^2dx
 +C\epsilon\mathcal G_0+C\{\epsilon(1+t)^{-3/2}
          +\epsilon^2(1+t)^{-5/4}\}\mathcal E_N .
 \label{eq:HLM-AR-remainder}
\end{align}
By \eqref{eq:viscosity-HN}, the variable-viscosity terms obey
\begin{align}
 &\sum_{i=1}^2\left|\int_{\mathbb R}2(D_R-d)(R_1+R_2)_x
 \big[(R_1+R_2)_x(\omega_i)^2
 +2(R_1+R_2)\omega_i(\omega_i)_x\big]dx\right|\notag\\
 &\leq C[\epsilon+X_N(t^*)]\mathcal D_N
 +\frac1{32}\mathcal G_0,
 \label{eq:HLM-viscosity-remainder}
\end{align}
where we have used
\begin{align}
 &\sum_{i=1}^2\left|\int_{\mathbb R}2(D_R-d)(R_1+R_2)_x
 \big[(R_1+R_2)_x(\omega_i)^2
 +2(R_1+R_2)\omega_i(\omega_i)_x\big]dx\right|\notag\\
 &\quad\leq C\|D_R-d\|_\infty\mathcal D_N
 +\frac1{32}\mathcal G_0
 +C\|D_R-d\|_\infty^2\mathcal D_N.
 \label{eq:HLM-viscosity-two-products}
\end{align}
For the two forcing rows,
\begin{align}
 &\sum_{i=1}^2\left|\int_{\mathbb R}f
 \big[(R_1+R_2)_x(\omega_i)^2
 +2(R_1+R_2)\omega_i(\omega_i)_x\big]dx\right|\notag\\
 &\leq\frac1{32}\{\mathcal D_N+\mathcal G_0\}
 +C X_N(t^*)^2\log^2(2+t)(1+t)^{-2}\mathcal E_N+C\|\mathcal E_1^{\mathrm{app}}+\mathcal E_2^{\mathrm{app}}\|_{H^{N+1}}^2,\notag\\
 &\sum_{i=1}^2\left|\int_{\mathbb R}g
 \big[(R_2-R_1)_x(\omega_i)^2
 +2(R_2-R_1)\omega_i(\omega_i)_x\big]dx\right|\leq\frac1{32}\{\mathcal D_N+\mathcal G_0\}
 +C\|\mathcal E_1^{\mathrm{app}}-\mathcal E_2^{\mathrm{app}}\|_{H^{N+1}}^2.
 \label{eq:HLM-g-remainder}
\end{align}
The two products in the \(f\)-row are
\begin{align}
 &\left|\int_{\mathbb R} f(R_1+R_2)_x(\omega_i)^2dx\right|
 \leq\frac1{128}\|(R_1+R_2)_x\|_2^2+C\|f\|_2^2,\notag\\
& 2\left|\int_{\mathbb R}f(R_1+R_2)\omega_i(\omega_i)_x\,dx\right|
 \leq\frac1{128(1+t)}\int_{\mathbb R}(R_1+R_2)^2\times
 \exp\left[-\frac{[x-\lambda_i(1+t)]^2}{C_G(1+t)}\right]dx
 +C\|f\|_2^2,\notag\\
 &\|f\|_2^2\leq CX_N(t^*)^2\log^2(2+t)(1+t)^{-2}\mathcal E_N
 +C\|\mathcal E_1^{\mathrm{app}}+\mathcal E_2^{\mathrm{app}}\|_2^2.
 \label{eq:HLM-f-two-products}
\end{align}
The two products in the \(g\)-row are
\begin{align}
 \left|\int_{\mathbb R}g(R_2-R_1)_x(\omega_i)^2dx\right|
 &\leq\frac1{128}\|(R_2-R_1)_x\|_2^2+C\|g\|_2^2,\notag\\
 2\left|\int_{\mathbb R}g(R_2-R_1)\omega_i(\omega_i)_x\,dx\right|
 &\leq\frac1{128(1+t)}\int_{\mathbb R}(R_2-R_1)^2\times
 \exp\left[-\frac{[x-\lambda_i(1+t)]^2}{C_G(1+t)}\right]dx
 +C\|g\|_2^2.
 \label{eq:HLM-g-two-products}
\end{align}
Fix the Young constants in
\eqref{eq:HLM-viscosity-remainder}--\eqref{eq:HLM-g-remainder}, and then
choose \(\epsilon+X_N(t^*)\) sufficiently small.  The terms in
\eqref{eq:HLM-viscosity} and the \(1/32\)-terms are absorbed.  Summing the
two families in \eqref{eq:HLM-inequality} and using
\eqref{eq:signed-transport-two-families}--\eqref{eq:HLM-g-remainder} proves
\eqref{eq:gaussian-energy-bound}.
\end{proof}

\subsection{The background residual} Finally, we need to deal with $\sum_{i=1}^2 \|\mathcal E_i^{\text{app}} \|^2_{H^{N+1}}$ in \eqref{eq:closed-differential-energy}

\begin{lemma}
Under the assumptions of Lemma \ref{qwe}, it holds
\begin{align}
 &\sum_{i=1}^2\left\{
 \int_0^\infty\|\mathcal E_i^{\mathrm{app}}(\tau)\|_{H^{N+1}}d\tau
 +\left(\int_0^\infty
 \|\mathcal E_i^{\mathrm{app}}(\tau)\|_{H^{N+1}}^2d\tau\right)^{1/2}\right\}
 \leq C_N\epsilon,\notag\\
 &\int_0^\infty\sum_{i=1}^2
 \|\mathcal E_i^{\mathrm{app}}(\tau)\|_{H^{N+1}}^2\,d\tau
 \leq C_N\epsilon^2.
\label{eq:background-HN1-integral}
\end{align}
\end{lemma}

\begin{proof}
From \eqref{eq:V-bg}, \eqref{eq:exact-background-decomposition}, and
the profile estimates give
\begin{align}
 &\|\phi_i\theta_i^3\|_{H^{N+1}}
 +\sum_{\alpha\ne i}\|\phi_i\theta_\alpha^3\|_{H^{N+1}}
 +\sum_{\alpha=1}^2\|\phi_i\Xi_\alpha^2\|_{H^{N+1}}
 \leq C_N\epsilon^3(1+t)^{-5/4},\notag\\
 &\left\|\phi_i\left[
 \theta_1\left(\Xi_2-\frac{\theta_1^2}{4c}\right)
 +\theta_2\left(\Xi_1+\frac{\theta_2^2}{4c}\right)
 \right]\right\|_{H^{N+1}}
 \leq C_N\epsilon^3(1+t)^{-3/2}.
 \label{eq:background-explicit-products}
\end{align}
The derivative range used in the Sobolev--Moser estimates is
\begin{align}
& \|\partial_x^j\theta_\alpha(t)\|_\infty
 \leq C_j\epsilon(1+t)^{-(j+1)/2},\qquad\qquad
 \|\partial_x^j\theta_\alpha(t)\|_2
 \leq C_j\epsilon(1+t)^{-j/2-1/4},\notag\\
& \|\partial_x^j\Xi_\alpha(t)\|_\infty
 \leq C_j\epsilon^2(1+t)^{-j/2-3/4},\notag\\
 &\|\partial_x^j\Xi_\alpha(t)\|_2
 \leq C_j\epsilon^2(1+t)^{-j/2-1/2},\hspace{15mm}0\leq j\leq N+3,\notag\\
 &\|\phi_i(t)\|_\infty\leq C,
\qquad\qquad \|\partial_x^j\phi_i(t)\|_\infty
 \leq C_j\epsilon(1+t)^{-j/2},\qquad1\leq j\leq N+3.
 \label{eq:background-profile-derivative-range}
\end{align}
Recall the definition of $V_i^{\mathrm{bg}}$ \eqref{eq:V-bg} and $\mathcal R_i^{\mathrm{int}}$ \eqref{eq:complete-integrable-remainder}, the termwise application gives
\begin{align}
 &\|\partial_xV_i^{\mathrm{bg}}(t)\|_{H^{N+1}}
 \leq C_N(\epsilon+\epsilon^2)(1+t)^{-5/4},
 \label{eq:Vbg-high-order}\\
 &\|\mathcal R_i^{\mathrm{int}}(t)\|_{H^{N+1}}
 \leq C_N\epsilon^2(1+t)^{-3/2}
 +C_N(\epsilon^2+\epsilon^4)(1+t)^{-7/4}
 +C_N\epsilon^2e^{-\kappa_0(1+t)}.
 \label{eq:Rint-high-order}
\end{align}
Thus
\begin{align*}
 \|\mathcal E_i^{\mathrm{app}}(t)\|_{H^{N+1}}
 \leq C_N\Big[&
 (\epsilon+\epsilon^2)(1+t)^{-5/4}
 +\epsilon^2(1+t)^{-3/2}+(\epsilon^2+\epsilon^4)(1+t)^{-7/4}
 +\epsilon^2e^{-\kappa_0(1+t)}\Big].
\end{align*}
Integration proves \eqref{eq:background-HN1-integral}.
\end{proof}

\subsection{Uniform energy bound and completion of the proof}

\subsubsection{Uniform energy estimate}

\begin{proposition}
\label{prop:uniform-energy}
Suppose \(N\geq5\),
\(v(\mathbb R\times[0,t^*])\subset I_0\), and
\(\epsilon+X_N(t^*)\leq\rho_E\).  Then, for \(0\leq t\leq t^*\),
\begin{equation}
 \sup_{0\leq\tau\leq t}\mathcal E_N(\tau)
 +c_E\int_0^t\big[\mathcal D_N(\tau)+\mathcal G_0(\tau)\big]d\tau
 \leq C_E\left[\mathcal E_N(0)
 +\int_0^t\sum_{i=1}^2
 \|\mathcal E_i^{\mathrm{app}}(\tau)\|_{H^{N+1}}^2d\tau\right].
 \label{eq:uniform-energy-bound}
\end{equation}
Moreover,
\begin{equation}
 \mathcal E_N(0)\leq C E_N^2,
 \label{eq:initial-energy-input}
\end{equation}
\begin{equation}
 \int_0^\infty\sum_{i=1}^2
 \|\mathcal E_i^{\mathrm{app}}(\tau)\|_{H^{N+1}}^2\,d\tau
 \leq C\epsilon^2\leq CE_N^2,
 \label{eq:residual-energy-input}
\end{equation}
and
\begin{equation}
 \sup_{0\leq t\leq t^*}\|(R_1,R_2)(t)\|_{H^N}\leq C E_N.
 \label{eq:uniform-HN-bound}
\end{equation}
\end{proposition}

\begin{proof}
Let \(b_{t^*}(t)\) be the coefficient multiplying \(\mathcal E_N\) in
\eqref{eq:closed-differential-energy}.  Then
\begin{equation}
 \|b_{t^*}\|_{L^1(0,t^*)}
 \leq C\{\epsilon+\epsilon^2+X_N(t^*)+X_N(t^*)^2\}.
 \label{eq:energy-Gronwall-coefficient}
\end{equation}
Integrating \eqref{eq:closed-differential-energy} gives
\begin{align}
 \mathcal E_N(t)+c_0\int_0^t\mathcal D_N(\tau)\,d\tau
 \leq&\mathcal E_N(0)+C_0\epsilon\int_0^t\mathcal G_0(\tau)\,d\tau
 +C_0[\epsilon+X_N(t^*)]\int_0^t\mathcal D_N(\tau)\,d\tau\notag\\
 &+C_0\int_0^tb_{t^*}(\tau)\mathcal E_N(\tau)\,d\tau
 +C_0\int_0^t\sum_i
       \|\mathcal E_i^{\mathrm{app}}(\tau)\|_{H^{N+1}}^2\,d\tau.
 \label{eq:integrated-energy-for-absorption}
\end{align}
Likewise, \eqref{eq:gaussian-energy-bound} gives
\begin{align}
 \int_0^t\mathcal G_0(\tau)\,d\tau
 \leq C_H\Bigg[&\mathcal E_N(0)+\int_0^t\mathcal D_N(\tau)\,d\tau
 +\int_0^tb_{t^*}(\tau)\mathcal E_N(\tau)\,d\tau+\int_0^t\sum_i
       \|\mathcal E_i^{\mathrm{app}}(\tau)\|_{H^{N+1}}^2\,d\tau\Bigg].
 \label{eq:gaussian-bound-for-absorption}
\end{align}
Choose \(\kappa>0\) and \(\rho_E>0\) so that
\begin{equation*}
 \kappa C_H\leq\frac{c_0}{4},\qquad
 C_0[\epsilon+X_N(t^*)]\leq\frac{c_0}{4},\qquad
 C_0\epsilon\leq\frac{\kappa}{2}.
\end{equation*}
Move the \(C_H\int_0^t\mathcal D_N\)-term in
\eqref{eq:gaussian-bound-for-absorption} to the left, multiply the
resulting inequality by \(\kappa\), and add it to
\eqref{eq:integrated-energy-for-absorption}.  This yields
\begin{align*}
& \mathcal E_N(t)+\frac{c_0}{2}\int_0^t\mathcal D_N(\tau)\,d\tau
 +\frac{\kappa}{2}\int_0^t\mathcal G_0(\tau)\,d\tau
 \\
 \leq & C\mathcal E_N(0)
 +C\int_0^tb_{t^*}(\tau)\mathcal E_N(\tau)\,d\tau+C\int_0^t\sum_i\|\mathcal E_i^{\mathrm{app}}(\tau)\|_{H^{N+1}}^2\,d\tau.
\end{align*}
Gronwall's inequality proves \eqref{eq:uniform-energy-bound}.
Thus, \eqref{eq:initial-remainder-explicit},
\eqref{eq:AR-positivity}, and \eqref{eq:energy-equivalence} give
\eqref{eq:initial-energy-input}; \eqref{eq:background-HN1-integral} gives
\eqref{eq:residual-energy-input}.  Finally,
\eqref{eq:energy-equivalence} proves \eqref{eq:uniform-HN-bound}.
\end{proof}

\subsubsection{Bootstrap and continuation,}

At \(t=0\), \eqref{eq:initial-remainder-explicit} and
\(\int_{\mathbb R}R_i(x,0)\,dx=0\) give
\begin{equation*}
 W_i(x,0)=\int_{-\infty}^xR_i(y,0)\,dy
          =-\int_x^\infty R_i(y,0)\,dy .
\end{equation*}
The weighted initial norm gives
\begin{align*}
& |R_i(x,0)|\leq C E_N\langle x\rangle^{-2}
             \leq C E_N\Psi_i(x,0),\\
 &\sum_{k=0}^2\|\partial_x^k[\phi_i(\cdot,0)W_i(\cdot,0)]\|_\infty
 +\|(R_1,R_2)(0)\|_{H^N}\leq C E_N.
\end{align*}
Hence local persistence supplies \(0<T_0\leq1\) such that
\begin{equation}
 v(\mathbb R\times[0,T_0])\subset\operatorname{int}I_0,
 \qquad X_N(T_0)\leq2C_{\mathrm{in}}E_N.
 \label{eq:bootstrap-local-start}
\end{equation}
Propositions \ref{prop:pointwise-mapping} and
\ref{prop:uniform-energy}, together with
\eqref{eq:recover-R}--\eqref{eq:recover-Rx}, give
\begin{equation}
 X_N(T)\leq C_{\mathrm{boot}}\left\{E_N+(\epsilon+\epsilon^2)X_N(T)
                    +X_N(T)^2\right\}
 \label{eq:final-continuity-inequality}
\end{equation}
whenever \(T<T_{\max}\),
\(v(\mathbb R\times[0,T])\subset I_0\), and
\(\epsilon+X_N(T)\leq\rho_{\mathrm{boot}}\).  Choose the data so that
\begin{equation}
 \epsilon+2C_{\mathrm{boot}}E_N<\rho_{\mathrm{boot}},
 \qquad
 2C_{\mathrm{boot}}(\epsilon+\epsilon^2)+4C_{\mathrm{boot}}^2E_N\leq\frac12,
 \qquad
 C_vE_N<\rho.
 \label{eq:bootstrap-data-smallness}
\end{equation}
Define
\begin{equation}
 T_b:=\sup\left\{T<T_{\max}:
 v(\mathbb R\times[0,T])\subset\operatorname{int}I_0,\quad
 X_N(T)\leq2C_{\mathrm{boot}}E_N\right\}.
 \label{eq:bootstrap-first-exit}
\end{equation}
For every \(T<T_b\), \eqref{eq:final-continuity-inequality} and
\eqref{eq:bootstrap-data-smallness} yield
\begin{equation}
 X_N(T)\leq\frac32C_{\mathrm{boot}}E_N.
 \label{eq:strict-continuity-improvement}
\end{equation}
Moreover,
\begin{align}
 \sup_{0\leq t\leq T}\|v(t)-1\|_\infty
 &\leq C_v\sup_{0\leq t\leq T}\left[
 \epsilon(1+t)^{-1/2}+\epsilon^2(1+t)^{-3/4}
 +X_N(T)\frac{\log(2+t)}{1+t}\right]\notag\\
 &\leq C_vE_N<\rho.
 \label{eq:bootstrap-strict-range}
\end{align}
Thus \(T_b=T_{\max}\).  Finally,
\eqref{eq:uniform-HN-bound} and the inverse transformation give
\begin{equation}
 \sup_{0\leq t<T_{\max}}\|(v-1,u)(t)\|_{H^N}\leq CE_N,
 \qquad
 \int_0^{T_{\max}}\|u_x(t)\|_{H^N}^2\,dt\leq CE_N^2.
 \label{eq:continuation-controlled-inputs}
\end{equation}
The strict range estimate and the last bound exclude both continuation
obstructions, so \(T_{\max}=\infty\).  Proposition
\ref{prop:pointwise-mapping} and
\eqref{eq:recover-R}--\eqref{eq:recover-Rx} prove
\eqref{eq:main-pointwise-estimate}--\eqref{eq:main-uniform-estimates}, while
\eqref{eq:uniform-HN-bound} proves \eqref{eq:main-energy-estimate}.  This
completes the proof of Theorem~\ref{thm:main-endpoint}.

\appendix
\section{The five-region endpoint convolutions}\label{app:five-regions}

This appendix proves the pointwise convolution estimates used
in Section~\ref{sec:pointwise}.  Here $C_f$ denotes the source size in
Lemma~\ref{lem:endpoint-convolution}.

\subsection{Elementary Gaussian--algebraic convolutions}\label{mmm}

For \(t-\tau\geq1\), differentiation of the convected heat kernel gives
\begin{align}
 |\partial_xg_i(x-y,t-\tau)|
 &\leq C(t-\tau)^{-1}
 \exp\left[-\frac{(x-y-\lambda_i(t-\tau))^2}{C(t-\tau)}\right],\notag\\
 \|\partial_xg_i(t-\tau)\|_1&\leq C(t-\tau)^{-1/2}.          \notag
\end{align}
The following slice estimates imply the time-integrated
convolution bounds in \cite[Appendix~A]{Koike2023}.  When
\(t-\tau\geq1+\tau\),
\begin{align}
 &\int_{\mathbb R}|\partial_xg_i(x-y,t-\tau)|
 \bigl((y-\lambda_\alpha(1+\tau))^2+1+\tau\bigr)^{-3/2}\,dy\notag\\
 &\quad\leq C(t-\tau)^{-1}(1+\tau)^{-1}
 \exp\left[-\frac{(x-\lambda_i(t-\tau)-\lambda_\alpha(1+\tau))^2}
 {C(t-\tau)}\right]\notag\\
 &\qquad+C(t-\tau)^{-1/2}
 \bigl((x-\lambda_i(t-\tau)-\lambda_\alpha(1+\tau))^2+1+t\bigr)^{-3/2}.       \label{eq:app-early-convolution}
\end{align}
When \(1+\tau\geq t-\tau\geq1\), one has
\begin{align}
 &\int_{\mathbb R}g_i(x-y,t-\tau)
 \bigl((y-\lambda_\alpha(1+\tau))^2+1+\tau\bigr)^{-3/2}\,dy\notag\\
 &\leq C
 \bigl((x-\lambda_i(t-\tau)-\lambda_\alpha(1+\tau))^2+1+\tau\bigr)^{-3/2},     \label{eq:app-late-convolution-zero}\\
 &\int_{\mathbb R}|\partial_xg_i(x-y,t-\tau)|
 \bigl((y-\lambda_\alpha(1+\tau))^2+1+\tau\bigr)^{-3/2}\,dy\notag\\
 &\leq C(t-\tau)^{-1/2}
 \bigl((x-\lambda_i(t-\tau)-\lambda_\alpha(1+\tau))^2+1+\tau\bigr)^{-3/2}.     \label{eq:app-late-convolution-one}
\end{align}
Indeed, put
\[
 s=t-\tau,\qquad T=1+\tau,\qquad
 X=x-\lambda_i s-\lambda_\alpha T,\qquad
 u=y-\lambda_\alpha T.
\]
Split the \(u\)-line into
\[
 |u|\leq |X|/2,\qquad |X-u|\leq |X|/2,
\]
and their complement.  On the first region integrate the algebraic factor
and retain the Gaussian phase; on the second integrate the Gaussian and
evaluate the algebraic factor at \(X\); on the complement both lower bounds
hold.  Using
\[
 \int_{\mathbb R}(u^2+T)^{-3/2}\,du=CT^{-1}
\]
gives \eqref{eq:app-early-convolution} when \(s\geq T\).  Interchanging the
roles of the two scales when \(T\geq s\), and using respectively the
\(L^1\) norms of \(g_i(s)\) and \(\partial_xg_i(s)\), gives
\eqref{eq:app-late-convolution-zero} and
\eqref{eq:app-late-convolution-one}.
There are two harmless endpoint conventions which will be used below.  First,
\eqref{eq:app-early-convolution} is applied only on
\[
 I_{\mathrm{early}}(t):=\left[0,\frac{t-1}{2}\right],
\]
where \(t-\tau\geq1+\tau\).  If
\[
 I_{\mathrm{transition}}(t):=\left[\frac{t-1}{2},\frac t2\right],
\]
then \(|I_{\mathrm{transition}}(t)|=1/2\), \(1+\tau\geq t-\tau\), and,
for \(t\geq1\),
\[
 \frac{1+t}{4}\leq t-\tau\leq1+t,
 \qquad
 \frac{1+t}{2}\leq1+\tau\leq1+t.
\]
For \(t\) large enough that
\(t-\tau\geq1\), the contribution of \(I_{\mathrm{transition}}(t)\) is estimated by
\eqref{eq:app-late-convolution-one} and is bounded by the same right-hand
side as the adjacent \(I_{\mathrm{early}}(t)\)-integral.
Below, an integral with upper limit $t/2$ denotes the
$I_{\mathrm{early}}(t)$-estimate together with this half-unit correction.

Second, set
\[
 T_K:=\max\left\{4,\frac{64K^2}{|\lambda_1-\lambda_2|^2}\right\}.
\]
When \(1+t\leq T_K\), the terminal centers remain in a fixed bounded set.
The same three-piece spatial split used above, followed by the last-unit
measure estimate, gives \(CC_f\langle x\rangle^{-2}\), which is bounded by
\(CC_f\Psi_i(x,t)\) because the first summand of
\eqref{eq:endpoint-weight} is bounded below by
\(C_K^{-1}\langle x\rangle^{-2}\).  Hence the remaining regional proof may
assume \(1+t\geq T_K\).  In that range the two tubes of radius
\(2K\sqrt{1+t}\) are disjoint.  Boundary points are assigned in the order in
which the five cases are listed.

\subsection{The same characteristic family}

If \(\alpha=i\), then
\[
 x-\lambda_i(t-\tau)-\lambda_i(1+\tau)=x-\lambda_i(1+t),
\]
so there is no moving resonance.  On \(0\leq\tau\leq t/2\),
\eqref{eq:app-early-convolution} yields
\begin{align}
 &\int_0^{t/2}\int_{\mathbb R}|\partial_xg_i(x-y,t-\tau)|
 \bigl((y-\lambda_i(1+\tau))^2+1+\tau\bigr)^{-3/2}\,dy\,d\tau\notag\\
\leq& C(1+t)^{-1}\log(2+t)
 \exp\left[-\frac{(x-\lambda_i(1+t))^2}{C(1+t)}\right]+C(1+t)^{1/2}
 \bigl((x-\lambda_i(1+t))^2+1+t\bigr)^{-3/2}.              \label{eq:app-same-early}
\end{align}
On \(t/2\leq\tau\leq(t-1)_+\), keep the derivative on the kernel and use
\eqref{eq:app-late-convolution-one}; then
\begin{align}
 &\int_{t/2}^{(t-1)_+}\int_{\mathbb R}|\partial_xg_i(x-y,t-\tau)|
 \bigl((y-\lambda_i(1+\tau))^2+1+\tau\bigr)^{-3/2}\,dy\,d\tau\notag\\
 &\qquad\leq C(1+t)^{1/2}
 \bigl((x-\lambda_i(1+t))^2+1+t\bigr)^{-3/2}.              \label{eq:app-same-late}
\end{align}
Since
\[
 (1+t)^{1/2}\bigl((x-\lambda_i(1+t))^2+1+t\bigr)^{-3/2}
 \leq\bigl((x-\lambda_i(1+t))^2+1+t\bigr)^{-1},
\]
 \eqref{eq:app-same-early}--\eqref{eq:app-same-late} prove
\eqref{eq:same-cone-map} at every spatial point.

\subsection{Opposite characteristic speeds}\label{nnn}

Assume \(\alpha\neq i\).  Multiplying \eqref{eq:app-early-convolution} by the
factor \(\log(2+\tau)(1+\tau)^{-1/2}\) in
Lemma~\ref{lem:endpoint-convolution}, and using
\((1+t)/3\leq t-\tau\leq1+t\) on \(0\leq\tau\leq t/2\) in the present
range \(t\geq3\), gives
\begin{align}
 C\,C_f\log(2+t)\Bigg[&
 (1+t)^{-1}\int_0^{t/2}(1+\tau)^{-3/2}
 \exp\left(-\frac{\left[\substack{x-\lambda_i(t-\tau)-\lambda_\alpha(1+\tau)}\right]^2}{C(1+t)}\right)d\tau\notag\\
 &+(1+t)^{-1/2}\int_0^{t/2}(1+\tau)^{-1/2}
 \left(\left[\substack{x-\lambda_i(t-\tau)-\lambda_\alpha(1+\tau)}\right]^2
       +1+t\right)^{-3/2}
 d\tau\Bigg].                                                \label{eq:app-opposite-early}
\end{align}
On the whole late interval keep the derivative on the kernel and use
\eqref{eq:app-late-convolution-one}.  Put \(s=t-\tau\).  Since
\((1+t)/2\leq1+\tau\leq1+t\) there and
\[
 x-\lambda_i(t-\tau)-\lambda_\alpha(1+\tau)
 =x-\lambda_\alpha(1+t)+(\lambda_\alpha-\lambda_i)s,
\]
the late contribution is bounded by
\begin{equation}
 C\,C_f\log(2+t)(1+t)^{-1/2}
 \int_1^{t/2}s^{-1/2}
 \bigl([x-\lambda_\alpha(1+t)
       +(\lambda_\alpha-\lambda_i)s]^2+1+t\bigr)^{-3/2}\,ds .
 \label{eq:app-opposite-late}
\end{equation}
An integral with its lower endpoint larger than its upper endpoint is understood to
be zero.  We now estimate \eqref{eq:app-opposite-early}--
\eqref{eq:app-opposite-late} in the five spatial regions.

\paragraph{The receiving-cone tube.}
If
\(|x-\lambda_i(1+t)|\leq2K\sqrt{1+t}\), the first integral in
\eqref{eq:app-opposite-early} is bounded by a constant and the second by
\(C(1+t)^{-1}\).  In the late integral the phase has size comparable with
\(1+t\).  Hence
\begin{equation}
 \eqref{eq:app-opposite-early}+\eqref{eq:app-opposite-late}
 \leq\frac{C\,C_f\log(2+t)}{(x-\lambda_i(1+t))^2+1+t}.          \label{eq:app-receiving-tube}
\end{equation}

\paragraph{The source-cone tube.}
If
\(|x-\lambda_\alpha(1+t)|\leq2K\sqrt{1+t}\), the early phase has size
comparable with \(1+t\).  In the late integral use \(s^{-1/2}\leq1\)
and extend the affine phase to the whole real line.  Thus
\[
 \begin{aligned}
 &(1+t)^{-1/2}\int_1^{t/2}s^{-1/2}
 \bigl([x-\lambda_\alpha(1+t)
       +(\lambda_\alpha-\lambda_i)s]^2+1+t\bigr)^{-3/2}\,ds\\
 &\qquad\leq
 (1+t)^{-1/2}\int_{\mathbb R}
 \bigl([x-\lambda_\alpha(1+t)
       +(\lambda_\alpha-\lambda_i)s]^2+1+t\bigr)^{-3/2}\,ds\\
 &\qquad\leq C(1+t)^{-3/2}.
 \end{aligned}
\]
Therefore
\begin{equation}
 \eqref{eq:app-opposite-early}+\eqref{eq:app-opposite-late}
 \leq\frac{C\,C_f\log(2+t)}{|x-\lambda_\alpha(1+t)|^2+(1+t)^{3/2}}. \label{eq:app-source-tube}
\end{equation}

\paragraph{Strictly between the two tubes.}
Here both terminal distances are at least \(K\sqrt{1+t}\), and their sum is
comparable with \(1+t\).  Split the first integral in
\eqref{eq:app-opposite-early} at the possible zero of
\(x-\lambda_i(t-\tau)-\lambda_\alpha(1+\tau)\).  In an interval of length
\(C\sqrt{1+t}\) about that zero, change variables by dividing the phase by
\(\sqrt{1+t}\); outside it use the algebraic denominator. The calculation of the second integral in  \eqref{eq:app-opposite-early} is similar.  Thus, we have 
\begin{align}
 \eqref{eq:app-opposite-early}\le C C_f\log(2+t)\left\{(1+t)^{-1/2}|x-\lambda_i(1+t)|^{-3/2}
 +C(1+t)^{-3/2}|x-\lambda_i(1+t)|^{-1/2}\right\}.             \label{eq:app-interior-early}
\end{align}
Split the integral in \eqref{eq:app-opposite-late} at \(s=\sqrt{1+t}\).  On
\(1\leq s\leq\sqrt{1+t}\), the source-cone distance keeps the possible zero
outside the interval.  Since
\(\int_1^{\sqrt{1+t}}s^{-1/2}\,ds\leq C(1+t)^{1/4}\), this part is bounded by
\[
 C\bigl(|x-\lambda_\alpha(1+t)|^2+1+t\bigr)^{-3/2}.
\]
On \(\sqrt{1+t}\leq s\leq t/2\), split at the possible zero of the phase and
use \(s^{-1/2}\leq C|x-\lambda_\alpha(1+t)|^{-1/2}\) in its
\(C\sqrt{1+t}\)-neighborhood.  The complementary monotone pieces are
smaller, and this part is bounded by
\[
 C(1+t)^{-3/2}|x-\lambda_\alpha(1+t)|^{-1/2}.
\]
Because the two terminal distances are at least \(K\sqrt{1+t}\) and their sum is
comparable with \(1+t\), all four bounds are at most
\begin{equation}
 \eqref{eq:app-opposite-early}+\eqref{eq:app-opposite-late}\le\frac{C\,C_f\log(2+t)}
 {(1+[x-\lambda_i(1+t)]^2)^{1/2}
  (1+[x-\lambda_\alpha(1+t)]^2)^{1/4}}.                  \label{eq:app-interaction-ridge}
\end{equation}

\paragraph{The exterior beyond the receiving cone.}
On this half-line the moving phase has no zero and increases away from its first
endpoint.  Integration of \eqref{eq:app-opposite-early}--
\eqref{eq:app-opposite-late} gives
\begin{equation}
 C\,C_f\log(2+t)\left\{
 \frac1{(x-\lambda_i(1+t))^2+1+t}+(1+t)^{-1/2}
 \bigl((x-\lambda_i(1+t))^2+1+t\bigr)^{-3/4}
 \right\}.
 \label{eq:app-receiving-exterior}
\end{equation}

\paragraph{The exterior beyond the source cone.}
The phase is again monotone, now from the opposite endpoint.  Direct integration for \eqref{eq:app-opposite-early}--
\eqref{eq:app-opposite-late}
gives
\begin{equation}
 C\,C_f\log(2+t)\left\{
\frac1{|x-\lambda_\alpha(1+t)|^2+(1+t)^{3/2}}
 +(1+t)^{-1/2}
 \frac1{|x-\lambda_\alpha(1+t)|^{3/2}+1+t}
 \right\}.
 \label{eq:app-source-exterior}
\end{equation}
Here is the tail calculation suppressed in the preceding two displays.  On either
exterior half-line the derivative of the moving phase has the fixed nonzero size
\(|\lambda_i-\lambda_\alpha|=2c\).  After changing from \(\tau\) to the
absolute value of that phase, the endpoint integrations use
\begin{align}
 \int_r^\infty e^{-u^2/[C(1+t)]}\,du
 &\leq \frac{C(1+t)}{r+\sqrt{1+t}}
       e^{-r^2/[C'(1+t)]},\notag\\
 \int_r^\infty (u^2+1+t)^{-3/2}\,du
 &\leq \frac{C}{r^2+1+t}.                                  \label{eq:app-exterior-tail-integrals}
\end{align}
For the two algebraic Koike tails the same conclusion follows by dividing
\([r,\infty)\) into \([2^mr,2^{m+1}r]\), with the first interval replaced by
\([0,\sqrt{1+t}]\) when \(r<\sqrt{1+t}\).  The resulting geometric sums are
\begin{align}
 \sum_{m\geq0}\frac{2^mr}
 {((2^mr)^2+1+t)^{3/2}}
 &\leq \frac{C}{r^2+1+t},\notag\\
 \sum_{m\geq0}\frac{(2^mr)^{1/2}}
 {(2^mr)^{3/2}+1+t}
 \min\left\{1,\frac{\sqrt{1+t}}{2^mr}\right\}
 &\leq C(1+t)^{1/2}
       \frac1{r^{3/2}+1+t}.                                \label{eq:app-exterior-dyadic-tails}
\end{align}
In the receiving exterior take
\(r=|x-\lambda_i(1+t)|\); in the source exterior take
\(r=|x-\lambda_\alpha(1+t)|\).  Substitution of
\eqref{eq:app-exterior-tail-integrals}--
\eqref{eq:app-exterior-dyadic-tails}, together with the prefactors in
\eqref{eq:app-opposite-early}--\eqref{eq:app-opposite-late}, gives exactly
\eqref{eq:app-receiving-exterior} and \eqref{eq:app-source-exterior}.  Thus the
exterior estimates do not use the middle-region cutoff, which vanishes there.
The five estimates \eqref{eq:app-receiving-tube}--
\eqref{eq:app-source-tube} and \eqref{eq:app-interaction-ridge}--
\eqref{eq:app-source-exterior} are bounded by the corresponding components of
\(C\,C_f\log(2+t)\Psi_i\), and prove
\eqref{eq:separated-cone-map}.

For the enlarged source in Lemma~\ref{lem:endpoint-convolution}, the term with
\(\log^2(2+\tau)\) gives the same early and late integrals with that factor inserted.  In the
receiving tube,
\[
 \int_0^\infty\log^2(2+\tau)(1+\tau)^{-3/2}\,d\tau<\infty,
\]
so its contribution is still bounded by \(C\,C_f(1+t)^{-1}\).  In the source
tube, \(s^{-1/2}\leq1\) and the whole-line estimate above give
 \begin{align*}
 &C\,C_f\log^2(2+t)(1+t)^{-1/2}\times\int_{\mathbb R}
 \bigl([x-\lambda_\alpha(1+t)
       +(\lambda_\alpha-\lambda_i)s]^2+1+t\bigr)^{-3/2}\,ds\notag\\
 &\leq C\,C_f\log^2(2+t)(1+t)^{-3/2}.
 \end{align*}
This is bounded by the secondary source-cone term of
\(C\,C_f\log(2+t)\Psi_i\).  Between the cones, evaluate the logarithm at the
possible zero of the phase; every positive power of either terminal distance
 absorbs the extra logarithm.  On the two exterior half-lines the phase is monotone,
 and the endpoint integrations give respectively \eqref{eq:app-receiving-exterior}
 and the secondary version of \eqref{eq:app-source-exterior}.  This proves the
enlarged assertion in all five regions.

\subsection{Two output derivatives and the algebraic profile tail}\label{appendix 4}

For a source on the receiving cone satisfying
\begin{gather*}
 |f_i|\leq C_f\log(2+\tau)
 \bigl((y-\lambda_i(1+\tau))^2+1+\tau\bigr)^{-1},\\
 |f_{i,y}|\leq C_f\log(2+\tau)
 \bigl((y-\lambda_i(1+\tau))^2+1+\tau\bigr)^{-3/2},
\end{gather*}
keep two derivatives on the early heat kernel and transfer one to the source on the
late interval.   \eqref{eq:app-same-early}--
\eqref{eq:app-same-late} then give
\[
 \left|\partial_x^2\int_0^{(t-1)_+}g_i(t-\tau)*f_i(\tau)\,d\tau\right|
 \leq C\,C_f\log(2+t)\Psi_i(x,t).
\]
For an opposite-cone source, assume in addition that
\begin{equation}
 |f_{\alpha,yy}(y,\tau)|
 \leq C_f\log(2+\tau)(1+\tau)^{-1/2}
 \bigl((y-\lambda_\alpha(1+\tau))^2+1+\tau\bigr)^{-3/2}.
 \label{eq:app-two-derivative-material}
\end{equation}
Both output derivatives stay on the heat kernel for
\(0\leq\tau\leq t/2\).  On \(t/2\leq\tau\leq(t-1)_+\), put both derivatives
on the source when \(1\leq t-\tau\leq\sqrt{1+t}\), using
\eqref{eq:app-two-derivative-material}, and leave one derivative on the kernel
when \(\sqrt{1+t}\leq t-\tau\leq t/2\).
{In the strict middle region the critical contribution is}
\begin{align*}
 &C C_f\log(2+t)(1+t)^{-1}
 \langle x-\lambda_i(1+t)\rangle^{-1/2}\leq
 \frac{C C_f\log(2+t)}
 {\langle x-\lambda_i(1+t)\rangle
  \langle x-\lambda_\alpha(1+t)\rangle^{1/2}},
\end{align*}
{because the two terminal distances have sum comparable with
\(1+t\).  The other four regions are unchanged, and hence}
\begin{equation*}
 \left|\partial_x^2\int_0^{(t-1)_+}
 g_i(t-\tau)*f_\alpha(\tau)\,d\tau\right|
 \leq C\,C_f\log(2+t)\Psi_i(x,t).
\end{equation*}

For the linear higher-profile term use the exact identity
\[
 \phi_i\Xi_{i,yy}=\partial_y(\phi_i\Xi_{i,y})
       +\frac{\theta_i}{2d}\phi_i\Xi_{i,y}.
\]
The Koike estimate \cite{Koike2023}, \eqref{eq:resummed-pointwise-bound}, and
\eqref{eq:phi-explicit-identities} give the two pointwise envelopes actually
used here:
\begin{align}
 &|\phi_i\Xi_{i,y}|(y,\tau)
 \leq C\epsilon^2(1+\tau)^{-1/2}
 \left[
 \frac1{([y-\lambda_i(1+\tau)]^2+1+\tau)^{3/4}}+\frac1{|y-\lambda_{3-i}(1+\tau)|^{3/2}+1+\tau}\right],\notag\\
 &|\partial_y(\phi_i\Xi_{i,y})|(y,\tau)
 \leq C\epsilon^2(1+\tau)^{-1}
 \left[\frac1{([y-\lambda_i(1+\tau)]^2+1+\tau)^{3/4}}+\frac1{|y-\lambda_{3-i}(1+\tau)|^{3/2}+1+\tau}\right].            \label{eq:app-Xi-tail-source-envelopes}
\end{align}
The factor produced by \(\phi_{i,y}\) is smaller, since it also contains the
Gaussian \(\theta_i\).  For the divergence in the displayed identity, keep both
output derivatives on the early kernel and transfer one derivative to
\(\partial_y(\phi_i\Xi_{i,y})\) on the late interval.  The time factors are
\begin{align}
 \int_0^{t/2}(t-\tau)^{-1}(1+\tau)^{-1/2}\,d\tau
 &\leq C(1+t)^{-1/2},\notag\\
 \int_{t/2}^{(t-1)_+}(t-\tau)^{-1/2}(1+\tau)^{-1}\,d\tau
 &\leq C(1+t)^{-1/2}.                                    \notag
\end{align}
Applying the three-piece spatial split of
\eqref{eq:app-early-convolution} separately to the two braces in
\eqref{eq:app-Xi-tail-source-envelopes} preserves their two terminal algebraic
profiles.  The nondivergence term has the extra Gaussian factor \(\theta_i\) and
is smaller by the same calculation.  Thus, in each of the five regions, the
complete linear higher-profile action is
\begin{equation}
 C\epsilon^2\left\{
 (1+t)^{-1/2}\bigl((x-\lambda_i(1+t))^2+1+t\bigr)^{-3/4}+(1+t)^{-1/2}\frac1{|x-\lambda_{3-i}(1+t)|^{3/2}+1+t}
 \right\}.
 \label{eq:app-Xi-tail}
\end{equation}
This is the algebraic last line of \(\Psi_i\).  Treating
\(\Xi_{i,yy}\) as an undifferentiated source would miss the factor
\((1+t)^{-1/2}\).

\subsection{Weighted maps for the remaining sources}

We prove Lemma~\ref{lem:remaining-source-maps} from the same spatial splits,
rather than from an \(L^1\) output estimate.  The elementary spatial integrals
needed to check the time coefficients are
\begin{align}
 &\int_{\mathbb R}
 \frac{dy}{([y-\lambda_\alpha(1+\tau)]^2+1+\tau)^{3/2}}
 =C(1+\tau)^{-1},\notag\\
 &\int_{\mathbb R}
 \frac{dy}{([y-\lambda_\alpha(1+\tau)]^2+1+\tau)^{3/4}}
 =C(1+\tau)^{-1/4},\notag\\
 &\int_{\mathbb R}
 \frac{dy}{|y-\lambda_\alpha(1+\tau)|^{3/2}+1+\tau}
 =C(1+\tau)^{-1/3}
 \leq C(1+\tau)^{-1/4}.                                  \label{eq:app-remaining-profile-norms}
\end{align}
Put \(\sigma=t-\tau\geq1\) and \(T_\tau=1+\tau\).  After the displayed
output derivative, the Gaussian modes are direct instances of
\eqref{eq:app-early-convolution}--\eqref{eq:app-late-convolution-one} with one
additional heat derivative.  The regular remainder is treated by
\eqref{eq:Rr-pointwise}--\eqref{eq:Rr-norms}
with \(k=1\).  Retaining both characteristic centers in the same three-piece
source split gives
\begin{align}
 &\left|\partial_x\mathcal R^r_{i\beta}(\sigma)*
 \frac1{([\,\cdot-\lambda_\alpha T_\tau]^2+T_\tau)^{3/2}}\right|(x)\notag\\
 &\quad\leq C\sum_{\ell=1}^2\Bigg\{
 \frac{\sigma^{-1}T_\tau^{-1}}
 {[x-\lambda_\ell\sigma-\lambda_\alpha T_\tau]^2+\sigma}
 +\frac{\sigma^{-3/2}}
 {([x-\lambda_\ell\sigma-\lambda_\alpha T_\tau]^2+T_\tau)^{3/2}}
 \Bigg\},\notag\\
 &\left|\partial_x\mathcal R^r_{i\beta}(\sigma)*
 \frac{T_\tau^{-3/4}}
 {([\,\cdot-\lambda_\alpha T_\tau]^2+T_\tau)^{3/4}}\right|(x)\notag\\
 &\quad\leq C\sum_{\ell=1}^2\Bigg\{
 \frac{\sigma^{-1}T_\tau^{-1}}
 {[x-\lambda_\ell\sigma-\lambda_\alpha T_\tau]^2+\sigma}
 +\frac{\sigma^{-3/2}T_\tau^{-3/4}}
 {([x-\lambda_\ell\sigma-\lambda_\alpha T_\tau]^2+T_\tau)^{3/4}}
 \Bigg\},\notag\\
 &\left|\partial_x\mathcal R^r_{i\beta}(\sigma)*
 \frac{T_\tau^{-3/4}}
 {|\,\cdot-\lambda_\alpha T_\tau|^{3/2}+T_\tau}\right|(x)\notag\\*
 &\quad\leq C\sum_{\ell=1}^2\Bigg\{
 \frac{\sigma^{-1}T_\tau^{-13/12}}
 {[x-\lambda_\ell\sigma-\lambda_\alpha T_\tau]^2+\sigma}
 +\frac{\sigma^{-3/2}T_\tau^{-3/4}}
 {|x-\lambda_\ell\sigma-\lambda_\alpha T_\tau|^{3/2}+T_\tau}
 \Bigg\}.
 \label{eq:remaining-regular-convolutions}
\end{align}
The exponentially decaying term in
\eqref{eq:Rr-pointwise} obeys the same right-hand sides by its
weighted \(L^1\) bound.  The five-region argument is applied separately to
the two characteristic centers, so the finite sum only changes the constant.
The early--late coefficients used in all five regions obey
\begin{align}
 &(1+t)^{-3/2}\int_0^{t/2}(1+\tau)^{-1}\,d\tau
 +(1+t)^{-3/2}\int_1^{(1+t)/2}\sigma^{-1}\,d\sigma\leq C\log(2+t)(1+t)^{-3/2},\notag\\
 &(1+t)^{-2}\int_0^{t/2}(1+\tau)^{-1}\,d\tau
 +(1+t)^{-3/2}\int_1^{(1+t)/2}\sigma^{-3/2}\,d\sigma\leq C[(1+t)^{-3/2}+\log(2+t)(1+t)^{-2}].
 \label{eq:remaining-early-late-coefficients}
\end{align}
Consequently \eqref{eq:app-remaining-profile-norms} and the first set of hypotheses in
Lemma~\ref{lem:remaining-source-maps} imply
\begin{align*}
 \|h(\tau)\|_1
 &\leq C\,C_f\log(2+\tau)(1+\tau)^{-1},\\
 \|h_y(\tau)\|_1+\|h_0(\tau)\|_1
 &\leq C\,C_f\log(2+\tau)(1+\tau)^{-3/2},\\
 \|h_y(\tau)\|_\infty+\|h_0(\tau)\|_\infty
 &\leq C\,C_f\log(2+\tau)(1+\tau)^{-2}.
\end{align*}
Keeping both derivatives on the early heat kernel and transferring one
derivative from the kernel to \(h\) on the late interval gives
\begin{align}
 &C\,C_f\int_0^{t/2}\left\{
 (t-\tau)^{-3/2}\log(2+\tau)(1+\tau)^{-1}
 +(t-\tau)^{-1}\log(2+\tau)(1+\tau)^{-3/2}\right\}d\tau
 \notag\\
 &\quad+C\,C_f\int_{t/2}^{(t-1)_+}
 (t-\tau)^{-1/2}\log(2+\tau)(1+\tau)^{-2}\,d\tau\leq C\,C_f\left\{(1+t)^{-1}
 +\log^2(2+t)(1+t)^{-3/2}\right\}.                       \notag
\end{align}
This is only the uniform check.  Pointwise, insert each of the explicitly
displayed Gaussian and algebraic summands into
\eqref{eq:app-early-convolution}--\eqref{eq:app-late-convolution-one}.
The three-piece source-line split retains the phase in every summand, and the
five terminal cases give the full display below.

For the structured regular-remainder assertion, the receiving-family term
\(e^{(0)}\) in \eqref{eq:remaining-structured-hypotheses} gives
\begin{align}
 &C\,C_f(1+t)^{-2}
 \int_0^{t/2}(1+\tau)^{-1/2}\,d\tau
 +C\,C_f(1+t)^{-3/2}
 \int_1^{(1+t)/2}\sigma^{-1}\,d\sigma\notag\\
 &\leq C\,C_f\log(2+t)(1+t)^{-3/2}.
 \notag
\end{align}
The early actions of \(e^{(1)}\), \(\partial_yh\), and \(h_0\) are bounded by
\begin{align}
 &C\,C_f(1+t)^{-2}
 \int_0^{t/2}\log(2+\tau)(1+\tau)^{-1}\,d\tau+C\,C_f(1+t)^{-5/2}
 \int_0^{t/2}\log(2+\tau)(1+\tau)^{-1/2}\,d\tau\notag\\
 &\leq C\,C_f\log^2(2+t)(1+t)^{-2}
 \leq C\,C_f\log(2+t)(1+t)^{-3/2},
 \notag
\end{align}
whereas their late action is
\begin{equation*}
 C\,C_f\log(2+t)(1+t)^{-3/2}
 \int_1^{(1+t)/2}
 [\sigma^{-3/2}+\sigma^{-3/2}+\sigma^{-3/2}]\,d\sigma
 \leq C\,C_f\log(2+t)(1+t)^{-3/2}.
\end{equation*}
For the term containing \(\partial_y\mathcal R^r_{i\beta}\), the output
derivative is transferred to \(h\); for the other two terms it is kept on the
regular kernel.

For the coefficient source in \eqref{eq:eta-quarter-hypotheses}, a global
norm estimate alone would not imply the required cone-resolved conclusion.
Put
\[
 T=1+t,\qquad T_\tau=1+\tau,\qquad
 r_\ell=x-\lambda_\ell T\quad(\ell=1,2),\qquad
 \Phi_\alpha=x-\lambda_i(t-\tau)-\lambda_\alpha T_\tau .
\]
For a fixed \(\alpha\neq i\), the early contribution on \(I_{\mathrm{early}}(t)\), followed by
the two late subintervals in the variable \(s=t-\tau\), is bounded by
\begin{align}
 C C_f\Bigg\{&
 T^{-1}\int_{I_{\mathrm{early}}(t)}\log(2+\tau)T_\tau^{-5/4}
 e^{-\Phi_\alpha^2/(CT)}\,d\tau\notag\\
 &+T^{-1/2}\int_{I_{\mathrm{early}}(t)}\log(2+\tau)T_\tau^{{-1/2}}
 (\Phi_\alpha^2+T)^{{-5/4}}\,d\tau\notag\\
 &+\log(2+t)T^{{-1}}\int_1^{\sqrt T}
 \bigl([r_\alpha+(\lambda_\alpha-\lambda_i)s]^2+T\bigr)^{{-5/4}}\,ds\notag\\
 &+\log(2+t)T^{-3/2}\int_1^{\sqrt T}
 \bigl([r_\alpha+(\lambda_\alpha-\lambda_i)s]^2+T\bigr)^{-3/4}\,ds\notag\\
 &+\log(2+t)T^{{-1/2}}\int_{\sqrt T}^{t/2}s^{-1/2}
 \bigl([r_\alpha+(\lambda_\alpha-\lambda_i)s]^2+T\bigr)^{{-5/4}}\,ds
 \Bigg\}.
 \notag
\end{align}
The half-unit interval \(I_{\mathrm{transition}}(t)\) is the correction already described after
\eqref{eq:app-late-convolution-one}.  Splitting at the possible zero of the
affine phase, and using
\[
 \int_1^\infty u^{-5/4}\,du<\infty,\qquad
 {\int_{\mathbb R}(v^2+T)^{-5/4}\,dv=CT^{-3/4}},\qquad
 \int_{\mathbb R}(v^2+T)^{-3/4}\,dv=CT^{-1/4},
\]
gives the following five explicit regional bounds:
\begin{equation}
 C C_f\log(2+t)
 \begin{cases}
 \displaystyle (r_i^2+T)^{-1},
   & |r_i|\leq2K\sqrt T,\\[1mm]
 \displaystyle \dfrac{1+\log(2+t)}{r_\alpha^2+T^{3/2}},
   & |r_\alpha|\leq2K\sqrt T,\\[3mm]
 \displaystyle
 \dfrac{\chi_K(x,t)}
 {\langle r_i\rangle\langle r_\alpha\rangle^{1/2}},
   & \text{strictly between the two tubes},\\[4mm]
 \displaystyle
 \dfrac1{r_i^2+T}
 +\dfrac{T^{-1/2}}{(r_i^2+T)^{3/4}},
   & \text{in the receiving exterior},\\[4mm]
 \displaystyle
 \dfrac1{r_\alpha^2+T^{3/2}}
 +\dfrac{T^{-1/2}}{|r_\alpha|^{3/2}+T},
   & \text{in the source exterior}.
 \end{cases}
 \label{eq:app-eta-quarter-five-regions}
\end{equation}
For the middle region one repeats the phase-zero split leading to
\eqref{eq:app-interior-early}; for the exterior regions the monotone changes
of variables are exactly
\eqref{eq:app-exterior-tail-integrals}--
\eqref{eq:app-exterior-dyadic-tails}.  The same-family summand
\(\alpha=i\) is bounded by the first line.  Every line of
\eqref{eq:app-eta-quarter-five-regions} is {bounded by the corresponding component} of
\(\Psi_i\), so \eqref{eq:eta-quarter-map} follows.

{For the nonlinear source \(n\), use the second hypothesis in the
strict middle region and the first hypothesis in the other four regions.  Direct
spatial integration gives}
\[
 \|n(\tau)\|_1\leq
 C\,C_f\log^2(2+\tau)(1+\tau)^{-3/2},
 \qquad
 \|n(\tau)\|_\infty\leq
 C\,C_f\log^2(2+\tau)(1+\tau)^{-2}.
\]
The time check is
\begin{align}
 &C\,C_f(1+t)^{-1}
 \int_0^{t/2}\log^2(2+\tau)(1+\tau)^{-3/2}\,d\tau\notag\\
 &\quad+C\,C_f\log^2(2+t)(1+t)^{-2}
 \int_{t/2}^{(t-1)_+}(t-\tau)^{-1/2}\,d\tau
 \leq C\,C_f\log(2+t)(1+t)^{-1}.                         \label{eq:app-remaining-nonlinear-time}
\end{align}
{For the spatial conclusion, use the strict-middle bound there and
convolve each remaining component separately with the phase in
\eqref{eq:app-early-convolution}.}

Finally, the last hypothesis of Lemma~\ref{lem:remaining-source-maps} and
\eqref{eq:app-remaining-profile-norms} give
\(\|f_{\rm rem}(\tau)\|_1\leq C C_f(1+\tau)^{-1}\) and
\(\|f_{\rm rem}(\tau)\|_\infty\leq C C_f(1+\tau)^{-3/2}\).
The extra heat derivative in the nonleading Gaussian entries and the
half-power gain in \eqref{eq:Rr-pointwise}--\eqref{eq:Rr-norms}, combined
with \eqref{eq:remaining-regular-convolutions} and
\eqref{eq:remaining-early-late-coefficients}, give the same five regional
outputs.  We record them once below.

Let \(I_i(x,t)\) denote, only in the next display, in turn the complete action
on the left of \eqref{eq:remaining-divergence-map}, the action on the left of
\eqref{eq:remaining-nonlinear-map}, or the action on the left of
\eqref{eq:remaining-structured-map}.
Applying the phase split separately in the receiving tube, source tube, strict
middle, receiving exterior, and source exterior gives
\begin{equation}
\begin{aligned}
 |I_i(x,t)|&\leq C\,C_f\log(2+t)\Bigg[
 \frac1{[x-\lambda_i(1+t)]^2+1+t}+\frac{1+\log(2+t)}
 {[x-\lambda_{3-i}(1+t)]^2+(1+t)^{3/2}}\\
 &\qquad\qquad\qquad\qquad+\frac{\chi_K(x,t)}
 {\bigl(1+[x-\lambda_i(1+t)]^2\bigr)^{5/8}
  \bigl(1+[x-\lambda_{3-i}(1+t)]^2\bigr)^{1/4}}\\
 &\qquad\qquad\qquad\qquad+\frac{(1+t)^{-1/2}}
 {([x-\lambda_i(1+t)]^2+1+t)^{3/4}}+\frac{(1+t)^{-1/2}}
 {|x-\lambda_{3-i}(1+t)|^{3/2}+1+t}\Bigg].
\end{aligned}
\label{eq:app-remaining-five-region-output}
\end{equation}
In the two exteriors use
\eqref{eq:app-exterior-tail-integrals}--\eqref{eq:app-exterior-dyadic-tails};
in the middle use \eqref{eq:app-interaction-ridge}; {for the nonlinear
action use its strict-middle bound and the componentwise split in the other four
regions.  The \(5/8\) middle term is bounded by the corresponding \(1/2\) term
in \(\Psi_i\).}  This proves the first three assertions of
Lemma~\ref{lem:remaining-source-maps}.

For the fourth assertion, let \(I_i(x,t)\) denote one of the nonleading actions.
The polynomial bound in \eqref{eq:Rr-pointwise}, or the corresponding Gaussian
bound in \eqref{eq:explicit-gaussian-matrix}, gives in the receiving tube,
source tube, strict middle, receiving exterior, and source exterior, respectively,
\begin{align}\label{eq:app-nonleading-five-region-output-1}
 |I_i(x,t)|&\leq
 \frac{C C_f\log(2+t)}{[x-\lambda_i(1+t)]^2+1+t},\notag\\
 |I_i(x,t)|&\leq
 \frac{C C_f\log(2+t)}{[x-\lambda_{3-i}(1+t)]^2+(1+t)^{3/2}},\notag\\
 |I_i(x,t)|&\leq C C_f\log(2+t)
 \frac{\chi_K(x,t)}
 {(1+[x-\lambda_i(1+t)]^2)^{1/2}
  (1+[x-\lambda_{3-i}(1+t)]^2)^{1/4}}\notag\\
 &\quad+C C_f\log(2+t)(1+t)^{-1/2}
 \frac1{([x-\lambda_i(1+t)]^2+1+t)^{3/4}}\notag\\
 &\quad+C C_f\log(2+t)(1+t)^{-1/2}
 \frac1{|x-\lambda_{3-i}(1+t)|^{3/2}+1+t},\\
 |I_i(x,t)|&\leq C C_f\log(2+t)
 \frac1{[x-\lambda_i(1+t)]^2+1+t}+C C_f\log(2+t)
 \frac{(1+t)^{-1/2}}
 {([x-\lambda_i(1+t)]^2+1+t)^{3/4}},\notag\\
 |I_i(x,t)|&\leq C C_f\log(2+t)
 \frac1{[x-\lambda_{3-i}(1+t)]^2+(1+t)^{3/2}}+C C_f\log(2+t)
 \frac{(1+t)^{-1/2}}
 {|x-\lambda_{3-i}(1+t)|^{3/2}+1+t}.\notag
\end{align}
The middle-region split and the two monotone exterior integrations are the
same as in \eqref{eq:app-interaction-ridge} and
\eqref{eq:app-exterior-tail-integrals}--\eqref{eq:app-exterior-dyadic-tails}.
Together with \eqref{eq:remaining-early-late-coefficients}, these five bounds are
{bounded by the corresponding components of \(\Psi_i\)}
and prove \eqref{eq:remaining-nonleading-map}.

\subsection{Gaussian extraction and the material boundary terms}

For \(\alpha=i\), the semigroup identity is exact:
\begin{equation}
 g_i(t-\tau)*\Gamma_i(\cdot,1+\tau)=\Gamma_i(\cdot,1+t).       \label{eq:app-semigroup}
\end{equation}
This proves the same-family part of Lemma~\ref{lem:gaussian-extraction}.

Let \(\alpha\neq i\).  The convolution of the two Gaussians has variance
\(1+t\) and center
\[
 \lambda_i(t-\tau)+\lambda_\alpha(1+\tau)
 =\lambda_i(1+t)+(\lambda_\alpha-\lambda_i)(1+\tau).
\]
Put \(H=1+t\), \(\zeta=(x-\lambda_iH)/\sqrt H\), and
\[
 F_k(z)=\partial_z^k\frac{e^{-z^2/(4d)}}{\sqrt{4\pi d}}.
\]
After subtracting the cutoff multiple of \(\Gamma_i(x,H)\), the \(k\)-th
derivative is a fixed constant times
\begin{equation}
 H^{-(k+1)/2}
 \int_{H^{-1/2}}^{[1+(t-1)_+]H^{-1/2}}
 \frac{b_\alpha(u\sqrt H-1)}{u}
 \bigl[F_k(\zeta-(\lambda_\alpha-\lambda_i)u)
       -\chi_{\mathrm{time}}(u)F_k(\zeta)\bigr]\,du.
 \label{eq:app-cross-gaussian}
\end{equation}
For the global estimate, no derivative of \(b_\alpha\) is needed.  On \(0<u\leq1\),
\(\chi_{\mathrm{time}}(u)=1\), and the mean-value theorem gives
\[
 |F_k(\zeta-(\lambda_\alpha-\lambda_i)u)-F_k(\zeta)|
 \leq C u\sup_{0\leq v\leq u}
 |F_{k+1}(\zeta-(\lambda_\alpha-\lambda_i)v)|.
\]
This cancels the factor \(u^{-1}\).  On \(u\geq1\),
\[
 \sup_{\zeta\in\mathbb R}\int_1^\infty
 \frac{|F_k(\zeta-(\lambda_\alpha-\lambda_i)u)|}{u}\,du\leq C,
\]
and the cutoff term is supported in \(1\leq u\leq2\).  Hence, for
\(0\leq k\leq2\), \eqref{eq:app-cross-gaussian} is bounded in
\(L^\infty_x\) by \(C C_f(1+t)^{-(k+1)/2}\), using only
\(\|b_\alpha\|_\infty\leq C_f\).

For the cone-resolved differentiated estimate, use in addition
\((1+\tau)|b_\alpha'(\tau)|\leq C_f\).  Since
\(1+\tau=u\sqrt{1+t}\),
\[
 \left|\frac d{du}\frac{b_\alpha(u\sqrt{1+t}-1)}u\right|
 \leq\frac{\sqrt{1+t}\,|b_\alpha'(u\sqrt{1+t}-1)|}{u}
      +\frac{|b_\alpha(u\sqrt{1+t}-1)|}{u^2}
 \leq\frac{C C_f}{u^2}.
\]
Writing the first Gaussian derivative as a \(u\)-derivative and integrating
by parts, the two endpoint terms give the receiving- and source-cone
contributions, while the \(u^{-2}\) integral gives the intermediate and
exterior contributions.  This is the Gaussian part of
\eqref{eq:material-extraction-pointwise}.

For the zero-mass primitive, the exact identity
\eqref{eq:integrated-material-identity} produces an upper boundary, a lower
boundary, and a material integral.  At \(\tau=(t-1)_+\), move the derivative to
\(Q_{\alpha,y}\); because the two cone centers change by only a bounded distance
on one time unit,
\begin{equation}
 |\partial_xg_i(t-(t-1)_+)*Q_\alpha((t-1)_+)|(x)
 \leq\frac{C\,C_f}{|x-\lambda_\alpha(1+t)|^2+(1+t)^{3/2}}.       \label{eq:app-upper-material-boundary}
\end{equation}

At the lower boundary, \(\|Q_\alpha(0)\|_1\leq C_f\) gives
\begin{equation}
 \begin{aligned}
 |\partial_xg_i(t)*Q_\alpha(0)|(x)
 \leq\frac{C\,C_f}{[x-\lambda_i(1+t)]^2+1+t}.
 \end{aligned}
 \label{eq:app-lower-material-boundary}
\end{equation}
For the additional term in
\eqref{eq:material-primitive-bounds}, the same early--late split gives, when
\(\alpha\neq i\),
\begin{align}
 &\int_0^{(t-1)_+}(1+\tau)^{-3}\Bigg\{
 \int_{|y-\lambda_\alpha(1+\tau)|\leq c(1+\tau)/2}
 |\partial_xg_i(x-y,t-\tau)|\,dy\notag\\[-1mm]
 &\hspace{32mm}+\int_{\mathbb R}|\partial_xg_i(x-y,t-\tau)|
 e^{-\frac{(y-\lambda_\alpha(1+\tau))^2}{C(1+\tau)}}\,dy\Bigg\}d\tau
 \leq
 \frac{C}{[x-\lambda_i(1+t)]^2+1+t}.
 \label{eq:app-cutoff-material-map}
\end{align}
In the first spatial integral,
\(3c/2\leq|\lambda_\alpha-\lambda_i+
(y-\lambda_\alpha(1+\tau))/(1+\tau)|\leq5c/2\).  Splitting at
\(1+\tau=(1+t)/2\), the phase-zero calculation in
\eqref{eq:app-opposite-early}--\eqref{eq:app-opposite-late} gives its bound;
the other integral follows from the same two-Gaussian convolution.
The five-region bounds \eqref{eq:app-receiving-tube}--
\eqref{eq:app-source-exterior}, applied to
\eqref{eq:material-primitive-bounds}, together with
\eqref{eq:app-cutoff-material-map}, bound the material integral in
\eqref{eq:integrated-material-identity}.  Together with
\eqref{eq:app-upper-material-boundary} and
\eqref{eq:app-lower-material-boundary}, this proves the pointwise assertion.

The material estimate for the exceptional divergence in
\eqref{eq:exceptional-material} is obtained in exactly the same identity, but its
time factors can be checked directly.  On \(0\leq\tau\leq t/2\), the early heat
bound gives, for \(0\leq k\leq2\),
\begin{align}
 &(1+t)^{-(k+1)/2}\int_0^{t/2}\epsilon\left\{
 |a_j'(\tau)|(1+\tau)^{-1/2}
 +|a_j(\tau)|(1+\tau)^{-3/2}\right\}d\tau\leq C\epsilon X_N(1+t)^{-(k+1)/2}.              \label{eq:app-exceptional-early-material}
\end{align}
On \(t/2\leq\tau\leq(t-1)_+\), transfer one derivative to the material
derivative when \(k=2\).  Using \eqref{eq:no-prelog-input},
\(|a_j|\leq X_N\log(2+\tau)\) and
\(|a_j'|\leq X_N(1+\tau)^{-1}\), the only borderline integral is
\begin{equation}
 C\epsilon X_N(1+t)^{-2}\log(2+t)\log(2+t)
 \leq C\epsilon X_N(1+t)^{-3/2}.                         \label{eq:app-exceptional-late-material}
\end{equation}
The upper and lower material boundaries are smaller than the right-hand sides of
\eqref{eq:app-exceptional-early-material}.  Equations
\eqref{eq:app-exceptional-early-material}--
\eqref{eq:app-exceptional-late-material} prove
\eqref{eq:exceptional-divergence-map}, including its two-derivative endpoint.

We now prove the spatially resolved companion
\eqref{eq:exceptional-divergence-pointwise}.  Put
\[
 T=1+\tau,\qquad r_j=y-\lambda_jT,\qquad A_j=r_j^2+T,
 \qquad q_j=1-e^{-M_j/(2d)},
\]
and abbreviate
\[
 P_j=\frac{a_j\phi_i\theta_j^2}{4dq_j},
 \qquad
 \mathcal L_j=\partial_\tau+\lambda_j\partial_y-d\partial_y^2,
 \qquad j=3-i.
\]
As everywhere above, quotients containing \(q_j\) are understood by continuous
extension at \(M_j=0\).  The Burgers and Cole--Hopf identities give
\[
 \mathcal L_j\theta_j=-\theta_j\theta_{j,y},\qquad
 \mathcal L_j\phi_i=(\lambda_j-\lambda_i)\phi_{i,y},
 \qquad \phi_{i,y}=-\frac{\theta_i}{2d}\phi_i.
\]
Consequently the material derivative has the exact expansion
\begin{align*}
 \mathcal L_jP_j=\frac1{4dq_j}\Big\{&
 a_j'\phi_i\theta_j^2
 +a_j\big[-2\phi_i\theta_j^2\theta_{j,y}
           -2d\phi_i\theta_{j,y}^2+(\lambda_j-\lambda_i)\phi_{i,y}\theta_j^2
           -4d\phi_{i,y}\theta_j\theta_{j,y}\big]\Big\}.
\end{align*}
For \(0\leq m\leq3\), the identity
\(\theta_j/q_j=2d\Gamma_j/\phi_j\),
\eqref{eq:burgers-pointwise-bound}, and
\eqref{eq:phi-explicit-identities} give
\[
 |\partial_y^m\theta_j|
 \leq C\epsilon T^{-(m+1)/2}e^{-r_j^2/(CT)},
 \qquad
 \left|\partial_y^m\left(\frac{\theta_j}{q_j}\right)\right|
 \leq CT^{-(m+1)/2}e^{-r_j^2/(CT)}.
\]
Also, for \(1\leq m\leq2\),
\[
 |\partial_y^m\phi_i|
 \leq C\epsilon T^{-m/2}
 e^{-[y-\lambda_iT]^2/(CT)}.
\]
Combining these bounds with \eqref{eq:no-prelog-input} gives
\begin{align}
 |P_{j,y}(y,\tau)|
 &\leq C\epsilon X_N\log(2+\tau)T^{-3/2}
       e^{-r_j^2/(CT)},\notag\\
 |\mathcal L_jP_j(y,\tau)|
 +T^{1/2}|\partial_y\mathcal L_jP_j(y,\tau)|
 &\leq C\epsilon X_N\log(2+\tau)T^{-2}
       e^{-r_j^2/(CT)}.                                  \label{eq:app-exceptional-material-spatial}
\end{align}
Terms containing a derivative of \(\phi_i\) are smaller: they contain the
opposite Gaussian \(\theta_i\), and the two characteristic centers are
separated by \(2cT\).  In particular,
\begin{equation}
 |\mathcal L_jP_j|+|\partial_y\mathcal L_jP_j|
 \leq C\epsilon X_N\log(2+\tau)T^{-1/2}A_j^{-3/2}.
 \label{eq:app-exceptional-material-separated-envelope}
\end{equation}
Thus \eqref{eq:separated-cone-map}, with
\(f_j=\mathcal L_jP_j\), bounds the material integral on the right of
\eqref{eq:integrated-material-identity} by
\(C\epsilon X_N\log(2+t)\Psi_i(x,t)\).

It remains to check the two material boundary terms.  The lower one vanishes
because \(a_j(0)=0\) by \eqref{eq:ai-definition}.  If \(t\geq1\), integration
by parts in space at the upper boundary and
\eqref{eq:app-exceptional-material-spatial} give
\begin{align*}
 |\partial_xg_i(1)*P_j(t-1)|(x)
 &=|g_i(1)*P_{j,y}(t-1)|(x)\\
 &\leq C\epsilon X_N\log(2+t)(1+t)^{-3/2}
 e^{-[x-\lambda_j(1+t)]^2/[C(1+t)]}\\
 &\leq C\epsilon X_N\log(2+t)\Psi_i(x,t).
\end{align*}
For \(0\leq t<1\) the long-time integral is empty.  Since
\(|\lambda_j-\lambda_i|=2c\), the integrated material identity now proves
\eqref{eq:exceptional-divergence-pointwise}.

\subsection{The last time unit}

Let \((t-1)_+\leq\tau\leq t\).  The cone centers move by a bounded amount and
\(1+\tau\) is comparable with \(1+t\).  The elementary translation inequalities
needed for the exact finite measures are
\begin{align}
 &\bigl((x-z-\lambda_\alpha(1+\tau))^2+1+\tau\bigr)^{-3/2}\leq C(1+|z|)^3
 \bigl((x-\lambda_\alpha(1+t))^2+1+t\bigr)^{-3/2},           \label{eq:app-last-gaussian-weight}\\
 &\bigl((x-z-\lambda_\alpha(1+\tau))^2+1+\tau\bigr)^{-3/4}
 +\bigl(|x-z-\lambda_\alpha(1+\tau)|^{3/2}+1+\tau\bigr)^{-1}\notag\\
 &\leq C(1+|z|)^2\left\{
 \bigl((x-\lambda_\alpha(1+t))^2+1+t\bigr)^{-3/4}
 +\bigl(|x-\lambda_\alpha(1+t)|^{3/2}+1+t\bigr)^{-1}
 \right\}.
 \label{eq:app-last-algebraic-weight}
\end{align}
They follow separately from the triangle inequality after the distances are divided
by \(\sqrt{1+t}\) and \((1+t)^{2/3}\), respectively.

By \eqref{eq:short-time-green-measure}--
\eqref{eq:short-time-green-derivative}, the exact kernel measures have a
bounded fourth weighted moment.  Integrating
\eqref{eq:app-last-gaussian-weight}--\eqref{eq:app-last-algebraic-weight} against
those measures therefore preserves every component of \(\Psi_i\).  For a
general column, commute the output derivative to the complete source and use
\eqref{eq:last-unit-general}.  For a common vector, use
\eqref{eq:last-unit-common} and
\[
 \int_0^1s^{-1/2}\,ds=2.
\]
Thus the whole last-unit integral is bounded by
\(C\,C_f\log(2+t)\Psi_i(x,t)\) whenever the source derivative has the
corresponding two-cone and algebraic envelopes.  This proves the finite-measure
last-unit assertions used in \eqref{eq:complete-duhamel}; no pointwise density is
assigned to the singular signed measure.

\section*{Data availability}

Data sharing is not applicable to this article as no datasets were generated
or analysed during the current study.

\section*{Declaration on the use of generative AI}

Generative AI tools were used during preparation of this manuscript to assist
with language editing, \LaTeX{} formatting, and the
organization of the presentation.  They were not treated as an independent
mathematical source.  The authors take full responsibility for verifying all
mathematical statements, computations, proofs, and references, and for the
final content of the manuscript.

\providecommand{\bysame}{\leavevmode\hbox to3em{\hrulefill}\thinspace}
\providecommand{\MR}{\relax\ifhmode\unskip\space\fi MR }
\providecommand{\MRhref}[2]{%
  \href{http://www.ams.org/mathscinet-getitem?mr=#1}{#2}
}
\enlargethispage{2\baselineskip}

\end{document}